\documentclass[a4paper,10pt,oneside]{article}%oneside per scrivere solo a dx
\usepackage[utf8]{inputenc}
\usepackage[english]{babel}
\usepackage{amsthm}
\usepackage{amssymb}
\usepackage{amsmath}
\usepackage{hyperref}
\usepackage{graphicx}
\usepackage{url}
\usepackage{multicol}
\usepackage[all]{xy}
\DeclareMathAlphabet{\mathpzc}{OT1}{pzc}{m}{it}
\newtheoremstyle{note}{11pt}{11pt}{}{}{\bfseries}{.}{.5em}{}
\newtheorem{theo}[equation]{Theorem}
\newtheorem{coro}[equation]{Corollary}
\newtheorem{prop}[equation]{Proposition}

\newtheorem{conj}[equation]{Conjecture}

\numberwithin{equation}{section}
\newtheorem{lemma}[equation]{Lemma}

\newcommand{\Qb}{\overline{\mathbb{Q}}}
\newcommand{\Q}{\mathbb{Q}}
\newcommand{\F}{\mathbb{F}}

\newcommand{\Aa}{\mathbb{A}}
\newcommand{\fa}{{\mathbb{A}}_F}
\newcommand{\faf}{F_f}
\newcommand{\faft}{F_{f}^{\times}}
\newcommand{\fai}{F_{\infty}}
\newcommand{\fait}{F_{\infty}^{\times}}
\newcommand{\Z}{\mathbb{Z}}
\newcommand{\R}{\mathbb{R}}
\newcommand{\C}{\mathbb{C}}

\newcommand{\calE}{\mathcal{E}}

\newcommand{\N}{\mathcal{N}}

\newcommand{\I}{\mathcal{I}}

\newcommand{\Ll}{\mathcal{L}}

\newcommand{\Cc}{\mathcal{C}}

\newcommand{\oo}{\mathcal{O}}
\newcommand{\h}{\mathcal{H}}

\newcommand{\p}{\mathpzc{p}}

\newcommand{\cycl}{\mathcal{N}}

\newcommand{\G}{\Gamma}
\newcommand{\LL}{\Lambda}
\newcommand{\eps}{\varepsilon}

\newcommand{\lgr}{\left\{}
\newcommand{\rgr}{\right\}}
\newcommand{\rra}{\right\rangle}
\newcommand{\lla}{\left\langle}

\newcommand{\gl}{\mbox{GL}_2}
\newcommand{\Sl}{\mbox{SL}_2}
\newcommand{\clpinf}{\mathrm{Cl}_{\Nfrak}( p^{\infty})}
\newcommand{\clpinfc}{\mathrm{Cl}_{\cfrak}( p^{\infty})}
\newcommand{\clpinfm}{\mathrm{Cl}_{\Mfrak}( p^{\infty})}

\newcommand{\bolda}{\mathbf{a}}

\newcommand{\boldG}{\mathbf{G}}
\newcommand{\boldf}{\mathbf{f}}
\newcommand{\boldM}{\mathbf{M}}
\newcommand{\boldm}{\mathbf{m}}
\newcommand{\boldH}{\mathbf{H}}
\newcommand{\boldh}{\mathbf{h}}
\newcommand{\boldS}{\mathbf{S}}
\newcommand{\boldW}{\mathbf{W}}
\newcommand{\boldI}{\mathbf{I}}
\newcommand{\boldK}{\mathbf{K}}
\newcommand{\boldB}{\mathbf{B}}
\newcommand{\boldg}{\mathbf{g}}
\newcommand{\boldF}{\mathbf{F}}
\newcommand{\bexp}{\mathbf{e}}
\newcommand{\boldN}{\mathbf{N}}
\newcommand{\beps}{\boldsymbol\varepsilon}

\newcommand{\rif}{\mathfrak{r}}

\newcommand{\afrak}{\mathfrak{a}}
\newcommand{\lfrak}{\mathfrak{l}}
\newcommand{\cfrak}{\mathfrak{c}}

\newcommand{\pfrak}{\mathfrak{p}}
\newcommand{\mfrak}{\mathfrak{m}}
\newcommand{\Mfrak}{\mathfrak{M}}
\newcommand{\qfrak}{\mathfrak{q}}
\newcommand{\efrak}{\mathfrak{e}}
\newcommand{\Cfrak}{\mathfrak{C}}
\newcommand{\Lfrak}{\mathfrak{L}}

\newcommand{\ffrak}{\mathfrak{f}}
\newcommand{\nfrak}{\mathfrak{n}}
\newcommand{\dfrak}{\mathfrak{d}}
\newcommand{\Nfrak}{\mathfrak{N}}

\begin{document}
\title{Derivative at $s=1$ of the $p$-adic $L$-function of the symmetric square of a Hilbert modular form}
\author{Giovanni Rosso 
\footnote{PhD Fellowship of Fund for Scientific Research - Flanders, partially supported by a JUMO grant from KU Leuven (Jumo/12/032) and a ANR grant  (ANR-10-BLANC 0114 ArShiFo), 
%KU Leuven, Department of Mathematics, Celestijnenlaan 200B - box 2400, BE - 3001 Heverlee,\\
%D\'epartement de Math\'ematiques, UMR 7539, LAGA Institut Galil\'ee, Universit\'e Paris 13, FR - 93430  Villetaneuse}
}}
\maketitle

Let $p \geq 3$ be a prime and $F$ a totally real number field. Let $\boldf$ be a Hilbert  cuspidal eigenform of parallel weight $2$, trivial Nebentypus and ordinary at $p$. It is possible to construct a $p$-adic $L$ function which interpolates the complex $L$-function associated with the symmetric square representation of $\boldf$.  This $p$-adic $L$-function vanishes at $s=1$ even if the complex $L$-function does not. 
Assuming $p$ inert and $\boldf$ Steinberg at $p$, we give a formula for the $p$-adic derivative at $s=1$ of this $p$-adic $L$-function, generalizing unpublished work of Greenberg and Tilouine. Under some hypotheses on the conductor of $\boldf$ we  prove a particular case of a conjecture of Greenberg on trivial zeros. 

\tableofcontents

\section{Introduction}

The aim of this paper is to prove a conjecture of Greenberg on trivial zeros of $p$-adic $L$-functions \cite{TTT}, generalizing the classical conjecture of Mazur-Tate-Teitelbaum \cite{MTT} to ordinary motives.\\
Let $V=\left\{V_l \right\}$ be a compatible system of continous, finite-dimensional $l$-adic representations of $G_{\Q}=\mathrm{Gal}(\Qb/\Q)$. Let us denote by $V^*$ the dual motive whose $l$-adic realizations are $V_l^*=\mathrm{Hom}(V_l,\Q_l(1))$.\\
We can associate in a standard way a complex $L$-function $L(s,V)$, converging for $Re(s) \gg 0$; we suppose that $L(s,V)$ extends to an holomorphic function on the whole complex plane and that it satisfies the functional equation
\begin{align*}
L(s,V)\G(s,V) = \beps(s)L(1-s,V^*)\G(s,V^*)
\end{align*}
where $\G(V,s)$ is a product of complex and real Gamma functions and $\beps(s)$ a function of the form $\zeta N^{s}$, for $\zeta$ a root of unity and $N$ a positive integer. We suppose moreover that both $\G(0,V)$ and $\G(1,V^*)$ are defined; in this case $s=1$ is critical \`a la Deligne \cite{Del} and we suppose in addition the existence of a Deligne period $\Omega$ such that 
\begin{align*}
\frac{L(V,0)}{\Omega} \in \Qb.
\end{align*}
Let us fix a prime number $p$;  we will fix once and for all an isomorphism $\mathbb{C} \cong \mathbb{C}_p$. This in particular defines an embedding of $\overline{\Q}$ in $\mathbb{C}$ and $\mathbb{C}_p$. We say that $V$ is ordinary at $p$ if $V_p$ as $G_{\Q_p}=\mathrm{Gal}(\Qb_p/\Q_p)$-representation admits a decreasing  and exhaustive filtration $\mathrm{Fil}^i V_p$ such that $I_p$ acts on $\mathrm{Gr}^i V_p$ as $\chi_{\mathrm{cycl}}^i$, where $\chi_{\mathrm{cycl}}$ denotes the $p$-adic cyclotomic character.\\
Let us denote by $K$ the coefficient field of $V_p$ and by $\oo$ its valuation ring. Under these hypotheses, Coates and Perrin-Riou \cite{CPR} have formulated the following conjecture;
\begin{conj}
There exist a formal series $G(T,V)$ in $\oo[[T]]$ such that for all finite order, non trivial characters $\eps$ of $1+p\Z_p $ we have 
\begin{align*}
G(\eps(1+p)-1,T)= C_{\eps}\frac{L(V\otimes \eps,0)}{\Omega};
\end{align*}
here $C_{\eps}$ is an explicit determined algebraic number. Moreover 
\begin{align*}
G(0)= \calE(V)\frac{L(V,0)}{\Omega};
\end{align*}
where $\calE(V)$ is  a finite product of Euler type factors at $p$.
\end{conj}
A description of $\calE(V)$ when $V_p$ is semistable can be found in \cite[\S 0.1]{BenEZ}. \\
It is clear from the above interpolation formula that if one of the factors of $\calE(V)$ vanishes then the $p$-adic $L$-function does not give us any infomation about the special value $\frac{L(V,0)}{\Omega}$. In this case, we say that we have an ``exceptional'' or ``trivial'' zero. In \cite{TTT}, Greenberg proposes a conjecture which generalizes the well known conjecture of Mazur-Tate-Teitelbaum.
\begin{conj}\label{MainCo}
Let us denote by $g$ the number of factors of $\calE(V)$ which vanish. We have then 
\begin{align*}
G(T,V)=\Ll(V) \calE^*(V)\frac{L(V,0)}{\Omega} \frac{T^g}{g!} + (T^{g+1})
\end{align*}
where  $\calE^*(V)$ denotes the non-zero factor of $\calE(V)$ and $\Ll(V)$ is a non zero error term.
\end{conj}
The factor $\Ll(V)$, usually called the $\Ll$-invariant, has conjecturally an interpretation in terms of the Galois cohomology of $V_p$.
At present, not many cases of these conjecture are known; up to the 90's, the conjecture was known essentialy when $V$ is a Dirichlet character \cite{FG} or an elliptic curve over $\Q$ with split multiplicative reduction at $p$ \cite{SSS} or the symmetric square of an elliptic curve with multiplicative reduction at $p$ (Greenberg and Tilouine, unpublished).\\
Recently, many people \cite{DDP,Mok,Spi1,Spi2} have obtained positive results on this conjecture in the case when $V$ is an induction from a totally real field of the Galois representation associated with a Hecke character or a modular elliptic curve; these cases can be seen both as an attempt to test the conjecture for higher rank $V$ and as a (highly non-trivial) generalization of the known cases. The natural next step is the case of the induction of the symmetric square representation associated with an elliptic curve over a totally real field;  it is the main subject of this paper.\\
Let $F$ be a totally real number field and $\rif$ its ring of integers, let $I$ be the set of its real embeddings. Let $\boldf$ be a Hilbert modular form of weight $k={(k_{\sigma})}_{\sigma \in I}$, $k_{\sigma} \geq 2$ for all $\sigma$. Let $K \subset \C_p$ be a $p$-adic field containing the Hecke eigenvalues of $\boldf$ and let $\oo$ be its valuation ring. It is well known \cite{BR,Tayl} that there exists a $2$-dimensional $p$-adic Galois representation 
\begin{align*}
\rho_{\boldf}: & \mathrm{Gal}(\overline{F}/F) \rightarrow \gl(\oo)
\end{align*}
associated with $\boldf$. Let $\mathrm{Sym}^2(\rho_{\boldf})$ be the symmetric square of $\rho_{\boldf}$, it is a $3$-dimensional Galois representation. For almost all primes $\qfrak$ of $\rif$, the Euler factor of $L(s,\mathrm{Sym}^2(\rho_{\boldf}))$ is
\begin{align*}
 (1 - {\alpha(\qfrak)}^{2}{\N(\qfrak)}^{-s})(1 - { \alpha \beta(\qfrak)}{\N(\qfrak)}^{-s})(1 - {\beta(\qfrak)}^2{\N(\qfrak)}^{-s})
\end{align*}  
where $\alpha(\qfrak)$ and $\beta(\qfrak)$ are the two roots of the Hecke polynomial of $\boldf$ at $\qfrak$.\\
We define similarly the twisted $L$-function $L(s,\mathrm{Sym}^2(\rho_{\boldf}),\chi)$ for $\chi$ a gr\"ossencharacter of $\Aa_F^{\times}$.
We shall write $L(s,\mathrm{Sym}^2(\boldf),\chi)=L(s,\mathrm{Sym}^2(\rho_{\boldf}),\chi)$. Let us denote by $k^0 = \mathrm{max}(k_{\sigma})$ and  $k_0 = \mathrm{min}(k_{\sigma})$; in \cite{Im}, Im proves Deligne's conjecture for this $L$-function, i.e. he shows for any integer $n$ in the critical strip $\left[k^0 - k_0 + 1, k^0 -1 \right]$  such that $\chi_{v}(-1)={(-1)}^{n+1}$ for all $v \mid \infty$ the existence of non-zero complex number $\Omega(\boldf,n)$ such that the ratio 
\begin{align*}
\frac{L(n,\mathrm{Sym}^2(\boldf),\chi)}{\Omega(\boldf,n)}
\end{align*}
is an algebraic number. This algebraic number will be called the special value of $L(s,\mathrm{Sym}^2(\boldf),\chi)$ at $n$.\\
We can make a ``good choice'' of the periods $\Omega(\boldf,n)$ such that the special values can be $p$-adically interpolated. Actually, there exists a  power series $G(T,\mathrm{Sym}^2(\boldf))$ in $\oo[[T]]$ constructed in \cite{Wu} and in  Section \ref{padicL} of the present work which satisfies the following interpolation formula
\begin{align*}
G(\eps(1+p)(1+p)^n-1,\mathrm{Sym}^2(\boldf)) = & C_{\eps,n} E(n,\eps,\boldf) \frac{L(n,\mathrm{Sym}^2(\boldf),\eps)}{\Omega(\boldf,n)}.
\end{align*}
for $n$ critical. Here $\eps$ is a finite order character of $1+p\Z_p$, $C_{\eps,n}$ is a non-zero explicit number, and $E(n,\eps,\boldf)$ is a product of  some Euler-type factors at primes above $p$ as predicted by the above conjecture. \\
Let us pose $L_p(s,\mathrm{Sym}^2(\boldf)):=  G(u^s -1,\mathrm{Sym}^2(\boldf))$. The trivial zeros occur when the Nebentypus of $\boldf$ is trivial at $p$ and $s=k^0-1$. In this case the number of vanishing factors is equal to the number of primes above $p$ in $F$. The problem of trivial zeros for the symmetric square has already been studied in great generality by Hida \cite{ASG,HIwa,HTate1}. Hida calculated the $\Ll$-invariant, defined according to the cohomological definition of Greenberg \cite{TTT}, in  \cite{HIwa}.  It is a hard problem to show that it is non-zero, but we know that it is ``frequently'' non-zero when $\boldf$ varies in a $p$-adic family which is Steinberg for all primes above $p$ except possibly one. \\
The main result of the paper is the following 
\begin{theo}\label{MainTh}
Let $p\geq 3$ be a prime such that there is only one prime ideal $\pfrak$ of $F$ above $p$ and let $\boldf$ be a Hilbert  cuspidal eigenform  of parallel weight $2$ and conductor $\Nfrak \pfrak$  such that the Nebentypus of $\boldf$ is trivial.  Suppose that $\Nfrak$ is squarefree and divisible by all the primes of $F$ above $2$; suppose moreover that $\pi(\boldf)_{\pfrak}$ is a Steinberg representation. Then the formula for the derivative in Conjecture \ref{MainCo} is true. \end{theo}
In the case where $\boldf$ corresponds to an elliptic curve, such a theorem has been claimed  in \cite{RosCR}. We were able to substitute the hypothesis {\it $p$ inert} by {\it there is only one prime ideal of $F$ above $p$} thanks to a clever remark of \'Eric Urban to whom we are very grateful.\\
In this particular case of multiplicative reduction, the $\Ll$-invariant depends only on the restriction of $V_p$ to $G_{\Q_p}$ and was already calculated in \cite{SSS}. By  the ``{\it Th\'eor\`eme de Saint-\'Etienne}'' \cite{Saint}, we know that $\Ll(\mathrm{Sym}^2(\boldf))$ is non-zero when $\boldf$ is associated with an elliptic curve. \\
In reality, we could weaken slightly the hypothesis on the conductor; see Section \ref{proofMT} for one example.\\
%A direct consequence of the proof is that if $g \geq 2$, then we have also $\mathrm{ord}_{s=1} L_p(s,\mathrm{Sym}^2(\boldf)) \geq 2$.\\

We point out that this theorem is a generalization of the theorem (unpublished) of Greenberg and Tilouine when $F=\Q$. In fact, our proof is the natural generalization of theirs to the totally real case.\\
Recently, again in the case $F=\Q$, Dasgupta has shown Conjecture \ref{MainCo} in the case the symmetric square of modular forms of any weights. He uses the strategy outlined in \cite{Citro} which suggests to factor the three variable $p$-adic $L$-function $L_p(k,l,s)$ associated with the Rankin product of a Hida family with itself \cite{H1bis} when $k=l$. To do this, he exploits recent results of Bertolini, Darmon and Rotger expressing the values at $s=l=k=2$ of $L_p(\eps_1 k,\eps_1 l,\eps_2 s)$ in term of complex and $p$-adic regulators of Beilinson-Flach elements (and circular units). Clearly at this moment such a method is almost impossible to generalize to an arbitrary  totally real field.\\
 
The proof of Theorem \ref{MainTh} follows the same line as in the work of Greenberg and Stevens for the standard $L$-function. \\
We view the $p$-adic $L$-function $L_p(s,\mathrm{Sym}^2(\boldf))$ as the specialization of a two-variable $p$-adic $L$-function $L_p(s,k)$. Such a function vanishes identically on the line $s=k-1$ so that we can relate the derivative with respect to $s$ to the one with respect to $k$. In order to evaluate the derivative with respect to $k$, we  find a factorization of $L_p(1,k)$ into two $p$-analytic functions, one which gives rise to the trivial zero (it corresponds to the product of some Euler factors at primes above $p$) and one which interpolates the complex $L$-value (the {\it improved } $p$-adic $L$-function in the terminology of Greenberg-Stevens). The construction of the improved $L$-functions is done as in \cite{HT} by substituting the convolution of two measures used in the construction of $L_p(s,m)$ by a product of a function and a measure.\\
The main limit of this approach is due to the fact that we are using only two variables and so, such a method can not go beyond the first derivative without a new idea. In fact, one can define  a $p$-adic $L$-function associated with the nearly ordinary family deforming $\boldf$ which has $d+2 + \delta$ variables; however, we can only find a factorization of this multi variable $p$-adic $L$-function only after restricting to the ordinary part. This is because the Euler factors which we remove from our $p$-adic $L$-functions are not $p$-adic analytic functions of the {\it nearly ordinary} variables. Moreover, this method works only when $\boldf$ is Steinberg at $p$ and such representation at $p$ in the nearly ordinary setting can appear only when the weight is parallel and equal to $2t$.\\
Still, we want to point out that in the case when $\boldf$ is Steinberg at all primes above $p$ we can show a formula similar to one in Conjecture \ref{MainCo} but for the derivative with respect to the weight. It seems quite an hard problem to relate this derivative to the one with respect to $s$.\\
As this method relies on Shimura's integral formula which relates the symmetric square $L$-function to a suitable Petersson product, it can be generalized to all the cases of finite slope families,  assuming similar hypothesis and the existence of the many variable $p$-adic $L$-function. In particular, thanks to the recent work of Urban \cite{UrbNholo} on families of nearly-overconvergent modular forms, this method can now be successfully applied in the case $F=\Q$. This is the main subject of \cite{RosOC}.\\
We point out that understanding the order of this zero and an exact formula for the derivative of the $p$-adic $L$-function has recently become more important after the work of Urban on the main conjecture for the symmetric square representation \cite{Urb} of an elliptic modular form. We shall explain how its proof can be completed with Theorem \ref{MainTh}.\\
We shall also extend the result of Theorem \ref{MainTh} to the base change of a Hilbert modular form as in Theorem \ref{MainTh}, providing examples of the conjecture for trivial zeros of higher order.\\

This paper is the result of the author's PhD thesis and would have never seen the light of the day without the guidance of Prof. Jacques Tilouine and his great helpfulness. The author would like to thank him very much. \\
The author would like to thank also Arno Kret for his great patience in answering all the author's questions and Prof. Ralph Greenberg, Prof. Haruzo Hida, Chung Pang Mok, Prof. Johannes Nicaise, Prof. \'Eric Urban and John Welliaveetil for useful discussions and suggestions. Part of this work has been written during a stay at the Hausdorff Institute during the {\it Arithmetic and Geometry} program, the author would like to thank the organizers and the institute for the support and the optimal working condition.\\

The structure of the paper is the following: in  Section \ref{Hilbertforms} we recall the theory of complex and $p$-adic modular Hilbert modular forms of integral weight and in Section \ref{HalfHilbert} we do the same for the half-integral weight, presenting our definition of $p$-adic modular form of half-integral weight. In Section \ref{LFun} we recall the definition of the $L$-function for the symmetric square of $\boldf$ and relate it to the Rankin product of $\boldf$ with a theta series.
In Section \ref{operators} we present some operators which will be used in the construction of $p$-adic $L$-function, and in Section \ref{p-measure} we recall briefly the theory of $p$-adic measures.
In Section \ref{padicL} we use the method of Hida and Wu to construct a many variable $p$-adic $L$-function for the symmetric square (relaxing also some of their hypotheses) and the {\it improved} one, while in Section \ref{proofMT} we prove Theorem \ref{MainTh}. In Section \ref{proofBC} we deal with the case of  base change. In Section \ref{AppMC} we explain how our result allows one to complete the proof of the main conjecture given by Urban.
Finally, in Appendix \ref{App} and \ref{AppB} we prove a theorem on the holomorphicity of the $p$-adic $L$-functions constructed in Section \ref{padicL}.

%%%%%%%%%%%%%%%%%%%%%%%%%%%%%%%%%%%%%%%%%%%%%%%%%%%%%%%%%%%%%%%%%%%%%%%%%%%%%%%%%%%%%%%%%%%%%%%%%%%%%%%%%%%%%%%%%%%%%%%%%%%%%%%%%%%%%%%%%%%%%%%

\section{Classical Hilbert modular forms}\label{Hilbertforms}
In this section we review the theory of complex and $p$-adic Hilbert modular forms.
\subsection{Complex Hilbert modular forms}\label{ComplexHilbert}
As our interest lies more in Hecke algebras than in modular forms {\it per se}, we look at adelic Hilbert modular forms. Let $F$ be a totally real field of degree $d$ and $\rif$ its ring of integers. Let $I$ be the set of its real embeddings. An element $w \in \Z[I]$ is called a weight. For two weights $k=\sum_{\sigma \in I} k_{\sigma} \sigma$ and $w=\sum_{\sigma \in I} w_{\sigma}$, we say that $k \geq w$ if $k_{\sigma} \geq w_{\sigma}$ for all $\sigma$ in $I$. In particular $k \geq 0$ if $k_{\sigma} \geq 0$ for all $\sigma$ in $I$. Define $t=\sum_{\sigma \in I} \sigma$.\\
Let $m \geq 0$ be an integer and $v$ in $\Z[I]$. Consider  a Hilbert  cuspidal eigenform $\boldf$ of weight $k = (m+2)t - 2v$ and define $k_0$ as the minimum of $k_{\sigma}$ for $\sigma$ in $I$. Let $\dfrak$ be the different of $F$ and $D_F=\N_{F/\Q}(\dfrak)$ the discriminant of $F$. By abuse of notation, we shall denote by $\dfrak$ also a fixed id\`ele which is $1$ at the primes which do not divide the different,  and is a generator of the local different at the primes dividing it. \\ 
We denote by $\fa$ the ad\`eles of $F$, and we factor $\fa=F_f  \times F_{\infty}$ into its finite component and infinite  component. For $y$ in $\fa$, we denote by $y_f$ resp. $y_{\infty}$ the projection of $y$ to $F_f$ resp. $F_{\infty}$.\\
Let $\gl$ be the algebraic group, defined over $F$, of $2 \times 2$ invertible matrices. For $y$ in $\gl(\fa)$, we denote by $y_f$, resp. $y_{\infty}$, the projection of $y$ to $\gl(F_f)$ resp. $\gl(F_{\infty})$.\\
Let $S$ be a compact open subgroup of $\gl(\faf)$ and let $\gl(F_{\infty})^+$ be the connected component of $\gl(\fai)$ containing the identity. Let $\h$ be the upper half plane then $\gl(F_{\infty})^+$ acts on $\h^I$ by linear fractional transformation. Let  $C_{\infty +}$ be the stabilizer of the point $z_0=(i,\ldots,i)$ in $\h^I$ under this action.\\
Fix two elements $k$ and $w$ of $\Z[I]$ such that $k \geq 0$. Let $t=\sum_{\sigma \in I} \sigma$. We say that an element of $\Z[I]$ is parallel if it is of the form $m t$ with $m$ an integer.\\
We use the multi-index notation; for $x$ in $F$ and $k=\sum_{\sigma \in I} k_{\sigma} \sigma$, $z^k$ will denote $\prod_{\sigma \in I} \sigma(z)^{k_{\sigma}}$. \\
We define $\boldM_{k,w}(S,\C)$ as the space of complex-valued functions on $\gl(\fa)$,  holomorphic on $\gl(\fai)$ such that
$$ \boldf (\alpha x u) = \boldf(x){j_{k,w}(u_{\infty},z_0)}^{-1} \mbox{ for } \alpha \in \gl(F) \mbox{ and } u \in S C_{\infty +}, $$
where $j_{k,w}\left( \left( \begin{array}{cc} a & b \\ c & d \end{array}\right) ,z \right) = {(ad-bc)}^{-w}{(cz+d)}^k$, under the multi-index convention.

We say that $\boldf$ is cuspidal if 
\begin{align*}
  \int_{F \setminus \fa} f(ga)\textup{d} a = 0 & \mbox{ for all } g \in \gl(\fa);
\end{align*}
we shall denote by  $\boldS_{k,w}(S,\C)$ the subset of cuspidal modular forms. \\
Fix a pair $(n,v)$ such that $k = n + 2t$ and $w = t - v$, and suppose that $n + 2v=mt$ (otherwise there are no non-zero Hilbert modular forms).
The choice of $v \in \Z[I]$ guarantees that the Fourier coefficients of $\boldf$, defined below, belong to a number field.\\
For an element $k \in \Q[I]$, we will denote by $[k]$ the integer such that $k=[k]t$ if $k$ is parallel or $0$ otherwise.\\
We define some congruence subgroups of $\gl(\fa)$. Let $\hat{\rif}$ be the profinite completion of  $\rif$, then let 
\begin{eqnarray*}
\tilde{U}_0(\Nfrak) = & \left\{ \left( \begin{array}{cc} a & b \\ c & d \end{array}\right) \in \gl(\hat{\rif}) | c \in \Nfrak \hat{\rif} \right\},\\
\tilde{U}_0(\Nfrak,\Mfrak) = & \left\{ \left( \begin{array}{cc} a & b \\ c & d \end{array}\right) \in U_0(\Nfrak) | b\in \Mfrak \hat{\rif} \right\}, \\
\tilde{V}_1(\Nfrak) = & \left\{ \left( \begin{array}{cc} a & b \\ c & d \end{array}\right) \in U_0(\Nfrak) | d \equiv 1 \bmod \Nfrak \hat{\rif} \right\},\\
\tilde{U}(\Nfrak) = & \left\{ \left( \begin{array}{cc} a & b \\ c & d \end{array}\right) \in V_1(\Nfrak) | a \equiv 1 \bmod \Nfrak \hat{\rif} \right\},\\
\tilde{U}(\Nfrak,\Mfrak) = & \left\{ \left( \begin{array}{cc} a & b \\ c & d \end{array}\right) \in U(\Nfrak) | b\in \Mfrak \hat{\rif} \mbox{ and } a \equiv d \equiv 1 \bmod \Nfrak\Mfrak \hat{\rif} \right\}
\end{eqnarray*}
for two integral ideals $\Nfrak,\Mfrak$ of $\rif$. The subgroup $\tilde{U}(\Nfrak,\Mfrak)$ is isomorphic to $\tilde{U}(\Nfrak\Mfrak)$ and we will use it in the following to simplify some calculations in the construction of the $p$-adic $L$-function.\\

We say that $\tilde{S}$ is congruence subgroup if $\tilde{S}$ contains $\tilde{U}(\Nfrak)$ for a certain ideal $\Nfrak$. For such a $\tilde{S}$,  we pose $\mathrm{det}\tilde{S}= \lgr \mathrm{det}(s) | s \in  \tilde{S} \rgr \subset \faft$, and decompose the id\`eles into 

\begin{align}\label{classdecomp}
 \fa^{\times} = & \bigcup_{i=1}^{h(S)} F^{\times}a_i \mathrm{det}\tilde{S}_F F_{\infty +}^{\times}. 
\end{align}
Here $h(S) $ is a positive integer and  $\lgr a_i \rgr$ is a set of id\`eles such that the $\lgr \afrak_i = a_i\rif \rgr$ form a set of representatives. We can and we will choose each $a_i$ such that $a_{i,p}=1$ and $a_{i,\infty}=1$. When $\tilde{S}=\tilde{U}(\Nfrak)$, we will use the notation $h(\Nfrak)$ for $h(S)$, as in this case $h(S)$ is the strict class number of $F$ of level $\Nfrak$. 
Define now 

$$ E = \left\{ \eps \in \rif^{\times} | \eps \gg 0 \right\}, \: \: 
\rif^{\times}(\Nfrak) = \left\{ \eps \in \rif^{\times} | \eps \equiv 1 \bmod \Nfrak \right\}, $$
$$E(\Nfrak) = E \cap \rif^{\times}(\Nfrak) \mbox{ and }  $$
$$\tilde{\G}[\Nfrak, \afrak] = \left\{  
\left( \begin{array}{cc} 
a & b \\ 
c & d 
\end{array}\right) \in  
\left( \begin{array}{cc} 
\rif     & \afrak \\
\Nfrak  & \rif 
 \end{array} \right) 
| ad-bc \in E,  \right\}.$$
 
In this way we can associate to each adelic form $\boldf$ for $\tilde{U}(\Nfrak)$ exactly $h(\Nfrak)$ complex Hilbert modular forms,
$$ \boldf_i(z) = j_{k,w}(u_{\infty},z)\boldf(t_i u_{\infty})=y_{\infty}^{-w}\boldf\left(t_i \left( \begin{array}{cc} y_{\infty} & x_{\infty} \\ 0 & 1 \end{array}\right) \right)   $$
where $u_{\infty} $ is such that $u_{\infty}(z_0)=z$. Each $\boldf_i$ is an element of $\boldM_{k,w}(\tilde{\G}^{\equiv}[\Nfrak\afrak_i,\afrak_i^{-1}])$, where 
\begin{align*}
 \tilde{\G}^{\equiv}[\Nfrak\afrak_i,\afrak_i^{-1}] = \left\{  
\left( \begin{array}{cc} 
a & b \\ 
c & d 
\end{array}\right) \in  \tilde{\G}[\Nfrak, \afrak]
| a \equiv d \bmod \Nfrak,  \right\}.
\end{align*}
This allows us to define the following isomorphism
$$
\begin{array}{cccc}
\I^{-1}_{\Nfrak} : & \boldM_{k,v}(U(\Nfrak),\C) & \rightarrow & \bigoplus_{i} \boldM_{k,v}(\tilde{\G}^{\equiv}[\Nfrak\afrak_i,\afrak_i^{-1}],\C) \\
& \boldf & \mapsto & (\boldf_1, \ldots, \boldf_{h(\Nfrak)})
\end{array}
$$
We point out that each $\boldf_i(z)$ is a Hilbert modular form on a different connected component of the Shimura variety associated with $\tilde{U}(\Nfrak)$.\\
Each $\boldf_i(z)$ has a Fourier expansion
$$ \boldf_i(z)= a(0,\boldf_i) + \sum_{0 \ll \xi \in \afrak\dfrak^{-1}} a(\xi,\boldf_i)\bexp_F(\xi z). $$
As in the classical case, $\boldf$ is cuspidal if and only if  $a(0,\boldf_i| \gamma)=0$ for $i=1, \ldots, h(\Nfrak)$ and all the matrices $\alpha \in \gl(F_{\infty})^+$.\\ 
Let $K_0$ be a finite extension of the Galois closure of $F$, and suppose that in its ring of integers all the integral ideals of $F$ are principal. We choose a compatible system of generators $\lgr y^{\sigma}\rgr$ for all id\`eles $y$ as in \cite[\S 3 ]{Hr1}. Let $K$ be the $p$-adic completion of $K_0$ with respect to the embedding of $\overline{\Q}$ in $\mathbb{C}_p$ fixed in the introduction. When we will be working with coefficients in $K$, we will suppose $\lgr y \rgr =1$ if $y$ is prime to $p$.\\
We define now functions on $F_{\mathbb{A}+}^{\times}$; let $y$ be an integral  id\`ele,  by the  decomposition given in (\ref{classdecomp}) write $y=\xi a_i \dfrak u$, for $u$ in $U_F(N)F_{\infty +}^{\times}$, and define 
\begin{align*}
\bolda(y,\boldf)& = a(\xi,\boldf_i)\lgr y^{-v}\rgr \xi^v |a_i|_{\mathbb{A}}^{-1},\\ 
\bolda_p(y,\boldf)& = a(\xi,\boldf_i) y_p^{-v} \xi^v {\mathcal{N}_p(a_i)}^{-1}
\end{align*}
and $0$ if $y$ is not integral. Here $\mathcal{N}_p$ is the cyclotomic character such that $\mathcal{N}_p(y)=y_p^{-t}{|y_f|}^{-1}_{\mathbb{A}}$. If the infinity-part and the $p$-part of $a_i$ are 1 as we chose before, then 
$$\bolda_p(y,\boldf) = \bolda(y,\boldf) \lgr y^{v}\rgr y_p^{-v}. $$
Multiplication by $\xi^v$ is necessary to ensure that these functions are independent of the choice of the decomposition of $y$. We define also the constant term
$$\bolda_0(y,\boldf)=a(0,\boldf_i)|a_i|_{\mathbb{A}}^{1-[v]}, \bolda_{0,p}(y,\boldf)=\bolda_0(y,\boldf){\cycl_p(y\dfrak^{-1})}^{[v]}, $$
and we can now state a proposition on the Fourier expansion of $\boldf$.
\begin{prop}\cite[Theorem 1.1]{Hr3}
Let  $\boldf$ be an Hilbert modular form of level $ \tilde{V}_1(\Nfrak)$, then $\boldf$ has a Fourier expansion of the form
\begin{align}\label{FourierHida}
\boldf\left( \left( \begin{array}{cc} y  & x  \\ 0 & 1 \end{array}\right) \right) = & |y|_{\mathbb{A}} \lgr \bolda_0(y\dfrak,\boldf)|y|_{\mathbb{A}}^{-[v]} +\right. \notag\\
 & \left. + \sum_{ 0 \ll \xi \in F^{\times}} \bolda(\xi \dfrak y,\boldf){\lgr \xi \dfrak y \rgr}^{v}{( \xi y_{\infty})}^{-v}\bexp_{\infty}(i\xi y_{\infty})\bexp_F(\xi x) \rgr .
\end{align}
and by formal substitution we obtain a $p$-adic $q$-expansion 
\begin{align*}
 \boldf = & {\cycl_p(y)}^{-1}\lgr \bolda_{0,p}(y\dfrak,\boldf)|y|_{\mathbb{A}}^{-[v]} + \sum_{0 \ll \xi \in F^{\times}} \bolda_p(\xi \dfrak y,\boldf)q^{\xi} \rgr.
\end{align*}
Here $\bexp_F$ is the standard additive character of $\fa / F$ such that $\bexp_{\infty}(x_\infty)=\exp(2 \pi i x_{\infty})$.
\end{prop}
We define 
\begin{align*} \boldf^u (x) = {|D_F|}^{-(m/2)-1}\overline{\boldf^{\rho}}(x)j(x_{\infty},z_0)^{k\rho} |det(x)|^{m/2}_{\mathbb{A}}; \end{align*}
$\boldf^u$ is the unitarization of (the complex conjugate of) $\boldf$. If $\boldf$ is of weight $(k,w)$, $\boldf^u$ is of weight $(k,k/2)$.\\
For arithmetic applications, we define the normalized Fourier coefficient as 
$$ C(\afrak,\boldf) = {\xi}^{-k/2}\bolda(\xi,\boldf^u), $$ where  we have $\afrak=\xi a_i^{-1}$. 
From \cite[(1.3 a)(4.3 b)]{Hr3} we have the relation
\begin{equation}\label{CoeffAr} 
C(\afrak,\boldf) =  a(\xi,\boldf_i) {\N(\afrak)}^{-m/2 -1} \xi^{v}|a_i|_{\Aa}^{-1} . 
\end{equation} 
From now on, following Shimura, we conjugate all the previous subgroups by $\left( \begin{array}{cc}\dfrak & 0 \\ 0 & 1 \end{array}\right)$.
The conjugated groups will be denoted by the same symbol without the tilde. The advantage of this choice is that the Fourier expansion of a Hilbert modular form $\boldf$ for ${V}_1(\Nfrak)$ is now indexed by totally positive elements in $\rif $ instead of totally positive elements belonging to the fractional ideal $\dfrak^{-1} $.

\subsection{The $p$-adic theory}\label{ptheory}
In \cite{Hr1,Hr2,HPEL} Hida develops the theory of nearly ordinary modular forms.  It is constructed via the duality  between modular forms and  their associated Hecke algebra.
Let $v \geq 0$ and $ k>0$ be a couple of weights and $m \geq 0$ an integer such that  $mt= k -2t -2v$; let $w=t-v$. Let  $S$ be a compact-open subgroup of $\gl(\hat{\rif})$, $U_0(\Nfrak) \supset S \supset U(\Nfrak)$. We suppose $\Nfrak$ prime to $p$.\\
Let $K$ be a $p$-adic field containing the Galois closure of $F$ and $\oo$ its ring of integers. For all integral domain $A$ which are $\oo$-algebra  we define :
\begin{align*}
\boldM_{k,w}(S,A) = & \lgr \boldf \in \boldM_{k,w}(S,\C) \mid \bolda_0(y,\boldf), \bolda(y,\boldf) \in A \rgr, \\
\boldS_{k,w}(S,A) = & \lgr \boldf \in \boldS_{k,w}(S,\C) \mid  \bolda(y,\boldf) \in A \rgr, \\
\boldm_{k,w}(S,A) = & \lgr \boldf \in \boldM_{k,w}(S,\C) \mid \bolda(y,\boldf) \in A \rgr .
\end{align*}

For all $ y \in \faft$, we define the Hecke operator $T(y)$ as in \cite[\S 3]{Hr1} and we pose $T_0(y)=\lgr y^{-v} \rgr T(y)$. $T(y)$ is an operator on $\boldM_{k,w}(S,A)$ and $\boldm_{k,w}(S,A)$ and, if $\lgr y^{-v} \rgr$ belongs to $A$ for all $y$, then $T_0(y)$ acts on both spaces. Furthermore, we  define $T(a,b)$ as in \cite[\S 2]{Hr3}. \\
We define the Hecke algebra $\boldH_{k,w}(S,A)$ (resp. $\boldh_{k,w}(S,A)$) as the sub-algebra of $\mbox{End}_A(\boldM_{k,w}(S,A))$ (resp. $\mbox{End}_A(\boldS_{k,w}(S,A))$) generated by $T(y)$ and $T(a,b)$.\\
For a positive integer $\alpha$, we pose $ S(p^{\alpha})= S \cap U(p^{\alpha})$ and define:

\begin{align*}
\boldM_{k,w}(S(p^{\infty}),A) = & \varinjlim_{\alpha } \boldM_{k,w}(S(p^{\alpha}),A),\\
\boldS_{k,w}(S(p^{\infty}),A) = & \varinjlim_{\alpha } \boldS_{k,w}(S(p^{\alpha}),A),\\
\boldm_{k,w}(S(p^{\infty}),A) = & \varinjlim_{\alpha } \boldm_{k,w}(S(p^{\alpha}),A),\\
\boldh_{k,w}(S(p^{\infty}),A) = & \varprojlim_{\alpha} \boldh_{k,w}(S(p^{\alpha}),A),\\
\boldH_{k,w}(S(p^{\infty}),A) = & \varprojlim_{\alpha} \boldH_{k,w}(S(p^{\alpha}),A).
\end{align*}

Hida shows in \cite[Theorem 2.3]{Hr2} that when $k\geq 2t$ we have an isomorphism $\boldh_{k,w}(S(p^{\infty}),A)\cong \boldh_{2t,t}(S(p^{\infty}),A)$ which sends $T_0(y)$ into $T_0(y)$. Let $\boldh_{k,w}^{\mathrm{n.ord}}(S(p^{\infty}),A)$ be the direct summand of $\boldh_{k,w}(S(p^{\infty}),A)$ where $T_0(p)$ acts as a unit. We will drop the weight from the notation in what follows. \\
 We let $e=\lim_n T_0(p)^{n!}$ be the idempotent of  $\boldh(S(p^{\infty}),A)$ which defines $ \boldh^{\mathrm{n.ord}}(S(p^{\infty}),A) $.\\ 
Let $\boldS^{\mathrm{n.ord}}(S(p^{\infty}),A)$ be defined as  $e \overline{\boldS}_{k,w}(S(p^{\infty}),A)$, where the line denotes the completion  with respect to the $p$-adic topology on the function $\bolda_p(y,\boldf)$. This completion is independent of $(k,w)$ as shown in \cite[Theorem 3.1]{Hr3}. \\
Let now $S_0(p^{\alpha})=S\cap U_0(p^{\alpha})$. We define 
\begin{align*}
 \boldG^{\alpha} = & S(p^{\alpha})\rif^{\times} / S_0(p^{\alpha})\rif^{\times}, \\
 \boldG =  & \varprojlim_{\alpha} \boldG^{\alpha}.
\end{align*}
Let $S_F$  (resp. ${S(p^{\alpha})}_F$) be $S \cap \faft$ (resp. $S(p^{\alpha}) \cap \faft$) and 
\begin{align*}
\mathrm{Cl}_S(p^{\infty})= & \varprojlim_{\alpha} S_F\rif^{\times} / {S(p^{\alpha})}_F\rif^{\times}.
\end{align*}
From \cite[Lemma 2.1]{Hr2} we have an isomorphism $\boldG \cong \rif_p^{\times} \times  \mathrm{Cl}_S(p^{\infty})$ via the map 
\begin{align}\label{GinGL2}
 S \ni s =  \left( \begin{array}{cc} 
a & b \\ 
c & d 
\end{array}\right) \mapsto (a_p^{-1}d_p,a) .
\end{align}
If $S=U(\Nfrak)$, then $\mathrm{Cl}_S(p^{\infty})=\mathrm{Cl}_{\Nfrak}( p^{\infty})$; in general, $\mathrm{Cl}_S(p^{\infty})$ is always a subgroup of $\mathrm{Cl}_\Nfrak (p^{\infty})$ of finite index. \\ 
We have an action of $\boldG^{\alpha}$  on $\boldM_{k,w}(S(p^{\alpha}),A)$; the couple $(a,z)$ acts via $T(z,z)T(a^{-1},1)$. In the following, we will denote $T(z,z)$ by $\lla z \rra$. 
For two characters $\psi$, $\psi'$ of finite order modulo $\Nfrak p^{\alpha}$, we define  $\boldS_{k,w}^{\mathrm{n.ord}}(S(p^{\alpha}),\psi,\psi';\oo)$ as the forms on which $(a,z)$ acts via $\psi(z) \N(z)^m\psi'(a)a^v $ and $\boldh_{k,w}^{\mathrm{n.ord}}(S(p^{\alpha}),\psi,\psi';\oo)$ in the obvious way.\\
We let $\boldW$ be the torsion-free part of $\boldG$. Let $\LL$ be the Iwasawa algebras of $\boldW$; we have that  $\boldh^{\mathrm{n.ord}}(S(p^{\infty}),A)$ is a module over $\LL$ and $\oo[[\boldG]]$. We say that a $\oo$-linear morphism $P: \oo[[\boldG]] \rightarrow \oo$ is arithmetic of type $(m,v,\psi,\psi')$ if $P(z,a)= \psi(z) \N(z)^m\psi'(a)a^v$ for $m \geq 0$ and $\psi$, $\psi'$ finite order characters of $\boldG^{\alpha}$.\\
We have the following theorem, which subsumes several results of Hida \cite[Theorem 2.3, 2.4]{Hr2}, \cite[Corollary 3.3]{Hr3}, \cite[\S 4]{HPEL} 
\begin{theo}\label{controlThm}
Let $S$ as above, then $\boldh^{\mathrm{n.ord}}(S(p^{\infty}),A)$ is torsion-free over $\LL$, and  for all points $P$ of type $(m,v,\psi,\psi')$ we have 
\begin{align*}
 \boldh^{\mathrm{n.ord}}(S(p^{\infty}),\oo)_P / P  \boldh^{\mathrm{n.ord}}(S(p^{\infty}),\oo)_P \cong &   \boldh_{k,w}^{\mathrm{n.ord}}(S(p^{\alpha}),\psi,\psi';\oo).
\end{align*}
Moreover, there is a duality as $\oo$-module between $\overline{\boldS}_{k,w}(S(p^{\infty}),\oo)$ (resp. $\overline{\boldm}_{k,w}(S(p^{\infty}),\oo)$) and $\boldh(S(p^{\infty}),\oo)$ (resp. $\boldH(S(p^{\infty}),\oo)$).
When $S=U(\Nfrak) $, then for all points $P$ of type $(m,v,\psi,\psi')$ we have also 
\begin{align*}
 \boldS^{\mathrm{n.ord}}(S(p^{\infty}),\oo)[P] \cong &   \boldS_{k,w}^{\mathrm{n.ord}}(S(p^{\alpha}),\psi,\psi';\oo).
\end{align*}
If moreover $p \nmid 6D_F$, then $\boldh^{\mathrm{n.ord}}(U(\Nfrak)(p^{\infty}),\oo)$ is free over $\oo[[\boldW]]$.
\end{theo}
In the following, we will concentrate on the case $S=V_1(\Nfrak)$, in what follows we will use the notation $\boldh^{\mathrm{n.ord}}(\Nfrak,\oo)$ for $\boldh^{\mathrm{n.ord}}(V_1(\Nfrak)(p^{\infty}),\oo)$. We will denote by  $\boldh^{\mathrm{ord}}(\Nfrak,\oo)$ the ordinary part of the Hecke algebra. Via the above mentioned duality, it corresponds to the Hilbert modular forms of parallel weight; it is a module over $\oo[[\mathrm{Cl}_\Nfrak (p^{\infty})]]$. We have  $\boldh^{\mathrm{ord}}(\Nfrak,\oo)=\boldh^{\mathrm{n.ord}}(\Nfrak,\oo)/P^{\mathrm{ord}},$ where $P^{\mathrm{ord}}$ is the ideal of $\oo[[\boldG]]$ corresponding to the projection on the second component $\boldG \rightarrow \mathrm{Cl}_\Nfrak (p^{\infty})$.\\

Let $\boldK$ be a field extension of $\mathrm{Frac}(\LL)$ and $\boldI$ the integral closure of $\LL$ in $\boldK$. By duality, to each morphism $\lambda:\boldh^{\mathrm{n.ord}}(\Nfrak,\oo) \rightarrow \boldI$ corresponds an $\boldI$-adic modular form $\boldF$ which is characterized by its $\boldI$-adic Fourier coefficient $\bolda_p(y,\boldF)=\lambda(T_0(y))$. To such a $\lambda$ corresponds an irreducible component of $\mathrm{Spec}(\boldh^{\mathrm{n.ord}}(\Nfrak,\oo))$, determined by the prime ideal $\mathrm{Ker}(\lambda)$. Let $X(\boldI)$ be the set of arithmetic points of $\mathrm{Spec}(\boldI)$, i.e. the $\Qb_p$-points that restrained to $\LL$ correspond to an arithmetic point of type $(m,v,\psi,\psi')$. By the previous theorem, to each arithmetic point $P$ corresponds a Hilbert modular eigenform $\boldf_P=P(\boldF)$ such that $\bolda_p(y,\boldf_P)=P(\lambda(T_0(y)))$, $\boldf_P|T(a^{-1},1)= a^v \psi'(a)\boldf_P$ and $\boldf_P|\lla z \rra = \N^m(z)\psi(z)\boldf_P$. In particular, the center of $\gl(\hat{\rif})$ acts via  $\N^{m}(-)\psi(-)$.

Following \cite[Theorem I]{Hr4}, we can associate to each family of nearly ordinary Hilbert modular forms with coefficients in a complete noetherian local integral domain of characteristic different from $2$ (i.e. to each $\lambda: \boldh^{n.ord}(\Nfrak,\oo) \rightarrow A$), a $2$-dimensional semisimple and continuous Galois representation $\rho_{\lambda}$, unramified outside $\Nfrak p$ such that for all $\qfrak \not| \Nfrak p$, we have 
\begin{align*}
 det(1 - \rho_{\lambda}(Fr_{\qfrak})) = 1 - \lambda(T(\qfrak))X + \lambda(\lla \qfrak \rra)\N(\qfrak)X^2.
\end{align*} 
Moreover this representation is nearly-ordinary at all $\pfrak | p $. For the exact definition of nearly ordinary, see \cite[Theorem I (iv)]{Hr4}. Essentially, it means that its restriction to $D_{\pfrak}$, the decomposition group at $\pfrak$, is upper triangular and unramified after a twist by a finite order character of $I_{\pfrak}$.\\

As in Section \ref{proofMT} we will discuss the relation between the analytic and the arithmetic $\Ll$-invariant, we give now briefly the definition of the cyclotomic Hecke algebra $\boldh^{\mathrm{cycl}}(\Nfrak,\oo)$; it has been introducted in \cite[3.2.9]{HIwa}. Roughly speaking, $\boldh^{\mathrm{cycl}}(\Nfrak,\oo)$ corresponds to the subspace of  $\boldh^{\mathrm{n.ord}}(\Nfrak,\oo)$ defined by the equation $v_{\sigma_1}=v_{\sigma_2}$, for $\sigma_1$ and $\sigma_2$ which induce the same $p$-adic place $\pfrak$ of $F$.\\
Consequently, we say that an arithmetic point $P$ is locally cyclotomic if $P$, as a character of ${T(\Z_p)} \equiv \lgr \left( \begin{array}{cc} a  & 0  \\ 0 & d \end{array}\right) \in \gl(\rif_p) \rgr $, factors through the local norms, i.e.,  $$P|_{{T(\Z_p)}} = \prod_{\pfrak | p } ({\N(a_{p})}^{v_{\pfrak}},{\N(d_{p})}^{m_{\pfrak}}),$$
up to a finite order character. \\ 
From the deformation-theoretic point of view, requiring that a Hilbert modular form $\boldf$ belongs to $\boldh^{\mathrm{cycl}}(\Nfrak,\oo)$ is equivalent to demand that, for all $\pfrak \mid p$ the local Galois representation at $\pfrak$, $\rho_{\boldf,\pfrak}$ is of type $\left( \begin{array}{cc}
\eps_{\pfrak} & \ast \\
0 & \delta_{\pfrak}
\end{array}\right)$, with the two characters $\eps_{\pfrak}$ and  $\delta_{p}$ of fixed type outside $\mathrm{Gal}(F_{\pfrak}(\mu_{p^{\infty}})/ F_{\pfrak})$ (\cite[\mbox{(Q4')}]{HIwa}).

%%%%%%%%%%%%%%%%%%%%%%%%%%%%%%%%%%%%%%%%%%%%%%%%%%%%%%%%%%%%%%%%%%%%%%%%%%%%%%%%%%%%%%%%%%%%%%%%%%%%%%%%%%%%%%%%%%%%%%%%%%%%%%%%%%%%%%%%%%%%%%%%%%%%%%%%%%%%%%%%%5

\section{Hilbert modular forms of half-integral weights}\label{HalfHilbert}
In this section we first recall the theory of Hilbert modular forms of half integral weight, with particular interest in theta series and Eisenstein series. We will use these series to give, in the next section, an integral expression for the $L$-function of the symmetric square representation. In the second part of the section, we develop a $p$-adic theory for half integral weight Hilbert modular forms which will be used in Section \ref{p-measure}.
\subsection{Complex forms of half-integral weights}\label{HIntweight}
We will follow in the exposition \cite{ShH3}. Half-integral weight modular forms are defined using the metaplectic groups as defined by Weil. One can define $M_{\fa}$ as a certain non trivial extension of $SL_2(\fa)$ by the complex torus $\mathbb{S}^1$ with a faithful unitary representation (called the Weil representation) on the space $L^2(\fa)$.\\
Similarly, for any place $v$, $M_v$ is defined as a certain extension of $SL_2(F_v)$ by the complex torus $\mathbb{S}^1$ together with a faithful unitary representation on $L^2(F_v)$. We denote by the same symbol $pr$ the projection of $M_{\Aa}$ onto $SL_2(\fa)$ and of $M_{v}$ onto $SL_2(F_v)$. In the following, we will sometimes use $G$ to denote the algebraic group $\Sl$.\\

Denote by $P_{\fa}$ the subgroup of upper triangular matrices in $SL_2(\fa)$  and by $\Omega_{\fa}$ the subset of $SL_2(\fa)$ consisting of matrices which have a bottom left entry invertible. \\ 
The extension $M_{\fa}$ is not split, but we have three liftings $r$, $r_P$ and $r_{\Omega}$ of $SL_2(F)$, $P_{\fa}$ and $\Omega_{\fa}$ such that $pr \circ  r_?$ is the identity, for $?= \emptyset, P, \Omega$. Their explicit description is given in \cite[(1.9),(1.10)]{ShH3} and they are compatible in the sense that $r=r_P$ on $G \cap P_{\fa}$ and $r=r_{\Omega}$ on $G \cap \Omega_{\fa}$. \\
For $x$ in $\gl(\faf)\times\gl(F_{\infty})^+$ and $z$ in $\h^I$ we define $j(x,z)=det(x_{\infty})^{-1/2}(cz+d)$, where $x_{\infty} = \left( \begin{array}{cc} a & b \\ 
c & d 
\end{array}\right)$ is the projection of $x$ to $\gl(F_{\infty})^+$. For $\tau$ in $M_{\fa}$ we pose $j(\tau,z)=j(pr(\tau),z)$. \\

We define $C'$ as the matrices $\gamma$ in $G(\fa)$ such that $\gamma_v$ belongs to the group $C'_v$ defined as follows
$$
C'_v = 
\left\{
\begin{array}{cc}  
 \left\{ \gamma \in SL_2(F_v) | \gamma i = i  \right\}  & \mbox { if } v|\infty  \\
  \left\{ \gamma \in SL_2(F_v) | \gamma \in
\left( 
\begin{array}{cc}
 * & 2\dfrak_v \\
 2\dfrak_v^{-1} & *
\end{array}
\right)                          \right\}        & \mbox { if } v \mbox{ finite}

\end{array}
\right. .
$$

Let $\eta_0$ be defined by

$$
(\eta_0)_v = 
\left\{
\begin{array}{cc}  
  1       & \mbox { if } v|\infty  \\

\left( 
\begin{array}{cc}
 0 &-\dfrak_v^{-1} \\
 \dfrak_v & 0
\end{array}
\right)
        & \mbox { if } v \mbox{ finite} 

\end{array}
\right. .
$$

$C''$ is define as the union of $C'$ and $C'\eta_0$.
By \cite[Proposition 2.3]{ShH3}, we can define for $\tau$ in $pr^{-1}(P_{\fa}C'')$ an holomorphic function $h(\tau,z)$ such that
\begin{align*}
h(\tau,z)^2 = & \zeta_{\tau}j(\tau,z), 
\end{align*}
where $\zeta_{\tau}$ is a certain fourth root of $1$. \\
Let $\G$ be a congruence subgroup contained in $C''G(F_{\infty})$, we define a Hilbert modular form of half-integral weight $k + \frac{1}{2}t$, $k$ in  $\Z[I]$ and $k \geq 0$, as a holomorphic function on $\h^{I}$ such that 
$$ f|_{k + \frac{1}{2}t}\gamma(z) = f(\gamma z) {j(\gamma, z)}^{-k}{h(\gamma,z)}^{-1} = f(z) $$ for all $\gamma$ in $\G$.  
Contrary to the case of $\gl$, such complex modular forms are the same as the functions $\boldf :M_{\Aa} \rightarrow \mathbb{C}$ such that 
$$ \boldf(\alpha x w) = \boldf(x) {j(w, z_0 )}^{-k}{h(w,z_0)}^{-1} \mbox{ for all } \alpha \in \Sl(F), w \in pr^{-1}(B),$$ where $B$ is a compact open subgroups of $C''$.\\
In fact, we can pass from one formulation to the other in the following way: to such a $\boldf$ we associate a Hilbert modular form $f$ on $G \cap B\G_{\infty}$ by $f(z)=\boldf(u){j(u, z_0)}^{-k}{h(u,z_0)}^{-1}$ for any $u$ such that $u(z_0)=z$.
Conversely, to such $f$ we can associate $\boldf$ such that
$$ \boldf(\alpha x) = f|_{(k + \frac{1}{2}t)} x (z_0) \mbox{ for all } \alpha \in \Sl(F), x \in pr^{-1}(BG(\fai)).$$

Let $\Nfrak$ and $\Mfrak$ be two fractional ideals such that $\Nfrak\Mfrak$ is integral and divisible by $4$.  
Congruence subgroups of interest are $D(\Nfrak,\Mfrak)$ which are defined as the intersection of $G(\Aa)$ and $U_0(\Nfrak,\Mfrak)\times SO(2)^{\bolda}$. The intersection of $D(\Nfrak,\Mfrak)$ with $G$ shall be denote by $\G^1[\Nfrak,\Mfrak]$. In particular, assuming $4 \mid \Nfrak$, $h(-,z)$ is a factor of automorphy and the map $\chi_{-1}:\tau \mapsto \zeta_{\tau}^2$, for the root of unit $\zeta_{\tau}$ defined above, is a quadratic character of $D(\Nfrak,4)$ depending only on the image of $d_{\gamma}$ modulo $4$. \\
Let $k'=k + \frac{1}{2}t$, with $k \in \Z[I]$ and $k \geq 0$ and choose a Hecke character $\psi$ of $\fa^{\times}$  (i.e. $\psi$ in trivial on $F^{\times}$) of conductor dividing $\Mfrak$ and such that $\psi_{\infty}(-1)={(-1)}^{\sum_{\sigma \in I} k_{\sigma}}$. Let $\psi_{\Mfrak}= \prod_{v|\Mfrak} \psi_{v}$. We define $\boldM_{k}(\G^1[\Nfrak,\Mfrak],\psi,\C)$ (resp.  $\boldS_{k}(\G^1[\Nfrak,\Mfrak],\psi,\C)$ ) as the set of all holomorphic functions $f$ (resp. cuspidal) such that
$$ f|_{k} \gamma = \psi_{\Mfrak\Nfrak}(a_{\gamma})f \mbox{ for all } \gamma \in \G^1[\Nfrak,\Mfrak].  $$
This definition implies that the corresponding adelic form $\boldf$ satisfies $\boldf(x \gamma) = \psi(d_{\gamma})\boldf(x)$ for $\gamma \in D(\Nfrak,\Mfrak)$.\\
To have a non-zero Hilbert modular form on $D(\Nfrak,\rif)$, we need to suppose that $4\mid \Nfrak$. So we will define $\boldM_{k}(\Nfrak,\psi,\C)=\boldM_{k}(\G^1[2^{-1} \Nfrak, 2],\psi,\C) $  and similarly for $\boldS_{k}(\Nfrak,\psi,\C)$\\
We can define a Fourier expansion also for half-integral weight. 
\begin{prop}\label{FExp}
Let $k \in \Z[I]$ and let $\boldf$ be a Hilbert modular form in $\boldM_{k'}(\Nfrak,\psi,\C)$, for $ k' = \frac{t}{2} +k $, or in $\boldM_{k,-k/2}(U_0(2^{-1}\Nfrak,2),\psi,\C)$, then we have (ignore $r_P$ if the weight is integral) 
\begin{align*}
\boldf \left( r_P \left( \begin{array}{cc} y  & x  \\ 0 & y^{-1} \end{array}\right)  \right) = & \psi_{\infty}^{-1}(y)y_{\infty}^{k}|y|_{\Aa}^{[k'-k]} \sum_{\xi \in F} \lambda(\xi,y\rif;\boldf,\psi)\bexp_{\infty}(i\xi y^2_{\infty}/2)\bexp_F(yx\xi /2).
\end{align*}
The following properties hold
\begin{itemize}
	\item $\lambda(\xi,\mathfrak{m};\boldf,\psi) \neq 0$ only if $\xi \in \mathfrak{m}^{-2}$ and $\xi=0$   or $\xi \gg 0$,
	\item $\lambda(\xi b^2,\mathfrak{m};\boldf,\psi)=b^{k}\psi_{\infty}(b)\lambda(\xi,b\mathfrak{m};\boldf,\psi)$ for all $b \in F^{\times}.$
\end{itemize}
In particular 
\begin{align*}
 f(z)=& \sum_{\xi \in F} \lambda(\xi,\rif;\boldf,\psi)\bexp_{\infty}(\xi z/2)
\end{align*}

We can compare with the previous Fourier expansion when $\boldf$ has integral weight and we have 
\begin{align}\label{ShiHida}
\lambda(\xi,y\rif;\boldf,\psi) = & \psi_{\infty}^{-1}(y)|y|_{\Aa}^{2+m}\bolda(y^2 \xi,\boldf) {\left\{ y^2 \xi\right\}}^v {( y^2_{\infty} \xi)}^{-v}.
\end{align}
\end{prop}

\begin{proof}
The Fourier expansion and the first properties are contained in \cite[Proposition 2.1]{Im}. For the last formula, we point out that $\boldf$ can be written as $\boldf'| {\left( \begin{array}{cc} 2^{-1}  & 0  \\ 0 & 1 \end{array}\right)}_f$ with $\boldf'$ a Hilbert modular form for ${V}_1(\Nfrak)$.  Then we note that
$$
\left( \begin{array}{cc} y  & x  \\ 0 & y^{-1} \end{array}\right) = \left( \begin{array}{cc} y^2  & xy  \\ 0 & 1 \end{array}\right)\left( \begin{array}{cc} y^{-1}  & 0  \\ 0 & y^{-1} \end{array}\right)
$$
and, if the central character of $\boldf$ is $\psi|\phantom{e}|^{-m}_{\Aa}$, we obtain the formula by comparison with \ref{FourierHida}.
\end{proof}

The interest of this adelic Fourier expansion is that it can give rise to the Fourier expansions at all cusps.\\
Let $\Nfrak$ be an integral ideal such that $4|\Nfrak$ and $\nfrak$ a finite id\`ele representing $\Nfrak$: $\nfrak \rif = \Nfrak$ and $\nfrak$ is $1$ at the place outside $\Nfrak$. We shall define now two operators $[{\nfrak}^2]$ and $\tau({\nfrak}^2)$ which will be useful to simplify the calculations of Eisenstein series and $p$-adic $L$ functions we shall perform later. The latter operator will be used to define a Hecke equivariant pairing on integral weight modular forms. Let $k' = \frac{t}{2} +k $. 
We set 
\begin{align*}
\boldf|[{\nfrak}^2](x)=  {\N(\nfrak)}^{-\frac{1}{2}}\boldf\left(x r_P {\left( \begin{array}{cc} \nfrak^{-1}  & 0  \\ 0 & {\nfrak} \end{array} \right)}_{f}  \right) . 
\end{align*}
 This operator sends a form of level group $D(\Cfrak,\Mfrak)$ into one of level group $D(\Cfrak\Nfrak^2,\Nfrak^{-2}\Mfrak)$ and its the analogue of \cite[\S 2 h3]{H6}.
It is easy to see how the Fourier expansion changes;
\begin{align}\label{coeff [p]}
\lambda(\xi,y\rif;\boldf|[\nfrak^2],\psi) = & \lambda(\xi,\nfrak^{-1} y\rif;\boldf,\psi).
\end{align}
In particular, we have from \cite[Proposition 2.1 (i)]{Im} that a necessary condition for $\lambda(\xi,\rif;\boldf|[\nfrak^2],\psi) \neq 0$ is that  $\xi \in \nfrak^2$ and $\xi \gg 0$. We write $f$ resp. $f'$  for the complex version of $\boldf$ resp. $\boldf|[\nfrak^2]$; we easily see that the Fourier expansion at infinity is
$$f'(z)=   \sum_{ 0 \ll \xi \in F^{\times}, \xi \in \nfrak^2} \lambda(\xi,\nfrak^{-1} \rif;\boldf,\psi)\bexp_{\infty}(\xi z/2). $$
If moreover $\Nfrak = (b)$, for a totally positive element $b$, then $f'(z)={b}^{k}f({b}^2 z)$.\\
% No pub 
%Moreover,  let $c$ in $F^{\times}$ and  such that $c\nfrak $ is integral, 
%\begin{align*}
%\left( \begin{array}{cc} 
%\nfrak & 0 \\
% 0 & \nfrak^{-1}
%\end{array} \right)_f= \left( \begin{array}{cc} 
%a & b \\
% c & d
%\end{array} \right)_f\left( \begin{array}{cc} 
%d  \nfrak   & -\frac{b}{\nfrak}\\
%- c \nfrak & \frac{a}{\nfrak}
%\end{array} \right)_f;
%\end{align*}
%we have then $f'(z)=f(z)|_{k}\left( \begin{array}{cc} 
%a & b \\
% c & d
%\end{array} \right) $. % No pub 
{\bfseries  Important remark}: in the next section we shall often use a similar operator $\left[ \frac{\nfrak^2}{4}\right]$. This operator is defined exactly as above with the difference that $4$ does not represent the ideal $4 \rif$, but 4 is {\it  diagonally embedded} in $F_f$.\\

We define the half-integral weight Atkin-Lehner involution
\begin{align*}
 \boldf |\tau({\nfrak}^2)= &\boldf |\eta_2| r_P {\left( \begin{array}{cc} \nfrak^{-1}  & 0  \\ 0 & {\nfrak} \end{array} \right)}_f \\
 = & \boldf | r_{\Omega} {\left( \begin{array}{cc}  0 & -2\dfrak^{-1}{\nfrak}^{-1} \\ 2^{-1} \dfrak\nfrak  & 0 \end{array} \right)}_{f}.
\end{align*}
The matrix $\eta_2$ is $\eta_0$ conjugated with the matrix ${\left( \begin{array}{cc}  2 & 0 \\ 0   &  1\end{array} \right)}_{f}$ (embedded diagonally).\\
This second operator sends $\boldM_{k'}(\Nfrak^2,\psi,\C)$ to $\boldM_{k'}(\Nfrak^2,\psi^{-1},\C)$, since the corresponding matrix normalizes $D(2^{-1}\Nfrak^2,2)$. Note that $\tau({\nfrak}^2)$ is almost an involution: $$\boldf|\tau(\nfrak^2)\tau(\nfrak^2)= \boldf| \left( \begin{array}{cc} -1  & 0  \\ 0 & {-1} \end{array}\right) = (-1)^{k}\boldf .$$ % No pub
%If $\nfrak$ can be chosen as  $\xi $ in $F$, then it acts the complex form $f$ as  
%$$f|\tau(\xi^2) (z)= \xi^{k'} (f \mid \eta_1)(\xi^0 z),$$ where $\eta_1$ is defined in \cite[page 984]{Im} (where it is called $\eta_0$) and it is any matrix such that its inverse belongs to   $\Sl(F) \cap Z(\fa)\eta_0$ (where $\eta_0$ here is given by our definition).\\

%%%%%%%%%%%%%%%%%%%%%%%%%%%%%%%%%%
\subsection{Theta Series}
We give now a first example of half-integral weight modular forms which will be used for the integral formulation of the symmetric square $L$-function. We follow \cite[\S 4]{ShH3}.  Recall that we have a unitary action of $M_{\fa}$ on the space $L^2(\fa)$. For a Schwartz function on the finite ad\`eles $\eta$ and $n \in \Z[I]$, we define a function for  $\tau \in M_{\fa}$
$$ \theta_n(\tau,\eta) = \sum_{\xi \in F} \tau \eta_n (\xi,i) $$
where we define for $\xi$ in $\fa$ and $w$ in $\h^I$ $$\eta_n(\xi,w)= \eta(\xi_{f})\prod_{\sigma \in I} \phi_{n_{\sigma}}(\xi_{\sigma},w_{\sigma}),$$
$\phi_{n_{\sigma}}(\xi_{\sigma},w_{\sigma})= y_{n_{\sigma}}^{-{n_{\sigma}}/2}H_{n_{\sigma}}(\sqrt{4 \pi y_{n_{\sigma}} } \xi_{\sigma})\bexp(\xi_{\sigma}^2 w_{\sigma}/2)$, where $H_n$ is the $n$-th  Hermite polynomial.
It is the adelic counterpart of the complex form
$$ \theta_n(z,\eta) = {(4\pi y)}^{-n/2}\sum_{\xi \in F} \eta (\xi,i)H_n(\sqrt{4\pi y} \xi)\bexp_{\infty}(\xi^2 z /2) $$

Choose a Hecke character $\chi$ of $F$ of conductor $\Nfrak$ such that $\chi_{\infty}(-1)={(-1)}^{\sum_{I} n_{\sigma}}$, we denote also by $\chi$ the corresponding  character on the group of fractional ideals.
Recall that we defined $t= \sum_{\sigma \in I}\sigma$; for $n=0,t$ we have
\begin{align*}
 \theta_0(\chi) & = \sum_{\xi \in F} \chi_{\infty}(\xi)\chi(\xi \rif) \bexp_{\infty}(\xi^2 z/2) \\
 \theta_t(\chi) & = \sum_{\xi \in F} \chi_{\infty}(\xi)\chi(\xi \rif) \xi^t \bexp_{\infty}(\xi^2 z/2). 
\end{align*}
From now on $\theta_n(\chi):=\theta_{nt}(\chi)$; $\theta_n(\chi)$ is holomorphic. Moreover, from \cite[Lemma 4.3]{ShH3}, we have that for all $\gamma$ in $\G^1[2\Nfrak^2,2]$ we have
$$ \theta_n(\chi)|_{(n+1/2) t} \gamma = \chi_{4\Nfrak^2}(a_{\gamma})\theta_n(\chi).$$
This tells us that $\theta_n(\chi)$ is a modular form in $\boldM_{n+\frac{t}{2}}(4\Nfrak^2,\chi_{4\Nfrak^2}^{-1},\C)$. The explicit coefficients, given by \cite[(4.21)]{ShH3}, are as follows:  
$$ \lambda(\xi,\mathfrak{m},\theta_n(\chi),\chi) = \left\{ 
\begin{array}{cc}
2 \chi_{\infty}(\eta)\chi(\eta \mathfrak{m}) \xi^n & \mbox{ if } 0 \neq \xi = \eta^2 \in \mathfrak{m}^{-2},\\
\chi(\mathfrak{m}) & \mbox{ if } 0 = \xi, \mathfrak{f}=\rif \mbox{ and } n =0 ,\\
0 & \mbox{ otherwise}.
\end{array} \right. $$
For a Hecke character $\chi$, we define, following \cite[(4.8)]{ShH2}, the Gau\ss{}  sum of $\chi$
\begin{align*} 
G(\chi) =& \sum_{x \in \ffrak^{-1}\dfrak^{-1}/\dfrak^{-1}} \chi_{\infty}(x)\chi_{f}(x\dfrak\ffrak) \bexp_F(x).
\end{align*}
We conclude with the following proposition which generalizes a well-known result for $F=\Q$.
\begin{prop}\label{Theta-Atkin} Let $\chi$ be a Hecke character of conductor $\Nfrak$. Then 
\begin{align*} \theta(\tau,\chi)|\tau(4\nfrak^2) =C(\chi) \theta(\tau,\chi^{-1}) \end{align*}
for $C(\chi)= G(\chi){\N(\Nfrak)}^{-1/2}\chi(\dfrak\Nfrak)$.
\end{prop}
\begin{proof}
We decompose $$ \left( \begin{array}{cc} 0  &  - \dfrak^{-1}\nfrak^{-1} \\ {\nfrak}\dfrak & 0  \end{array} \right) =  \left( \begin{array}{cc} \dfrak^{-1}\nfrak^{-1}  & 0  \\ 0 & {\nfrak}\dfrak \end{array} \right) \left( \begin{array}{cc} 0 & -1  \\ 1 & 0 \end{array} \right). $$
As in \cite[Lemma 4.2]{ShH3} we have 
$$\left( \begin{array}{cc} 0 & -1  \\ 1 & 0 \end{array} \right) \chi_f (x) = {\N(\dfrak)}^{-1/2}{\N(\nfrak)}^{-1}G(\chi) \chi(\dfrak\nfrak)\chi_f^{-1}(\dfrak\nfrak x).$$
Now, we are left to study the action of the diagonal matrix on $\chi$. As this matrix is in $P_{\fa}$, we can use \cite[(1.12)]{ShH3} to reduce ourself to local calculations
\begin{align*}
 r_P \left( \begin{array}{cc} \dfrak^{-1}\nfrak^{-1}  & 0  \\ 0 & {\nfrak}\dfrak \end{array} \right) \chi_f = & \prod_v r_P {\left( \begin{array}{cc} \dfrak^{-1}\nfrak^{-1}  & 0  \\ 0 & {\nfrak}\dfrak \end{array} \right)}_v \chi_v.
\end{align*}
With \cite[(1.9)]{ShH3} we see that the action of this diagonal matrix sends $\chi_v(x)$ to ${|\dfrak\nfrak|_v}^{-1/2}\chi_v(\dfrak_v^{-1}\nfrak_v^{-1} x)$, then 
$$\tau(4 \nfrak^2) \chi_{f} (x) =G(\chi){\N(\nfrak)}^{-1/2}\chi(\dfrak\Nfrak) \chi_{f}^{-1} (x) . $$
\end{proof}
%%%%%%% 

\subsection{Eisenstein series}\label{EisSer}

We give a second example of half-integral weight forms which we shall use in the next sections. The aim of this section is to find a normalization of these series which shall allow us to $p$-adically interpolated its Fourier coefficients \ref{Eisen_other2}.\\
Let us fix an integral weight $k \in \Z[I]$, a congruence subgroup of level $\cfrak$, a Hecke character $\psi$ of finite conductor dividing $\cfrak$ and of infinite part of type ${(x/|x|)}^k$. Decompose 
\begin{align*}
G \cap P_{\fa}D(2^{-1}\Nfrak,2)= \coprod_{\afrak_i \in \mathrm{Cl}(F)} P \beta_i \G^1[2^{-1}\Nfrak,2],
\end{align*}
where $P$ is the Borel of $G$ and $\beta_i = \left( \begin{array}{cc} t_i & 0  \\ 0 & 1 \end{array} \right)$, for $t_i$ an id\`ele representing $\afrak_i$.\\
We pose 
\begin{align*}
E(z,s; k,\psi,\cfrak) =  & \sum_{\afrak_i \in \mathrm{Cl}(F)}  {N(\afrak_i)}^{2s+ \frac{1}{2}} \sum_{\gamma \in P_{i} \setminus \beta_i \G^1[2^{-1}\Nfrak,2]}   \psi(d_{\gamma}\afrak_i^{-1})\psi_{\infty}(d_\gamma)\frac{y^{st-k/2}{h(\gamma,z)}^{-1}{j(\gamma,z)}^{-k}}{ {|j(\gamma,z)|}^{2st-k} },
\end{align*}
where $P_i = P \cap \beta_i \G^1[2^{-1}\Nfrak,2] \beta_1^{-1}$. We follow mainly the notation of \cite[\S 1]{Im}, but we point out that our Eisenstein series $E(z,s; k,\psi,\cfrak) $ coincides with Im's series $E$ evaluated at $s +1/4$ and $\afrak = 2^{-1}\rif $, and we refer to {\it loc. cit.} for the statements without proof which will follow. \\

Take an element $\eta_1$  in $\Sl(F) $ such that  $\eta_1^{-1} \in Z(\fa)\eta_0$ and  an element $\tilde{\eta}$ in $M_{\fa}$ such that $pr(\tilde{\eta})=\eta_0$; we can define
\begin{align*}
E'(z,s; k,\psi, \cfrak) = & E(z,s; k,\psi, \cfrak)|_{k+\frac{1}{2}} \eta_1.
\end{align*}
We can see $E(z,s; k,\psi, \cfrak)$ as an adelic form, as in Section \ref{HIntweight}: for all non-negative integers $\rho$ we define 
$$ J(\tau,z) = {h(\tau,z)}^{2\rho+1}{(j(\tau,z)/|j(\tau,z)|)}^{k} $$ and a function $f$ on $M_{\Aa}$ as 
$$f(\tau) = \psi_f(d_p)\psi_{\cfrak}(d_w^{-1})|J(\tau, z_0)|J(\tau,z_0)^{-1} $$ if $\tau=pw$ is in $P_{\fa}D(2^{-1}\cfrak \dfrak, 2 \dfrak^{-1})$ and 0 otherwise. We define moreover, for  $\tau$ in $P_{\fa} C$, $\tau=p w'$, $\delta_{\tau}=|d_p|_{\Aa}$ and for $\tau \in pr^{-1} P_{\fa} C$, $\delta_{\tau}=\delta_{pr (\tau)}$.\\
The adelic Eisenstein series is defined by 
$$ E_{\Aa}(\xi,s; (2\rho+1)/2, k ,\psi,\cfrak) = \sum_{g \in P \setminus SL_2(F)} f(g\xi){\delta_{g\xi}}^{-2s -\rho -1/2}. $$
We have therefore $E(z,s; k,\psi, \cfrak) =E_{\Aa}(u,s; 1/2, k ,\psi,\cfrak)J(u,z_0) $, where $u$ is  such that $z=u(z_0)$.
Moreover, $$E'(z,s; k,\psi, \cfrak) = E_{\Aa}(u \tilde{\eta},s; \frac{1}{2}t, k ,\psi,\cfrak)J(u,z_0) . $$
For $k=\kappa t$, $\kappa>0$ and $s=\kappa/2$ we have:
\begin{align*}
 E(z,\kappa/2; \kappa t,\psi, \cfrak) \in & \boldM_{k+\frac{1}{2}}(\cfrak; \psi_{\cfrak}^{-1}),\\
 E'(z,\kappa/2; \kappa t,\psi, \cfrak) \in & \boldM_{k+\frac{1}{2}}(2, 2^{-1}\cfrak; \psi_{\cfrak}^{-1}).
\end{align*}
% Hida piazza un (livello)^{s+1/4} di fronte a E' , Hida Series correspond to E^G(k,s)=E_{k+1/2}^H(2s-k)
We define a normalization of the previous Eisenstein series 
\begin{align*}
\calE ' \left(z,s;k,\psi, \cfrak \right) = &  L_{\cfrak}(4s,\psi^2) E'\left(z,s;k,\psi, \cfrak \right), 
\end{align*}
where $L_{\cfrak}(s,\psi)$ stands for $L(s,\psi)\prod_{\qfrak \mid \cfrak}(1-\psi(\qfrak){\N(\qfrak)}^{-s})$. We define $\calE \left(z,s;k,\psi, \cfrak \right)$ analogously. 
Let $A =  \psi(\dfrak){D_F}^{\kappa-1}\N(2\cfrak^{-1})i^{\kappa d}\pi^d 2^{d\left(\kappa - \frac{1}{2}\right)} $. We consider now the case when $s$ is negative; we have the following result on the algebraicity of these series

\begin{theo}[\cite{Im}, Proposition 1.5]\label{EisFou}
The series $A^{-1}\calE '(z,\frac{1-\kappa}{2};\kappa t,\psi, \cfrak )$ has the following Fourier development
$$
L_{\cfrak}(1-2\kappa) + \sum_{0 \ll \xi \in 2\cfrak^{-1} } L_{\cfrak}(1-\kappa,\psi\omega_{\xi})\beta\left(\xi, \frac{1-\kappa}{2}\right) \bexp_{\infty}(\xi x),
$$
where $\omega_{\xi}$ is the Hecke character associated with $F(\sqrt{\xi})$ and 
\begin{align*}
\beta(\xi,s)= &\sum_{\everymath{\scriptstyle}\tiny{\begin{array}{c}
                                       \afrak^2\mathfrak{b}^2 |\cfrak \xi\\
                                       \afrak, \mathfrak{b} \mbox{ prime to } \cfrak
                                       \end{array}}}
\mu(\afrak) \psi \omega_{\xi}(\afrak){\psi(\mathfrak{b})}^2 {N(\mathfrak{b})}^{-2s}{N(\mathfrak{a})}^{1 -4s}.
\end{align*}
\end{theo}
The fact that the Fourier expansion is shifted by $\cfrak$ is due to the fact that we are working with forms for the congruence subgroup $D(2, 2^{-1}\cfrak)$.  Similarly, if $\cfrak$ is the square of a principal  ideal in $\rif$ and $\cfrak_0$ an id\`ele representing this ideal, we can deduce the Fourier expansion of ${\N(\cfrak_0 2^{-1})}^{-\kappa}A^{-1}\calE' (z,\frac{1-\kappa}{2};\kappa t,\psi, \cfrak )[\cfrak^2_0 4^{-1} ]$ simply by taking the same sum but over all $0 \ll \xi \in 2^{-1}\rif$.\\

We introduce now the notion of nearly-holomorphic modular form. It has been extensively studied by Shimura. Let $k$ be in $\Z\left[ \frac{1}{2} \right][I]$, and $v $ in $\Z[I]$ when $k$ is integral. Let $\G$ be a congruence subgroup of $\gl(F_{\infty})$ or $\Sl(F_{\infty})$, according to the fact that the weight is integral or half-integral, and $\psi$ a finite order character of this congruence subgroup. Let $\beta \in \Z[I]$, $\beta \geq 0$, we define the space of nearly-holomorphic forms of degree at most $\beta$, $\boldN_{k,v}^{\beta}(\G,\psi;\C)$ as the set of all functions 
\begin{align*} 
 f : \h^I \rightarrow \C
\end{align*}
such that $f|_{k,v}\gamma (z) = \psi(\gamma)f(z) $ for all $\gamma$ in $\G$ and $f$ admits a generalized Fourier expansion 
\begin{align*}
 f = \sum_{i=0}^{\beta} \sum_n^{\infty} a_n^{(i)} q^n \frac{1}{{(4 \pi y)}^i}
\end{align*}
at all cups for $\G$.\\
In case of integral weight, it is possible to define adelic nearly-holomorphic forms for compact open subgroups of $\gl(\faf)$ as in \cite[\S 1]{Hr3}; to each such form we can associate form an $h$-tuple of complex forms as we defined.\\
We warn the reader that the function $g_i(z)= \sum_n^{\infty} a_n^{(i)} q^n$ are not modular forms for $\G$, except when $i=\beta$.\\
Let now be $l$ in $\Z[I]$, $\lambda$ in $\mathbb{R}$ and $\sigma$ in $I$, we define a (non-holomorphic) differential operator on $\h^I$ by 
\begin{align*}
 \partial_{\lambda}^{\sigma} = & \frac{1}{2 \pi i}\left (\frac{\lambda}{2 i y_{\sigma}} + \frac{\textup{d} } {\textup{d} z_{\sigma}}\right ), \\
 \partial_{k}^{l} = & \prod_{\sigma \in I}\left(\partial_{k_{\sigma}+2 l_{\sigma} -2 }^{\sigma} \cdots \partial_{k_{\sigma}}^{\sigma} \right),
\end{align*}
and here $z_{\sigma} = x_{\sigma} + i y_{\sigma}$ is the variable on the copy of $\h$ indexed by $\sigma$.\\
Let $f$ be in $\boldN_{k,v}^s(\G,\psi;\C)$, if $k > 2s$ (when $F\neq \Q$ this condition is automatically satisfied)  then we have the following decomposition 
\begin{align*}
 f = \sum_{i=0}^{s} \partial_{k-2i}^{i}f_i  
\end{align*}
with each $f_i$ holomorphic and modular. A proof of this, when $F=\Q$, is given in \cite[Lemma 7]{Sh5} and for general $F$ the proof is the same.\\ 
For such a form $f$ we define the {\it holomorphic projection} $H(f):=f_0$. For any $f$ in $\boldN_{k,v}^s(\G,\psi;\C)$ we define the {\it constant term projection} $c(f)=g_0$. If we see $f$ as a polynomial in $\frac{1}{{(4 \pi y)}^i}$, $c(f)$ is then the constant term. We will see when the weight is integral that $c(f)$ is a $p$-adic modular form. This two operators will be used to calculate the values of the $p$-adic $L$-function in Section \ref{padicL}.\\
We have the relation 
\begin{align}\label{Massgamma}
\partial_{k}^{l}(f |_{k,v}\gamma) = &(\partial_{k}^{l}f )|_{k+2l,v-l}\gamma 
\end{align} 
which tells us that $\partial_{k}^{l}$ sends nearly-holomorphic modular forms for $\G$ of weight $(k,v)$ to nearly-holomorphic modular forms for $\G$ of weight $(k+2r,v-r)$.\\
This operator allows us to find some relation between Eisenstein series of different weights.  Let $k$, $l$ be in $\Z[I]$, $k,l >0$. Then we have by the proof of  \cite[Proposition 1.1]{Im}
\begin{align}\label{deltaEis} 
\delta_{k + \frac{1}{2}t}^{l} E(z,s; k,\psi, \cfrak) = & {(4\pi)}^{-l}b_l\left( -\frac{k}{2} - \left(s + \frac{1}{2} \right)t \right) E(z,s; k + 2l,\psi, \cfrak) 
\end{align}
 where 
$$  b_l(x)= \prod_{\sigma \in I}\prod_{j=0}^{l_{\sigma} -1} (x_{\sigma} - j) \mbox{ for } x=(x_{\sigma}) \in \C^{I}. $$
If $k=\kappa t$ and $2s = 1 -\kappa$ then we obtain  $b_l(-t)={(-1)}^{\sum l_v}\prod_{v \in \bolda } (l_v)! $.\\
In particular, in the case we will need later, this coefficient is not zero. Let $k= m +2t - 2v \geq 0$ be an integral weight, let $k_0$ be the minimum of the $k_{\sigma}$'s  and $s$ an integer such that $m-k_0 -1  < s  $, we have that $b_v\left( \frac{2v - st -2t }{2} \right) \neq 0 $.\\
For $x$ in $\mathbb{R}^{I}$, we pose $\G_{\infty}(x)= \prod_{\sigma \in I} \G(x_{\sigma})$. \\
Let $k= m +2t - 2v \geq 0$ be an integral weight, let $k_0$ be the minimum of the $k_{\sigma}$'s  and $s$ an integer such that $m+1 \leq s  \leq m + k_0 -1 $, and $n \in \lgr 0,1 \rgr$, $n \not\equiv s (2)$ and $\psi(-1)=(-1)^{k-d-nd}$.\\
In the notation of \cite[Lemma 4.2]{ShH4} we pose, for $\sigma$ in $I$, $\alpha_{\sigma} = \frac{s-m -n +k_{\sigma}}{2}$, $\beta_{\sigma} = \frac{s-m -k_{\sigma} +n+1}{2}$.
\begin{prop}\label{Eisen_other}
  For $k$, $s$, $m$, $n$ and $\psi$ as above we have 
\begin{align*}
\calE '\left(z,\frac{s-m}{2}; k-(n+1)t,\psi, \cfrak \right)  = & A'_0 \sum_{0 \leq j \leq - \beta} e_j(k,s,\psi)\frac{\G_{\infty}(1-\alpha -j)/ \G_{\infty}(1 -\alpha) }{{(4 \pi y)}^{j}} 
\end{align*}
where  
\begin{align*}
  A_0' = &  i^{(n+1/2)d - k}{\pi}^{\alpha} {\G_{\infty}(\alpha)}^{-1} {2}^{k - \left(n+\frac{1}{2}\right)d } {(-1)}^{-d(n+1) + \sum_{\sigma} (k_{\sigma})} \times \\ 
      & \times \psi(\dfrak){\N(\dfrak)}^{m-s}\N(2\cfrak^{-1})\bexp_{\infty}([F:\Q]/8),\\
e_j(k,s,\psi) =&  \sum_{\xi \in 2\cfrak^{-1}, \xi \gg 0} \xi^{ -\beta -j} g_{f}(\xi,(s-m)/2)L_{\cfrak}(s-m,\psi\omega_{2\xi})\bexp_{\infty}(\xi z). 
  \end{align*} 
Here $g_{f}(\xi,s)$ is a product over the primes $\qfrak$ dividing $\xi\cfrak$ but prime with $\cfrak$ of polynomials in $\psi(\qfrak)\N(\qfrak)^s$. 
\end{prop}
\begin{proof}
 We just have to explicitate the Fourier coefficients denoted by $y^{\beta}\mathbf{\xi}(y,\xi;\alpha,\beta)$ in  \cite[Theorem 6.1]{ShH4}. They are essentially hypergeometric functions.\\ 
Notice that our definition is different from his; in fact we have 
\begin{align*}
 E'\left(z,s; k-(n+1)t,\psi, \cfrak\right) & =  y^{k/2}E_{\mathrm{Sh}}'(z,s;1/2,k,\psi).
\end{align*}
Note that in our situation the $\alpha_{\sigma}$'s are positive half-integers and the $\beta_{\sigma}$'s are non-positive integers, for all $\sigma$. \\
Using \cite[(6.9 b)]{Hr3}, the evaluation of  $y_{\sigma}^{\beta_{\sigma}}\mathbf{\xi}(y_{\sigma},\xi_{\sigma};\alpha_{\sigma},\beta_{\sigma})$ reduces to the following cases: if $\xi_{\sigma} >0$, we find
$$  i^{\beta_{\sigma}-\alpha_{\sigma}}{(2\pi)}^{\alpha_{\sigma}} {\G(\alpha_{\sigma})}^{-1}{2}^{-\beta_{\sigma}} \xi_{\sigma}^{\alpha_{\sigma} -1 } e^{-4 \pi y \xi_{\sigma}/2} W(4 \pi y\xi_{\sigma},\alpha_{\sigma},\beta_{\sigma})$$ for $W$ the Whittaker function given in \cite[page 359]{Hr3}
$$W(y,\alpha_{\sigma},\beta_{\sigma})= {\G(\beta_{\sigma})}^{-1} {y}^{-\beta_{\sigma}} \int_{\mathbb{R}^+}e^{-yx} {(x+1)}^{\alpha_{\sigma}-1}{x}^{\beta_{\sigma}-1}. $$ 
If $\xi_{\sigma}= 0 $ we find:
$$ i^{\beta_{\sigma}-\alpha_{\sigma}}{\pi}^{\alpha_{\sigma}}2^{\alpha_{\sigma} +\beta_{\sigma}} \frac{\G(\beta_{\sigma}+\alpha_{\sigma}-1)}{\G(\alpha_{\sigma})\G(\beta_{\sigma})}{(4\pi y)}^{1-\alpha_{\sigma}}$$
If $\xi_{\sigma} <0$ we find : 
$$ i^{\beta_{\sigma}-\alpha_{\sigma}}{(2\pi)}^{\beta_{\sigma}}  {\G(\beta_{\sigma})}^{-1}{2}^{-\alpha_{\sigma}} y^{\beta_{\sigma}-\alpha_{\sigma}} {|\xi_{\sigma}|}^{\alpha_{\sigma} -1} e^{-4 \pi y \xi_{\sigma}/2}W(- 4 \pi y \xi_{\sigma},\beta_{\sigma},\alpha_{\sigma}) .$$ 
Notice that $W(y,1,\beta_{\sigma})= 1$.\\
If $\xi$ is not zero and not a totally positive element, then for at least one $\sigma$ the corresponding hypergeometric function is 0 (because of the zero of ${\G(\beta_{\sigma})}^{-1}$). \\
So we are reduced to evaluate these $W$'s only for $\xi_{\sigma} \geq 0$. By \cite[Theorem 3.1]{Sh6} we have $W(y,\alpha_v,\beta_v) =W(y,1-\beta_v,1 -\alpha_v)$. 
As $1 -\beta_{\sigma} \geq 1$, we can use \cite[(6.5)]{Hr3} to obtain
\begin{align*}
y^{\beta}\xi(y,\xi;\alpha,\beta) = & i^{\beta-\alpha}{\pi}^{\alpha} {\G(\alpha)}^{-1}{2}^{\alpha-\beta}e^{\frac{-4 \pi y\xi}{2}} \sum_{j=0}^{-\beta} { -\beta \choose  j} \frac{\G_{\infty}(1 - \alpha +j)}{\G_{\infty}(1-\alpha)} \frac{\xi^{\alpha -1 -j}}{{(4 \pi y)}^{j}}. 
\end{align*}
\end{proof}
We point out that  $\calE'\left(z,\frac{k_0}{2}; k_0 t,\psi, \cfrak\right)$ is an holomorphic Hilbert modular form for all $k_0 \geq 1$.\\
Suppose now $\cfrak = \nfrak^2$, we are interested now in the Fourier expansion of $\calE '(z,\frac{s-m}{2}; k-(n+1)t,\psi, \Nfrak^2)| [\nfrak^2] $. 
Let us define, for all $u$ in $\fa$, $t(u)= r_P \left( \begin{array}{cc} 1 & u \\ 0 & 1 \end{array} \right)$ and recall the function on $M_{\fa}$
\begin{align*} 
 f(\tau) & = \psi_f(d_p)\psi_{\cfrak}(d_w^{-1})|J(\tau, z_0)|J(\tau,z_0)^{-1} \:\:\:\:\;\; (\tau=pw)
\end{align*}
 defined above. Let $\iota:= r_{\Omega} \left( \begin{array}{cc} 0 & -1 \\ 1 & 0 \end{array} \right)_f$. Let $\tilde{z} = r_P \left( \begin{array}{cc} y^{1/2} & x \\ 0 & y^{-1/2} \end{array} \right)$ for $y=(y_{\sigma}) \in \R_+^{I}$ and $x=(x_{\sigma}) \in \R^{I}$. We give now a useful lemma 
\begin{lemma}\label{fdelta}
For $u$ in $F_f$, we have $f(\iota t(u) \tilde{z} \tau(\nfrak^2/4)  ) \neq 0$ if and only if $ u \in 2 \dfrak^{-1} \rif$, and then 
\begin{align*}
f(\gamma \tau(\nfrak^2/4)  ) |\delta_{ \gamma \tau(\nfrak^2/4)}|^{-2s -  \frac{1}{2}}_{\Aa} = 
\psi_f(2^{-1} \nfrak){\N(\nfrak 2^{-1})}^{-2s - \frac{1}{2}} f(\gamma' \eta_0  ) |\delta_{\gamma' \eta_0 }|^{-2s - \frac{1}{2}}_{\Aa} 
\end{align*}
for $u_v' = {(u \nfrak^{2}4^{-1})}_v  $, $\gamma =\iota t(u)\tilde{z}$ and $\gamma' =\iota t(u')\tilde{z}$.
\end{lemma}
\begin{proof}
 For all finite place $v$, we have 
\begin{align*}
 (\iota t(u)  \tau(\nfrak^2/4))_v & = \left( \begin{array}{cc} -(2^{-1} \dfrak\nfrak)_v  & 0 \\ (u 2^{-1} \dfrak\nfrak)_v  & -(2 \dfrak^{-1}\nfrak^{-1})_v \end{array} \right).
\end{align*}
We have by definition that $f(\xi)\neq 0 $ if and only if $\xi \in P_{\fa}D(2^{-1} \Nfrak^2 \dfrak, 2 )$. Let us write then 
\begin{align*}
 \iota t(u) z \tau(\nfrak^2/4) = &
 r_P \left( \begin{array}{cc} \ast & \ast \\ 0 & e \end{array} \right) r_{\Omega}\left( \begin{array}{cc} * & * \\ c & d \end{array} \right).
\end{align*}
 We know that $d$ can be chosen as a $v$-adic unit for all $v \mid \dfrak\Nfrak$; combining this with the explicit expression of $\iota t(u) z \tau(\nfrak^2/4) $, we obtain $u \in 2 \dfrak^{-1} \rif  $ and this proves the first part of the lemma.\\
Concerning the second part of the lemma, the right hand side has been calculated in \cite[Lemma 4.2]{ShH4}. 
We have (compare this formula with \cite[(4.19)]{ShH4}) 
\begin{align*} 
\delta_{(\iota t(u) \tau(\nfrak^2/4))_v}= |u \dfrak \nfrak 2^{-1}|_{v} = \N( \nfrak / 2)|u' \dfrak |_{v}  
\end{align*}
and $h(\tau(\nfrak^2/4),z)= j(\tau(\nfrak^2/4),z)=1$. If we write \begin{align*} \iota t(u')  \eta_0& =
 r_P \left( \begin{array}{cc} \ast & \ast \\ 0 & e' \end{array} \right) r_{\Omega}\left( \begin{array}{cc} * & * \\ c' & d' \end{array} \right),\end{align*}
 we see we can choose $e=2\nfrak^{-1} e'$ and $d=d'$, and this allows us to conclude.
\end{proof}

We can now state the following proposition:
\begin{prop}\label{Eisen_other2}
 
Let $k$, $s$, $m$, $n$ and $\psi$ as in Proposition \ref{Eisen_other}, we have 
\begin{align*}
\calE' \left(z,\frac{s-m}{2}; k-(n+1)t,\psi, \nfrak^2\right)|\left[\frac{{\nfrak}^2}{4}\right]  = & A_0 \sum_{0 \leq j \leq - \beta} e'_j(k,s,\psi)\frac{\G_{\infty}(1-\alpha -j)/ \G_{\infty}(1 -\alpha) }{{(4 \pi y)}^{j}} 
\end{align*}
where  
\begin{align*}
  A_0 = &  i^{\left(n+\frac{1}{2}\right)d - k}{\pi}^{\alpha} {\G_{\infty}(\alpha)}^{-1} {2}^{k - \left(n+\frac{3}{2}\right)d } {(-1)}^{-d(n+1) + \sum_{\sigma} (k_{\sigma})} \times \\ 
      & \times \psi(\dfrak  \nfrak 2^{-1}) D_F {\N(\dfrak\nfrak 2^{-1})}^{m-s-1} \bexp_{\infty}([F:\Q]/8),\\
e'_j(k,s,\psi) =& L_{\cfrak}(2(s-m)-1,\psi\omega_{2\xi}) +  \sum_{\xi \in 2^{-1}\rif, \xi \gg 0} \xi^{ -\beta -j} g_{f}(\xi,(s-m)/2)L_{\cfrak}(s-m,\psi\omega_{2\xi})\bexp_{\infty}(\xi z). 
  \end{align*} 
Here $g_{f}(\xi,s)$ is a product over the primes $\qfrak$ dividing $2\xi$ but prime with $2$ of polynomials in $\psi(\qfrak)\N(\qfrak)^s$.\\ 
Moreover, let $\kappa $ be an integer, $\kappa \geq 2$, then 
\begin{align*}
 A_0^{-1} \calE '\left(z,\frac{1-\kappa}{2};\kappa t,\psi, \cfrak \right)|\left[\frac{{\nfrak}^2}{4}\right] = L_{\cfrak}(1-2\kappa) + \sum_{0 \ll \xi \in 2^{-1} } L_{\cfrak}(1-\kappa,\psi\omega_{\xi})\beta\left(\xi, \frac{1-\kappa}{2}\right) \bexp_{\infty}(\xi x),
\end{align*}
for $A_0 = i^{\kappa d}\pi^d 2^{d\left(\kappa -  \frac{1}{2}\right)} \psi_f(\dfrak  \nfrak 2^{-1})D_F{\N(\dfrak\nfrak 2^{-1})}^{\kappa -2}$.

\end{prop}
\begin{proof}
We give a proof only of the first formula as the proofs of the two formulae are identical.\\
The fact that the Fourier expansion is indexed by $\xi \gg 0$ and $\xi$ in $2^{-1} \rif$ can be proved exactly as in \cite[page 430]{Sh7}. Using  the adelic expression, \cite[pag. 300]{ShH4} and Lemma \ref{fdelta}, we see that up to a factor $\psi_f(2 \nfrak^{-1}){\N(\nfrak 2^{-1})}^{m-s-1}$, the sum which gives  $\calE '(z,\frac{s-m}{2}; k-(n+1)t,\psi, \nfrak^2)|[{\nfrak}^2 /4]$ is the same as the one for $\calE '(z,\frac{s-m}{2}; k-(n+1)t,\psi, \cfrak)$ which can be found in \cite[pag. 300]{ShH4}. We can calculate than the Fourier coefficients as in \cite[\S 5]{ShH4}; we choose a measure on $F_v$ which gives volume $1$ to the valuation ring of $F_v$. \\  
Then we have, in the notation \cite[\S 5]{ShH4}, for $\sigma \in 2^{-1}\dfrak^{-1} \rif$,
\begin{align*}
 c_v(\sigma, s) = |2|_{F_v}, \; \; \:\: v \mid 2\nfrak.
\end{align*}
  This implies that we have to substitute the factor $\N(2 \nfrak^{-2})$ which appears in the factor $A'_0$ defined in Proposition \ref{Eisen_other} with $2^{-d}$. The integrals at the other places are unchanged. 
\end{proof}

%%%%%%%%%%%%%%%%%%%%%%%%%%%%%%%%%%%%%%%%%%%%%%%%%%%%%%%%%%%%%%%%%%%%%%%%%%%%%%%%%%%%%%%%%%%%%%%%%%%%%%%%%%%%%%%%%%

\subsection{The $p$-adic theory again}\label{halfp}
We want now define the notion of a $p$-adic modular form of half-integral weight. This is done in order to construct in Section \ref{p-measure} some measures interpolating the forms introduced in the previous section and which will be used to construct the $p$-adic $L$-function in Section \ref{padicL}. In particular, we shall construct a $p$-adic trace \ref{ptrace} to pass from ``$\Sl$''-type $p$-adic modular forms (such as the product of two half-integral weight forms) to ``$\gl$''-type $p$-adic modular forms. This is done interpolating $p$-adically the traces defined in \cite[Proposition 4.1]{Im}.\\ 
We will construct half-integral weight $p$-adic modular forms as in \cite{Wu}; we will define them them via $q$-expansion. Although this approach is hardly generalizable to more general context where no $q$-expansion principle ia available, it will be enough for our purpose.\\
For a geometric approach using the Igusa tower, we refer the reader to \cite{DT}.\\
We fix as above an ideal $\Nfrak$ of $\rif$ such that $4 \mid \Nfrak$ and let $k$ be an half integral weight. 
For all $f$ in $\boldM_{k}(\Nfrak,\psi,\C)$, consider the Fourier expansion from Proposition \ref{FExp} 
\begin{align*}
 f(z)=& \sum_{\xi \in F} \lambda(\xi,\rif;\boldf,\psi)\bexp_{\infty}(\xi z/2).
\end{align*}
This gives us an embedding (because of the uniqueness of the Fourier expansion and the transformation properties of $f$) of $\boldM_{k}(\Nfrak,\psi,\C)$ into $\C[[q]]$.\\ 
For a subalgebra $A$ of $\C$ we define $\boldM_{k}(\Nfrak,\psi,A)$ as $\boldM_{k}(\Nfrak,\psi,\C) \cap A[[q]]$. 
By the $q$-expansion principle \cite[Proposition 8.7]{DT}, if $p$ is unramified, we are indeed considering geometric modular forms of half-integral weight defined over $A$. \\
Let $K$ a finite extension of $\Q_p$ and $\oo$ its ring of integer, choose a number field $K_0$, of the same degree as $K$, dense in $K$ for the $p$-adic topology induced by the fixed embeddings chosen at the beginning and let $\oo_0$ its ring of integers.
We define 
\begin{align*} 
\boldM_{k}(\Nfrak,\psi,\oo) = & \boldM_{k}(\Nfrak,\psi,\oo_0) \otimes \oo, \\
 \boldM_{k}(\Nfrak,\psi,K)  = & \boldM_{k}(\Nfrak,\psi,\oo)\otimes K .
\end{align*}
It is possible, multiplying by a suitable $\theta$ series, to show the independence of the choice of $K_0$ and that these spaces coincide with the completion w.r.t. the $p$-adic topology of $\boldM_{k}(\Nfrak,\psi,\oo_0)$ and $\boldM_{k}(\Nfrak,\psi,K_0)$. The $p$-adic topology is defined as the maximum of the $p$-adic norm of the coefficients in the $q$-expansion.
We define 
\begin{align*}
\overline{\boldM}_{\mathrm{half}}(\Nfrak,\oo) = & \overline{ \sum_{k \in \Z[I] + \frac{1}{2} t}\bigcup_{ r \in \mathbb{N}_{>0},\psi} \boldM_{k}(\Nfrak p^r,\psi,\oo)}, \\ 
\overline{\boldM}_{\mathrm{half}}(\Nfrak,K) = & \overline{ \sum_{k \in \Z[I] + \frac{1}{2} t}\bigcup_{ r \in \mathbb{N}_{>0},\psi} \boldM_{k}(\Nfrak p^r,\psi,K)}.
\end{align*}
If $p$ is unramified in $F$, this is again compatible with the geometric construction \cite[Application page 608]{DT}. From the considerations in \cite[\S 9]{DT}, it could be possible in most cases to prove that the reunion on both $p$-level and weight is superfluous.\\
The same construction allows us to define $\overline{\boldS}_{\mathrm{half}}(\Nfrak,\oo)$ and $\overline{\boldS}_{\mathrm{half}}(\Nfrak,K)$.\\
Let $\chi$ be an Hecke character of conductor $\cfrak$ of finite order. We consider it both as an id\`ele character $\chi=\chi_{f}\chi_{\infty}$ for two character of $\faft$ and $\fait$ such that $\chi$ is trivial on $F^{\times}$ and as a function on $\rif/ \cfrak$ which is the character $\chi_{\cfrak}=\prod_{\qfrak \mid \cfrak} \chi_{\qfrak} $ on ${(\rif/ \cfrak)}^{\times}$ and $0$ otherwise. \\
\begin{lemma}
Let $\chi$ be an Hecke character of $\mathrm{Cl}(\cfrak)$ and 
\begin{align*}
 f= & \sum_{\xi} \lambda(\xi,\rif;f,\psi)q^{\xi/2}
\end{align*} 
be an element of $\boldM_{k}(\Nfrak,\psi,\C)$, $k$ an half-integral weight. Then 
\begin{align*}
 \sum_{\xi} \chi(\xi \rif)\lambda(\xi,\rif;f,\psi)q^{\xi/2}
\end{align*} 
belongs to $\boldM_{k}(\Nfrak\cfrak^2,\psi\chi^2,\C)$ and we will denote it by $f|\chi $.
\end{lemma}
\begin{proof}
We proceed  as in \cite[\S 7.F]{Hr3}. We define 
\begin{align*} 
 f|\chi = & {G(\chi)}^{-1}\sum_{u } \chi(u) f|\left( \begin{array}{cc} 
1 & u \\ 
0 & 1 
\end{array}\right),
\end{align*}
for $u \in \rif_{\cfrak}$ which runs over a set of representative of $\cfrak^{-1}2\dfrak^{-1}/ 2\dfrak^{-1}$. 
We have \begin{align*}
\lambda(\xi,y\rif;f|\chi,\psi)= & {G(\chi)}^{-1} \lambda(\xi,y\rif;f,\psi)\sum_u \chi(u) \bexp_F(y^2\xi u /2) 
        \end{align*}
Recall that $\bexp_{\qfrak}(x_{\qfrak}) = 1 $ if $x_{\qfrak}$ is in $\dfrak^{-1}_{\qfrak}$, for $\qfrak$ a non-archimedean place, and $\sum_{u} \chi(u) =0$ if $\chi$ is not trivial. \\
We have in particular  $\lambda(\xi,\rif;f|{\chi},\psi)= \chi(\xi )\lambda(\xi,\rif;f,\psi)$. 
\end{proof}
Let  $\Cc(\rif_{\pfrak}/\pfrak^e,\oo)$ be the set of functions from $\rif_{\pfrak}/\pfrak^e $ to $\oo$, as a $\oo$-module it is spanned by the characters of $\rif_{\pfrak}^{\times}/\pfrak^r $ and the constant function. So we can define the operator $|\phi$ for all function $\phi$ in $\Cc(\rif_{\pfrak}/\pfrak^e,\oo)$. 
Let $\Cc(\rif_{\pfrak},\oo)$ be the set of continuous functions from $\rif_p$ to $\oo$, the locally constant function are dense in this set. For all $\phi$ in $\Cc(\rif_{\pfrak},\oo)$ we define the operator $|\phi$ on $\overline{\boldM}_{\mathrm{half}}(\Nfrak,\oo)$ in the following way: let $\lgr \phi_n \rgr$ be some locally constant functions such that $\phi = \lim_n \phi_n$, then we pose $\boldf|\phi = \lim_n \boldf|\phi_n$.\\
We can also define on $\overline{\boldM}_{\mathrm{half}}(\Nfrak,\oo)$ the differential operator $\textup{d}^{\sigma}$, for $\sigma \in I$; take $y$ in $\rif_p^{\times}$, we have the continuous map $(\phantom{e})^{\sigma}: y \mapsto y^{\sigma}$. We define $\textup{d}^{\sigma}\boldf = \boldf |(\phantom{e})^{\sigma}$. We have then $\lambda(\xi,\rif;\textup{d}^{\sigma}\boldf,\psi)=\xi^{\sigma}\lambda(\xi,\rif;\boldf,\psi)$. As $y \mapsto y^{-1}$ is a continuous map, we can define also $\textup{d}^{-\sigma}$ such that 
\begin{align*}
\lambda(\xi,\rif;\textup{d}^{-\sigma}\boldf,\psi) = \iota_p(\xi)\xi^{-\sigma}\lambda(\xi,\rif;\boldf,\psi),
\end{align*}
where $\iota_p(\xi)= 0$ if $(p,\xi) \neq \rif$ and $1$ otherwise. We define for all $r \in \Z[I]$ the operator $\textup{d}^r = \prod_{\sigma \in I} {\textup{d}^\sigma}^{r_\sigma}$.\\  
We have an action of $\rif_p^{\times}$ on $\overline{\boldM}_{\mathrm{half}}(\Nfrak,\oo)$; it acts via the Hecke operator $T(b,b^{-1})$ defined in \ref{ptheory}.  $T(b,b^{-1})$ acts by right multiplication by the matrix $r_P \left( \begin{array}{cc} 
b & 0 \\ 
0 & b^{-1} 
\end{array}\right)_f$, as this element normalizes the congruence subgroup $D(2^{-1}\Nfrak,2) $. This actions commutes with the action of ${(\rif/\Nfrak p^r\rif)}^{\times}$ via the diamond operators. We can describe this action in the complex setting as follows. \\
Let $(b',b)$ in ${(\rif/\Nfrak\rif)}^{\times} \times \rif_p^{\times}$ and take $b_0$ in $\faft$ which is projected to $(b',b)$.  For $f$ in $\boldM_{k}(\Nfrak p^r,\psi,\oo)$ we pose $ f | (b',b) = b^{k} f | \gamma_b$ with $\gamma_b \in \Sl(F)$ such that 
\begin{align*}
\gamma_b^{-1} r_P \left( \begin{array}{cc} 
b_0 & 0 \\ 
0 & b_0^{-1} 
\end{array}\right)_f \in D(2^{-1}\Nfrak,2)\Sl(F_{\infty})
\end{align*}
and then extend it by continuity on the whole $\overline{\boldM}_{\mathrm{half}}(\Nfrak,\oo)$. Here we are implicitly assuming that the coefficients ring $\oo$ contains all the square roots of elements of $\rif_p^{\times}$.\\
Such an action is compatible with the action of $\Sl(\hat{\rif})$ on the forms of integral weight as defined in \cite[\S 7.C]{Hr3}.\\

Take now $f $ in $\boldM_{k_1+ \frac{t}{2}}(\Nfrak,\psi_1,A)$ and  $g $ in $\boldM_{k_2+ \frac{t}{2}t}(\Nfrak,\psi_2,A)$; the product $fg$ is a modular form in $\boldM_{k_1+ k_2 + t} (\G^1[2^{-1}\Nfrak,2],\psi_1\psi_2\chi_{-1},A)$ where $\chi_{-1}$ is the quadratic character modulo $4$ defined before via the automorphy factor $h(\gamma,z)$.\\
It is clear that this product induces a map 
\begin{align*}
 \overline{\boldM}_{k_1+ \frac{t}{2}}(\Nfrak,\oo) \times \overline{\boldM}_{k_2+ \frac{t}{2}}(\Nfrak,\oo) \rightarrow \overline{\boldM}_{k_1 +k_2 +t}(\G^1[2^{-1}\Nfrak,2],\oo)
\end{align*}
which on the level of $q$-expansion is just multiplication of formal series.\\

We point out that when $F = \Q$, the congruence subgroups which are considered for the integral weights and the ones for the half-integral weights are the same, because in $\Z$ the only totally positive unit is $1$, while when $F\neq \Q$ the totally positive units are of positive rank.\\ 
This poses a problem because in the sequel we shall need to consider the product of two half-integral weight modular forms as a ``$\gl$-type modular form'' whereas it is only a form of ``$\Sl$-type''.\\ 
We follow the ideas of  \cite[Proposition 4.1]{Im} to work out this problem.\\
For simplicity, we pose $\G[\Nfrak]:= \G[2^{-1} \Nfrak, 2]$ and $\G^1[\Nfrak]:=\G^1[2^{-1} \Nfrak, 2]$. Let $k,v$ be {\it integral weights} and $A$ a $p$-adic $\oo$-algebra of characteristic $0$, we define 
\begin{align}
\overline{\boldM}_{k}(\G^1[\Nfrak],A) & = \overline{ \bigcup_{ r \in \mathbb{N}_{>0},\psi_0} \boldM_{k}(\G^1[\Nfrak p^r],\psi_0,A)}, \notag\\
\overline{\boldM}_{k,v}(\G[\Nfrak],A) & = \overline{ \bigcup_{ r \in \mathbb{N}_{>0},\psi,\psi'} \boldM_{k}(\G[\Nfrak p^r],\psi,\psi',A)},\notag\\
\overline{\boldM}(\G^1[\Nfrak],A)  & = \overline{ \sum_{k \in \Z[I]} \overline{\boldM}_{k}(\G^1[\Nfrak],A)}, \label{kunion} \\
\overline{\boldM}(\G[\Nfrak],A)  & = \overline{ \sum_{k, v \in \Z[I]} {\boldM}_{k,v}(\G[\Nfrak],A)},\notag
\end{align}
where completion is taken w.r.t. the $p$-adic topology defined as the maximum of the $p$-adic norm of the coefficients in the $q$-expansion.\\
We consider now the three following tori of $\gl(\rif/\Nfrak p^r\rif)$: 
\begin{align*}
\mathbb{T}_{\mathrm{ss}} = & \left\{ \left( \begin{array}{cc} 
a & 0 \\ 
0 & a^{-1} 
\end{array}\right)| a \in  {(\rif/\Nfrak p^r\rif)}^{\times} \rgr,\\
\mathbb{T}_{Z} = & \left\{ \left( \begin{array}{cc} 
a & 0 \\ 
0 & a 
\end{array}\right)| a \in  {(\rif/\Nfrak p^r\rif)}^{\times} \rgr,\\
\mathbb{T}_{\alpha} = & \left\{ \left( \begin{array}{cc} 
a & 0 \\ 
0 & 1 
\end{array}\right)| a \in  {(\rif/\Nfrak p^r\rif)}^{\times} \rgr.
\end{align*}
%We define a fourth torus, which we will use in the sequel, 
%\begin{align*}
%\mathbb{T}_{\delta} = & \left\{ \left( \begin{array}{cc} 
%1 & 0 \\ 
%0 & a 
%\end{array}\right)| a \in  {(\rif/\Nfrak p^r\rif)}^{\times} \rgr.
%\end{align*}
Let $\psi_0$ resp. $\psi$, $\psi'$ be a character of the torus $\mathbb{T}_{\mathrm{ss}}$, resp. $\mathbb{T}_{Z}$, $\mathbb{T}_{\alpha}$. We say that they form a {\it compatible triplet} if 
\begin{align*}
\psi_0\left( \left( \begin{array}{cc} 
a & 0 \\ 
0 & a^{-1} 
\end{array}\right) \right) &= \psi\left( \left( \begin{array}{cc} 
a^{-1} & 0 \\ 
0 & a^{-1} 
\end{array}\right) \right) \psi' \left( \left( \begin{array}{cc} 
a^{2} & 0 \\ 
0 & 1 
\end{array}\right) \right). 
\end{align*}

\begin{prop}
Let $\psi_0$, $\psi$ and $\psi'$ as above forming a compatible triplet. Suppose that $\psi$ factors through ${(\rif/\Nfrak p^r\rif)}^{\times} / E$. Let $w \in \Z^d$ be such that $k-2w=mt$, $m\geq 0$  and pose, as in Section \ref{ComplexHilbert}, $v=t-w$.\\
We have a map 
$$
\begin{array}{cccc}
\mathrm{Tr}_{\Sl}^{\gl}(v,\psi') : & \overline{\boldM}_{k}(\G^1[\Nfrak],\oo) & \rightarrow & \overline{\boldM}_{k,v}(U[{\Nfrak}],\oo)
\end{array}
$$
which is equivariant for the action of $\G[\Nfrak]$ and $\rif_p^{\times}$ and such that if $f$ is a modular form in $\boldM_k(\G^1[\Nfrak p^r],\oo;\psi_0)$, then $\mathrm{Tr}_{\Sl}^{\gl}(v,\psi')(f)$ is a modular form in $\boldM_{k,v}(U[\Nfrak p^r], \oo;\psi,\psi')$.
\end{prop}

\begin{proof}
We point out that we are adding just one variable, namely $m$. \\

We extend $f$ to $E\G^1[\Nfrak]$ simply requiring that it is invariant for the diagonal action of the units.\\
Let $R$ be a set of representatives of $\G[\Nfrak p^r] / {E}\G_{1}[\Nfrak p^r]\cong E/ {E}^2 $, we define  
\begin{align*}
\mathrm{Tr}_{E}(f) = \frac{1}{2^{d-1}} \sum_{\eps_i \in R} \psi'(\eps_i^{-1})f |_{k,w} \eps_i.
\end{align*}

For all $\gamma$ in $\G_{1}[\Nfrak p^r]$, we have $  \eps_i\gamma = \gamma_1  \eps_i $, with $\psi_0(\gamma)=\psi_0(\gamma_1)$. Moreover, the matrix  $
\left( \begin{array}{cc} 
\eps & 0 \\ 
0 & 1 
\end{array}\right)$ acts via $\psi'(\eps)$.\\
We call the space of such forms $\boldM_{k,v}(\G[\Nfrak p^{r}],\oo;\psi_0,\psi')$. 

It is  easy to see that $\mathrm{Tr}_{E}$ commutes with the following inclusions:
\begin{align*}
  \boldM_k(\G^1[\Nfrak p^r],\oo;\psi_0) \hookrightarrow & \boldM_k(\G^1[\Nfrak p^{r+1}],\oo;\psi_0),\\ 
  \boldM_{k,v}(\G[\Nfrak p^r],\oo;\psi_0,\psi') \hookrightarrow & \boldM_{k,v}(\G[\Nfrak p^{r+1}],\oo;\psi_0,\psi').
\end{align*}
At the beginning of Section \ref{ComplexHilbert} we decomposed the space $\boldM_{k,w}(U(\Nfrak p^r))$ of adelic modular forms for $U(\Nfrak p^r)$ into complex ones. Let $\afrak_i$ be a set of representatives for the strict class group of $F$ modulo $\Nfrak p^r$. As $\psi$ factors through ${(\rif/\Nfrak p^r\rif)}^{\times} / E$, we can see it as a character of $\mathrm{Cl}(\Nfrak p^r)$. We define
$$
\begin{array}{cccc}
\I :& \boldM_{k,w}(\G[\Nfrak],\psi_0,\psi')& \rightarrow & \boldM_{k,w}(U(\Nfrak),\psi,\psi')\\
 & f & \mapsto & \I_{\Nfrak p^r}(f,\ldots, \psi^{-1}(\afrak_i) f, \dots,\psi^{-1}(\afrak_{h(\Nfrak p^r)}) f),
\end{array}
$$
where $\I_{\Nfrak p^r}$ is the isomorphism defined in Section \ref{ComplexHilbert}. We can choose $\afrak_i$  in a compatible way when $r$ grows such that it is compatible with the obvious inclusion 
\begin{align*}
\boldM_{k,w}(U(\Nfrak p^{r})) \hookrightarrow \boldM_{k,w}(U(\Nfrak p^{r+1}))
\end{align*}
and compatible also with the Nebentypus decomposition.\\
We define $\mathrm{Tr}_{\Sl}^{\gl}(v,\psi') := \I \circ \mathrm{Tr}_{E}$. We can extend this map to the space of $p$-adic modular forms Nebentypus by Nebentypus and then to the completion, as it is clearly continuous. It is equivariant by construction.\\
\end{proof}

Let us show that the trace morphisms defined above vary $p$-adically continuously. We decompose $\rif_p^{\times}$ as the product of a torsion part $\mu$ and its free part $\boldW'$ which we identify with ${(1+ p \Z_p)}^d$. Fix elements $a_j$ for $j=1, \ldots, d$ which generates $\boldW'$.\\
Choose as before a set $R$ of representative (independent of $r$) for $\G[\Nfrak p^r] / {E}\G_{1}[\Nfrak p^r]$. Let $q=p^f$ and $s $ big enough such that all the torsion of $\rif_p^{\times}$ is killed by $(q-1)p^s$. Fix $\psi_0$, $\psi$ and $\psi'$ as above, then we have if  $v \equiv v' \bmod (q-1)p^s$
\begin{align*} 
\mathrm{Tr}_{E}(v,\psi')f = & \frac{1}{2^{d-1}} \sum_{R} \psi'(\eps_i^{-1}){\mathrm{det}(\eps_i)}^w f|_{k,0}  \eps_i \\ 
\equiv & \frac{1}{2^{d-1}} \sum_{R} \psi'(\eps_i^{-1}){\mathrm{det}{(\eps_i)}}^{w'} f|_{k,0}  \eps_i \bmod p^s\\ 
     = &  \mathrm{Tr}_{E}(v',\psi')f \bmod p^s.
\end{align*} 

This shows that the trace morphism $\mathrm{Tr}_{E}(v,\psi')$ can be $p$-adically continuously interpolated over $\rif_p^{\times}$.\\ 
For each $y$ in $\rif_p^{\times}$, we write the projection of $y$  to $\boldW'$ as $({\lla y \rra}_1, \ldots, {\lla y \rra}_d)$, we have by definition
\begin{align*}
{\lla y \rra}_j= &  a_j^{\log_p({\lla y\rra}_j)/\log_p(a_j)}.
\end{align*}
Let us identify $\oo[[\boldW']]$ with $\oo[[X_1, \ldots, X_d]]$, in  such a way  that $a_j$ corresponds to $1+X_j$ and define $g_i=\mathrm{det}(\eps_i)$. We pose
\begin{align*}
\mathrm{Tr}_{\Sl}^{\gl} :\overline{\boldM}_{k} (\G^1[\Nfrak],\oo) &  \rightarrow  \overline{\boldM}_{k}(U(\Nfrak), \oo[[\rif_p^{\times}]]) \\
 f & \mapsto \I \left( \frac{1}{2^{d-1}} \sum_{R} \psi'(\eps_i^{-1}) A_{g_i}(X) f|_{k,0} \eps_i\right)   
\end{align*}
where $A_{y}(X)=\prod_j (1+X_j)^{\log_p({\lla y \rra}_j)/\log_p(a_j)}$.\\
For all points $P$ of $\mathrm{Spec}(\oo[[\rif_p^{\times}]])$ of type $a \mapsto \psi'(a)a^{v}$, we obtain a commutative diagram 
$$
 \xymatrix{ 
 {\overline{\boldM}_{k} (\G^1[\Nfrak],\oo)} \ar[r]^-{\tiny{\mathrm{Tr}_{\Sl}^{\gl}}}  & \overline{\boldM}_{k}(U(\Nfrak), \oo[[\rif_p^{\times}]])\ar[d]^P \\
{{\boldM}_{k} (\G^1[\Nfrak p^r],\psi_0, \oo)} \ar[u] \ar[r]^-{\tiny{\mathrm{Tr}_{\Sl}^{\gl}(v,\psi'_P)}} & {\boldM}_{k} (U(\Nfrak),\psi,\psi'_P,\oo)  }
$$

for $\psi'_P(\zeta)=\psi'(\zeta)\zeta^{-v} $ (\cite[page 337]{Hr3}).\\

Consider the action of $\rif_p^{\times}$ defined above, we see exactly as in \cite[page 334]{Hr3} that $b \in \rif_p^{\times}$ acts via $b^{k}$. This shows that the sum over $k$ in (\ref{kunion}), before taking completion,  is a direct sum. Substituting in the above definition $A_{g_i}(X) f|_{k,0} \eps_i$ by $f|T(g_i,1)$, for $T(g_i,1)$ the Hecke operator defined in Section \ref{ptheory}, we can extend $\mathrm{Tr}_{\Sl}^{\gl}$ to the map below
\begin{align}\label{ptrace}
\mathrm{Tr}_{\Sl}^{\gl} :\overline{\boldM} (\G^1[\Nfrak],\oo) &  \rightarrow  \overline{\boldM}(U(\Nfrak), \oo[[\rif_p^{\times}]]).
\end{align}

We show now some compatibility with the Hecke action;  Shimura in \cite[\S 5]{ShH3} has defined Hecke operators $T'(p)$ in the half-integral weight case.\\ 
There is a correspondence between the integral weight Hecke operator $T_0(p^2)$ and the half-integral weight Hecke operator $T'(p)$. Notice that we can define an operator $T(p)$ on ${\boldM}_{k,v}(\G[\Nfrak], \oo)$ because $T(p)$, as defined in Section \ref{ptheory}, does not permute the connected components of the Shimura variety associated with $V_1(\Nfrak)$.\\ 
Using \cite[Proposition 7.4]{Hr3} (note that there is a misprint in the expression of $\bolda(y,\boldf| T_0^n(p))$ and that the correct expression can be found in \cite[(2.2b)]{Hr3}), we see that $T_0(p^n)$ is given by
\begin{align*}
 f=\sum_{0 \ll \xi \in F^{\times} } a(\xi,f)q^{\xi/2} \mapsto & f|T_0(p^n) = p^{nv}{\lgr p \rgr}^{-nv}\sum_{0 \ll \xi \in F^{\times} } a(p^n \xi,f)q^{\xi/2}
\end{align*}

Using \cite[Proposition 5.4]{ShH3}  we see that the operator $T'(p^n)$ defined by Shimura is given by
\begin{align*}
 f=\sum_{0 \ll \xi \in F^{\times} } a(\xi,f)q^{\xi/2} \mapsto & f|T'(p^n)=p^{-nk}\sum_{0 \ll \xi \in F^{\times} } a(p^{2n} \xi,f)q^{\xi/2}.
\end{align*}
So we have then $T(p^{2n}) = {\N(p^n)}^{(m+2)}T'(p^n)$. We can think of $T'(p^n)$ as the operator on ``unitarized'' Hilbert modular forms (for which $v=-k/2$). \\
This compatibility allows us to notice the following; let $\chi$ be a Hecke character of conductor $p^n$ and $f$ and $g$ two forms in $\boldM_{k_1+ \frac{t}{2}}(\Nfrak,\psi_1,\oo)$ and $\boldM_{k_2+ \frac{t}{2}}(\Nfrak,\psi_2,\oo)$, then as in \cite[Proposition 7.4]{Hr3} we obtain 
\begin{align*}
 T_0(p^n)(f \; g|\chi)= & \chi_{\infty}(-1){T_0(p^n)}(f|\chi \; g).
\end{align*}
A similar statement applies to linear combinations of characters. In particular for the idempotent $e$ of Section \ref{ptheory} and all $r$ in $\Z[I]$ we have  
\begin{align*}
 e (f \; g|d^r)= & {(-1)}^{r} e (f|d^r \; g).
\end{align*}

 %%%%%%%%%%%%%%%%%%%%%%%%%%%%%%%%%%%%%%%%%%%%%%%%%%%%%%%%%%%%%%%%%%%%%%%%%%%%%%%%%%%%%%%%%%%%%%%%%%%%%%%%%%%%%%%%%%%%%%%%%%%%%%%%%%%%%%%%%%%%%%%%%%%%%%%%%%%%%
\section{The $L$-function for the symmetric square}\label{LFun}
Nothing new is presented in this section; we shall first recall the definition of the symmetric square $L$-function and present in details the Euler factors at place of bad reduction for $\boldf$. We explain in formula \ref{thetashift} the origin of one of the Euler type factor at $p$ which appears in the interpolation formula of the $p$-adic $L$-function \ref{T1}.
\subsection{The imprimitive $L$-function}\label{ComLFun}
Let $\boldf$ be a Hilbert modular form of level $\Nfrak$, weight $(k,w)$. We put $v=t-w$, and let $m$ be the non-negative integer such that $(m+2)t=k+2v$. Suppose that $\boldf$ is an eigenvector for the whole Hecke algebra and let $\psi$ be the finite order character such that $\boldf | T(z,z)= \psi(z) \boldf$. Let $\lambda$ be the morphism of the Hecke algebra such that $\boldf |h=\lambda(h)$ and pose  
\begin{align*}
\sum_{\mfrak \subset \rif} \frac{\lambda(T_0(\mfrak))}{{\N(\mfrak)}^s} & = \prod_{\qfrak \in Spec^{\circ}(\rif)} {(1 - \lambda(T(\qfrak)) \N(\qfrak)^{-s} + \psi({\qfrak}){\N(\qfrak)}^{m+1-2s} )}^{-1}.
\end{align*}
We decompose each Euler factor as
\begin{align*}
(1 - \lambda(T(\qfrak)) \N(\qfrak)^{-s} + \psi({\qfrak}){\N(\qfrak)}^{m + 1-2s} ) & = (1 - \alpha(\qfrak)\N(\qfrak)^{-s})(1 - \beta(\qfrak)\N(\qfrak)^{-s}).
\end{align*}
Take a Hecke character $\chi$ of $F$ such that $\chi_{\infty}(-1)= (-1)^{nt}$ with $0 \leq n \leq 1 $ and $nt \equiv k \bmod 2$.
The $L$-function of the symmetric square of the Galois representation associated with $\boldf$ coincides, up to some Euler factor at the bad primes, with 
\begin{align*}
\Ll (s,\boldf,\chi) =  & \prod_{\qfrak \in Spec^{\circ}(\rif)} {D_{\qfrak}(\chi(\qfrak){\N(\qfrak)}^{-s})}^{-1}
\end{align*}
where 
\begin{align*}
 D_{\qfrak}(X)=& (1 - {\alpha(\qfrak)}^{2}X)(1 - \alpha(\qfrak)\beta(\qfrak)X)(1 - {\beta(\qfrak)}^{2}X).
\end{align*}
The following elementary identity holds
\begin{align*}
\Ll (s,\boldf,\chi) = & L_{\Nfrak}(2s-2,\N^{2m}\psi^2\chi^2) \sum_{ \mathfrak{m} \subset \rif} {\lambda(T(\mfrak^2))} \chi(\mathfrak{m}) {\N({\mathfrak{m}})}^{-s}.
\end{align*}
We want to express such a $L$-function in terms of the Petersson product of $\boldf$ with a product of two half-integral weight modular forms. 
We define now the Rankin product of two modular forms for a congruence subgroup of $\Sl(\fa)$ following \cite[\S 3]{Im}. Let $\boldf$ be a modular form of integral weight $k$ and $\boldg$ of half-integral weight $l'=l+  \frac{t}{2}$, using the Fourier coefficients defined in Proposition \ref{FExp}, we define the  Rankin product $D(s,\boldf,\boldg)$ as
\begin{align*} 
D(s,\boldf,\boldg) = & \sum_{W}{\lambda(\xi,\mathfrak{m};\boldf,\psi)}\lambda(\xi,\mathfrak{m};\boldg,\phi)\xi^{-(l+k)/2}{\N(\xi{\mathfrak{m}}^2)}^{-s}.  
\end{align*}
where $W$ is the set of classes $(\xi,\mfrak)$, modulo the equivalence relation $(\xi,\mfrak)=(\xi\xi_0^2,\xi_0^{-1} \mfrak)$.
Notice that we have changed the definition of Im, as we do not conjugate the coefficients of $\boldf$.\\
Suppose now that we have $\boldf$ of weight $k$ and $\boldg = \theta_{nt}(\chi)$ with $n\equiv k \bmod 2$, we obtain  
\begin{align*}
 D(s,\boldf,\theta_n(\chi)) = & \sum_{W}{\lambda(\xi^2,\mathfrak{m};\boldf)}2 \chi_{\infty}(\xi)\chi(\xi \mathfrak{m}) \xi^{-k}{\N(\xi^2{\mathfrak{m}}^2)}^{-s} \\  
 = & \sum_{W}{\lambda(1,\xi \mathfrak{m};\boldf)}2 \chi(\xi \mathfrak{m}) {\N(\xi^2{\mathfrak{m}}^2)}^{-s} \\
 = &\sum_{ \mathfrak{m} \subset \rif}{\lambda(1,\mathfrak{m};\boldf)}2 \chi( \mathfrak{m}) {\N({\mathfrak{m}}^2)}^{-s}.
\end{align*}
If $\boldf=(\boldf_1,\cdots,\boldf_{h(\Nfrak)})$ is a modular form for $U(\Nfrak)$, then define $\boldf^{(2)} = \boldf | 
\left( 
\begin{array}{cc}
2^{-1} & 0 \\
0 & 1
\end{array}
\right)_{f}$.\\
We have then $\boldf_1^{(2)}(z)=\boldf_1(2z)$. Similarly to \cite[Proposition 2.2]{Im}, using (\ref{ShiHida}) for $\boldf^{(2)}$ we obtain that
\begin{align}\label{lambdaap}
 \lambda\left(1,y; \boldf_1(2 z),\psi \right) =2^{-v}\psi_{\infty}^{-1}(y)\N(y)^{-2-m} y_p^{2v}\bolda_p(y^2,\boldf).
\end{align}

Using the explicit form of $\lambda(T(\qfrak))$ in terms of the Fourier coefficient given in \ref{ptheory}, we conclude
\begin{align}
 D(s,\boldf_1(2z),\theta_n(\chi)) = & 2^{-v+1} \sum_{ \mathfrak{m} \subset \rif}{\lambda(T(\mfrak^2))}\chi( \mathfrak{m}) {\N({\mathfrak{m}})}^{-2s-m-2}, \notag \\
 D\left(\frac{s-m-1}{2},\boldf_1(2 z),\theta_n(\chi)\right)= & 2^{-v+1}\frac{ \Ll({s+1, \boldf};\chi)}{L_{\Nfrak}(2s-2m,\psi^2\chi^2)}.\label{DL} 
\end{align}

We quote the following proposition \cite[(3.13)]{Im} which gives an integral expression of the Rankin product defined above
\begin{prop}\label{IntExpre}
Let $f$ be a Hilbert modular form  in $\boldM_{k}(\G^1[2^{-1} \Nfrak, 2],\psi_1,\C)$  for an integral weight $k$, and $g$ a Hilbert modular form  of half integral weight $l'$ in $\boldM_{k}(\Nfrak,\psi_2,\C)$. Define 
\begin{align*}
 R = & \lgr \gamma \in \G^1[2^{-1}\Nfrak, 2] | a_{\gamma} \equiv 1 \bmod \Nfrak  \rgr .
\end{align*}
 and let $\Phi$ be a fundamental domain for $R \setminus \h^I$.  Then
\begin{align*}
\int_{\Phi} {f} {g}{\mathcal {E}} y^k \textup{d}\mu(z) = & B D_F^{-1/2} {(2\pi)}^{ - d(s-3/4) - \frac{k+l'}{2}} \G_{\infty}\left(\left(s-\frac{3}{4}\right)t + \frac{k+l'}{2}\right)D\left(s-\frac{1}{2},f,g\right) 
\end{align*}
where 
\begin{align*}
B = & 2\left[{\G^1[2^{-1}\Nfrak, 2]}: \lgr \pm 1 \rgr R\right]{N(2\dfrak^{-1})},\\
\mathcal {E}= & \mathcal {E}\left(z,s; k-l'-\frac{1}{2}t,\chi_{-1}\psi_2\psi_1^{-1}, \Nfrak\right).
\end{align*}
\end{prop}

Let $\chi$ be a Hecke character of level $\cfrak$ and  $\chi_0$ be the associated primitive character of conductor $\cfrak_0$.  Let us investigate now the relation between $D(s,\boldf_1(2z),\theta_n(\chi))$ and $D(s,\boldf_1(2z),\theta_n(\chi_0))$ in order to make explicit some of the Euler factors at $p$ (more precisely, the one denoted by $E_1$) which will appear in the interpolation formula of the $p$-adic $L$-function of Section \ref{padicL}.\\
Suppose that $p$ divides the level of $\boldf$ and suppose that $\cfrak/\cfrak_0$ is divisible only by primes above $p$. Let $\left[\efrak^2 \right]$ be the operator defined in \ref{halfp} and suppose $\efrak \rif | p$, we have 
\begin{align}\label{ThetaLambda}
 D(s,\boldf_1(2 z),\theta_{nt}(\chi_0)|\left[\efrak^2\right]) = & 2^{-v+1} \sum_{ \mathfrak{m} \subset \efrak \rif} \chi(\mfrak d^{-1}){\lambda(T(\mfrak^2))} {\N({\mathfrak{m}})}^{-2s-m-2} \notag \\
 = & \lambda(T(\efrak^2)){\N(\efrak)}^{-2s-m-2} 2^{v+1} \frac{\Ll(2s+m+2, {\boldf}, \chi)}{L_{ p \Nfrak}(4s+2,\psi^2\chi^2)}.
\end{align}
and using the formula
\begin{align*}
 \theta_{nt}(\chi) & = \sum_{\efrak\mid p} \mu(\efrak)\chi_0(\efrak) \theta_{nt}(\chi_0)|[\efrak^2] 
\end{align*}
we can conclude that 
\begin{align}\label{thetashift}
  D(s,\boldf_1(2 z),\theta_{nt}(\chi_0)) = & \prod_{\pfrak_i | \cfrak/\cfrak_0 }(1 - \lambda(T(\pfrak_i^2)){\N(\pfrak_i)}^{-2s -m -2}) D(s,\boldf_1(2 z),\theta_{nt}(\chi))
\end{align}

In \cite[Theorem 5.3]{Im} it is shown, for a Hilbert modular form  $\boldf$ of weight $(k, - k/2)$, the algebraicity and Galois equivariance of  the values $\Ll (s_0,\boldf,\chi)$ when divided by suitable periods, where $s_0$ ranges over the critical integers (in the sense of Deligne \cite{Del})  for the symmetric square. \\
Let $k_0$ be the minimum of $k_{\sigma}$ for $\sigma $ in $I$ and $n \in \lgr 0,1 \rgr$ such that $\chi(-1)={(-1)}^{nt}$. Using Deligne's formalism, one finds (see \cite{Im}) that the critical integers are 
\begin{align*}
\left\{ 0 \leq s < k_0  -n -3/2, s \equiv k_0 - n (2) \mbox{ or } -k_0 +n + 1/2 < s \leq -1, s \not\equiv k_0 -n (2)    \right\} .
\end{align*}
As $\Ll({s, \boldf},\chi)=\Ll({s - m -1, \boldf^u},\chi)$, supposing that $m \geq s \geq 0 $, we have that $s+1$ is a critical integer when $s \not\equiv n \bmod 2$ and $ m+1 \leq s  \leq m + k_0 -1 $. The other half of the critical values corresponds to  $s \equiv n \bmod 2$ and $ m-k_0 +n +2 \leq s  \leq m  $. \\

For any $\rho$ two dimensional representation, $\mathrm{Ad}(\rho)$ denotes the adjoint representation of $\rho$ on $\mathfrak{s}\lfrak_2$, the Lie algebra of $\Sl$. It is known that $\mathrm{Ad}(\rho)$ is the twist of 	$\mathrm{Sym}^2(\rho)$ by the inverse of the $\mathrm{det}(\rho)$. We can define a na\"ive $L$-function $\Ll(s,\mathrm{Ad}(\boldf),\chi)$ for $\mathrm{Ad}(\rho_{\boldf})$ which is then
\begin{align*}
 \Ll(s,\mathrm{Ad}(\boldf),\chi) =  & \Ll (s,\boldf,\chi\N^{-m-1}\psi^{-1})\\
 = &  L_{\Nfrak}(2s,\chi^2) \sum_{ \mathfrak{m} \subset \rif} \frac{\lambda(T(\mfrak^2)) \chi\psi^{-1}( \mathfrak{m})}{{\N({\mathfrak{m}})}^{s +m+1}}.
\end{align*}
We have introduced the $L$-function of $\mathrm{Ad}(\rho)$ because in the next subsection we will describe explicitly the Euler factors at the bad primes for it and in the literature such a classification is given in term of the adjoint representation, and we prefer to follow the classical references.
%%%%%%%%%%%%%%%%%%%

\subsection{The completed $L$-function}\label{Extrafactors}
Let $\boldf$ be a Hilbert modular form of weight $(k,v)$, with $(m+2)t=k+2v$, and of character $\psi$. Let $\lambda$ be the corresponding morphism of the Hecke algebra and let $\pi({\boldf})$ be the automorphic representation of $\gl(\fa)$ spanned by $\boldf$. In \cite{GJ}, the authors construct an automorphic representation of ${\mbox{GL}_3}(\fa)$ denoted  $\hat{\pi}({\boldf})$ and usually called the base change to ${\mbox{GL}_3}$ of $\pi({\boldf})$. It is standard to associate to $\hat{\pi}({\boldf})$ a complex $L$-function $\LL(s,\hat{\pi}({\boldf}))$ which satisfies a nice functional equation and coincides, up to some Euler factors at the primes for which ${\pi}({\boldf})$ is ramified and factors at infinity, with $\Ll(s,Ad(\boldf))$. We explicit now the functional equation as we will need it in the Appendix \ref{App} to show that the $p$-adic $L$-function that we will construct is holomorphic (unless $\boldf$ has CM). \\
For a Hecke character of $F$, the automorphic representation $\hat{\pi}({\boldf})\otimes \chi$ is defined via its $L$-factor at all primes. For any place $v$ of $F$, we pose 
\begin{align*}
 L_v(s,\hat{\pi}(\boldf),\chi) = & \frac{L_v(s,{\pi}{({\boldf})}_v \otimes \chi_v \times \tilde{\pi}{({\boldf})}_v) }{L_v(s, \chi_v) },
\end{align*}
 where $\tilde{\phantom{  }}$ denotes the contragredient and ${\pi}{(\boldf)}_v  \times \tilde{\pi}{(\boldf)}_v$ is a representation of $\gl(F_v)\times\gl(F_v)$. \\ 
The completed $L$-function \begin{align*}
 \LL(s,\hat{\pi}({\boldf}),\chi) = & \prod_v L_v(s,\hat{\pi}(\boldf),\chi) 
\end{align*}
is holomorphic over $\C$ except in a few cases which correspond to CM-forms with complex multiplication by $\chi$ \cite[Theorem 9.3]{GJ}.\\
In  addition to the original article by Gelbart and Jacquet, two very good references for the classification of the $L$-factors at bad prime are \cite[\S 1]{Sc} for $F=\Q$ and \cite[\S 7]{HiT}, where the authors work in the context of nearly-holomorphic forms.\\
Let $\pi = \pi (\boldf)$ and let $\qfrak$ be a place where $\pi$ ramifies and let $\pi_{\qfrak}$ be the component at $\qfrak$. By twisting by a character of $F_{\qfrak}^{\times}$, we may assume that $\pi_{\qfrak}$ has  minimal conductor among its twist. In fact, this does not change the factor at $\qfrak$, as one sees from the explicit calculation given in \cite[Proposition 1.4]{GJ}.
We distinguish the following four cases
\begin{itemize}
	\item[(i)]   $\pi_{\qfrak}$ is a principal series $\pi(\eta,\nu)$, with both $\eta$ and $\nu$ unramified,
	\item[(ii)]  $\pi_{\qfrak}$ is a principal series $\pi(\eta,\nu)$ with $\eta$ unramified,
	\item[(iii)] $\pi_{\qfrak}$ is a special representation $\sigma(\eta,\nu)$ with $\eta$, $\nu$ unramified and $\eta\nu^{-1} = |\phantom{e}|_{\qfrak}$,
	\item[(iv)]  $\pi_{\qfrak}$ is supercuspidal. 
\end{itemize}
We will partition the set of primes dividing the conductor of $\boldf$ as $\Sigma_1,\cdots,\Sigma_4$ according to this decomposition.
Let $\varpi_{\qfrak}$ be  a uniformizer of $F_{\qfrak}$. Just for the next three lines, to lighten notation, we assume, by abuse of notation, that when  a character is ramified its value on  $\varpi_{\qfrak}$ will be $0$.
The Euler factor ${L_{\qfrak}(\hat{\pi}_{\qfrak}\otimes \chi_{\qfrak},s)}^{-1}$ is then 

\begin{itemize}
	\item[(i)] $(1-\chi_{\qfrak}\nu^{-1}\eta(\varpi_{\qfrak})\N(\qfrak)^{-s})(1-\chi_{\qfrak}(\varpi_{\qfrak})\N(\qfrak)^{-s})(1-\chi_{\qfrak}\nu\eta^{-1}(\varpi_{\qfrak})\N(\qfrak)^{-s})$,
	\item[(ii)] $(1-\chi_{\qfrak}\nu^{-1}\eta(\varpi_{\qfrak})\N(\qfrak)^{-s})(1-\chi_{\qfrak}(\varpi_{\qfrak})\N(\qfrak)^{-s})(1-\chi_{\qfrak}\nu\eta^{-1}(\varpi_{\qfrak})\N(\qfrak)^{-s})$,
	\item[(iii)] $(1-\chi_{\qfrak}(\varpi)\N(\qfrak)^{-s-1})$,
	\end{itemize}

The supercuspidal factors are slightly more complicated and depend on the ramification of $\chi_{\qfrak}$. They are classified by \cite[Lemma 1.6]{Sc}; we recall them briefly. Let $\qfrak$ be a prime such that $\pi_{\qfrak}$ is supercuspidal, and let $\xi_{\qfrak}$ be the unramified quadratic character of $F_{\qfrak}$. If $\chi_{\qfrak}^2$ is unramified, let $\lambda_1$ and $\lambda_2$ the two ramified characters such that $\chi_{\qfrak}\lambda_i$ is unramified (for completeness, we can suppose $\lambda_{1}=\chi_{\qfrak}$ and $\lambda_2=\chi_{\qfrak}\xi_{\qfrak}$). We consider the following disjoint subsets of $\Sigma_4$, the set of cuspidal primes:
\begin{align*}
\Sigma_4^0 &=\left\{\qfrak \in \Sigma_4 : \chi_{\qfrak} \mbox{ is unramified and } \pi_{\qfrak}\cong\pi_{\qfrak}\otimes\xi_{\qfrak}  \right \},\\
\Sigma_4^1 &=\left\{\qfrak \in \Sigma_4 : \chi_{\qfrak}^2 \mbox{ is unramified and } \pi_{\qfrak}\cong\pi_{\qfrak}\otimes\lambda_i \mbox{ for } i=1,2 \right \},\\
\Sigma_4^2 &=\left\{\qfrak \in \Sigma_4 : \chi_{\qfrak}^2 \mbox{ is unramified and } \pi_{\qfrak}\not\cong\pi_{\qfrak}\otimes\lambda_1 \mbox{ and }\pi_{\qfrak}\cong\pi_{\qfrak}\otimes\lambda_2   \right \}, \\
\Sigma_4^3 &=\left\{\qfrak \in \Sigma_4 : \chi_{\qfrak}^2 \mbox{ is unramified and } \pi_{\qfrak}\not\cong\pi_{\qfrak}\otimes\lambda_2 \mbox{ and }\pi_{\qfrak}\cong\pi_{\qfrak}\otimes\lambda_1 \right \}.
\end{align*}
If $\qfrak$ is in $\Sigma_4$ but not in $\Sigma_4^i$, for $i=0,\cdots, 3$, then $L_{\qfrak}(s,\hat{\pi}_{\qfrak},\chi_{\qfrak})=1$.
If $\qfrak$ is in $\Sigma_4^0$, then $$ {L_{\qfrak}(s,\hat{\pi}_{\qfrak},\chi_{\qfrak})}^{-1}=1+\chi(\varpi_{\qfrak})\N(\qfrak)^{-s}$$
and if $\qfrak$ is in $\Sigma_4^i$, for $i=1,2,3$ then 
$${L_{\qfrak}(s,\hat{\pi}_{\qfrak},\chi_{\qfrak})}^{-1}=\prod_{j \mbox{ s.t.} \pi_{\qfrak}\cong\pi_{\qfrak}\otimes\lambda_j} (1-\chi_{\qfrak}\lambda_j(\varpi_{\qfrak})\N(\qfrak)^{-s}).$$

If $\sigma$ is an infinite place, the $L$-factor at $\sigma$ depends only on the parity of the character by which we twist. As we are interested in the symmetric square, we consider the twist by $\psi_{\sigma}\chi_{\sigma}$, $\chi$ as before.  We suppose that the parity of $\psi^{-1}_{\sigma}\chi_{\sigma}$  is independent of $\sigma$. Let $\kappa=0,1$ according to the parity of $m$, from \cite[Lemma 1.1]{Sc} we have  $L(s-m-1,\hat{\pi}_{\sigma},\chi_{\sigma}\psi_{\sigma})=\G_{\mathbb{R}}(s-m -\kappa)\G_{\mathbb{C}}(s-m-2+k_{\sigma})$ for the complex and real $\G$-functions
\begin{align*}
\G_{\mathbb{R}}(s) = & \pi^{-s/2}\G(s/2), \\
\G_{\mathbb{C}}(s) = & 2{(2\pi)}^{-s}\G(s).
\end{align*}
We define 
\begin{align*}\calE_{\Nfrak}(s,\boldf,\chi)=& \prod_{\qfrak |\Nfrak } (1-\chi(\qfrak)\lambda(T(\qfrak))^2\N(\qfrak)^{-s}){L_{\qfrak}(s-m-1,\hat{\pi}_{\qfrak},\psi_{\qfrak}^{-1}\chi_{\qfrak})}.
\end{align*}
Note that $\lambda(T(\qfrak))=0$ if $\pi$ is not minimal at $\qfrak$ or if $\pi_{\qfrak}$ is a supercuspidal representation.
We multiply then  $\Ll (s,\boldf,\chi)$, the imprimitive $L$-function, by $\calE_{\Nfrak}(s,\boldf,\chi)$ to get 
\begin{align*}
 L(s,\mathrm{Sym}^2(\boldf),\chi):= & L(s-m-1,\hat{\pi}(\boldf)\otimes \chi\psi) \\
 = & \Ll (s,\boldf,\chi)\calE_{\Nfrak}(s,\boldf,\chi).
\end{align*}
We can now state the functional equation
\begin{align*}
\LL(s,\hat{\pi}(\boldf),\chi) & =  \beps(s,\hat{\pi}(\boldf),\chi)\LL(1-s,\hat{\pi}(\boldf),\chi^{-1}),\\
\LL(s,\mathrm{Sym}^2(\boldf),\chi) & =\beps(s-m-1,\hat{\pi}(\boldf),\chi\psi)\LL(2m+3-s,\mathrm{Sym}^2(\boldf^c),\chi^{-1}).
\end{align*}

For a finite place $\qfrak$ of $F$, let us recall the proper normalization $\beps_{\qfrak}$ of the $\beps$-factor which makes it  algebraic and Galois equivariant.  In \cite{GJ}, the authors use the Langlands normalization, see \cite[\S 3.6]{Tate}, which we will denote by $\beps_{L,\qfrak}$. For a fixed place $\qfrak$, it gives to $\oo_{F_{\qfrak}}$ the measure ${|\dfrak_{F_{\qfrak}/\Q_q}|}^{-1/2}$. This is not the $\beps$-factor we want to use, so we multiplying $\beps_{L,\qfrak}$ by the factor ${\N(\dfrak_{F_{\qfrak}/\Q_q})}^{-s+1/2}$ and we denote this new factor by $\beps_{\qfrak}$. This is the $\beps$-factor used in \cite{Sc}.\\
If $\qfrak$ is such that $\pi_{\qfrak}$ is a principal series $\pi(\eta,\nu)$ or a special representation $\sigma(\eta,\nu)$, we have explicitly  \cite[Proposition 1.4, 3.1.2]{GJ}
\begin{align*} 
 \beps_{\qfrak}=\beps(s, \hat{\pi}_{\qfrak},\chi_{\qfrak}) = & \frac{\beps(s,\pi_{\qfrak}\otimes \chi_{\qfrak}\eta^{-1})\beps(s,\pi_{\qfrak}\otimes \chi_{\qfrak}\nu^{-1})}{\beps(s,\chi_{\qfrak})}.
\end{align*}
In fact, if $\pi_{\qfrak}$ and $\chi_{\qfrak}$ are both unramified, then $\beps(\hat{\pi}_{\qfrak},\chi_{\qfrak})=1$.\\
We know from \cite[\S 6.26.1, Theorem]{BH} that if $\pi_{\qfrak}$ is a principal series  $$\beps(s,\pi_{\qfrak}\otimes \chi_{\qfrak}\eta^{-1})= \beps(s, \chi_{\qfrak})\beps(s,\nu \chi_{\qfrak}\eta^{-1})$$ and that if $\pi_{\qfrak}= St$ is a special representation and $\chi_{\qfrak}$ is ramified then
$$\beps(s,\pi_{\qfrak}\otimes \chi_{\qfrak})=  {\beps(s, \chi_{\qfrak})}^2.$$  
For a supercuspidal $\pi_{\qfrak}$ we do not need an explicit formula for  the $\beps$-factor.\\
For all prime ideal $\qfrak$,   we know from \cite[page 475]{GJ} that $\beps_{\qfrak,L}(s,\pi_{1,\qfrak} \times \pi_{2,\qfrak})$ is  $\N(\qfrak^c)^{-s} \beps_{\qfrak,L}(0,\pi_{1,\qfrak} \times \pi_{2,\qfrak})$, where $c$ is a positive integer such that $\qfrak^c$ is the exact conductor of $\pi_{1,\qfrak} \times \pi_{2,\qfrak}$.\\
The $\beps$-factor at infinity is given by \cite[1.12]{Sc} 
$\beps(s,\hat{\pi}_{\infty},\chi_{\infty})= (-1)^{md}{\chi(-1)}^{-1/2}$.\\
Our interest is to see how the $\beps$-factor  changes under twist by gr\"ossencharacters $\chi'$ of $p^r$-level.  This behavior  is studied in \cite[Lemma 1.4 b)]{Sc}, but there the form $\boldf$ is supposed to be of level prime to $p$. We have then to make explicit the $\beps$-factor ${\beps(s,\chi_{\qfrak})}$; let $\alpha_{\qfrak}$ be the conductor of $\chi_{\qfrak}$ and $e_{\qfrak}$ the ramification index of $F_{\qfrak}$ over $\Q_q$.  If we define \begin{align*}
G(\chi_{\qfrak}) = & \chi^{-1}_{\qfrak}(\varpi_{\qfrak}^{e_{\qfrak}+\alpha_{\qfrak}})\sum_{x \bmod \varpi_{\qfrak}^{\alpha_{\qfrak}}} \chi_{\qfrak}(x)    \bexp_{F_{\qfrak}}\left(\frac{x}{\varpi^{e_{\qfrak}+\alpha_{\qfrak}}}\right) \end{align*}
we have then 
\begin{align*}
 {\beps(s,\chi_{\qfrak})} = {\N(\varpi_{\qfrak}^{e_{\qfrak}+\alpha_{\qfrak}})}^{-s} {\N(\varpi_{\qfrak}^{e_{\qfrak}})}^{1/2}G(\chi_{\qfrak}).
\end{align*}
We can summarize this discussion in the following lemma
\begin{lemma}\label{epsfactor}
Let $\hat{\pi}$ be the automorphic representation of $\mathrm{GL}_3(\fa)$ associated with $\boldf$. Suppose that $\pi_{p}:=\otimes_{\pfrak | p} \pi_{\pfrak}$ is a product of principal series $\pi_{\pfrak}\cong \pi(\eta,\nu)$ or special representations $\pi_{\pfrak}\cong \sigma(\eta,\nu)$. Let $\chi$ be any gr\"ossencharacter and $\chi'$ a finite order character of the class-group modulo $p^{\infty}$, we have 
\begin{align*}
 \beps(s,\hat{\pi},\chi\chi') = & \beps(s,\hat{\pi},\chi) \prod_{\pfrak} \frac{{C(\chi_{\pfrak}\chi'_{\pfrak})}^{-s}{C(\nu\chi'_{\pfrak}\chi_{\qfrak})}^{-s}{C(\nu^{-1}\chi'_{\pfrak}\chi_{\pfrak})}^{-s}}{ {C(\chi_{\pfrak})}^{-s}{C(\nu\chi_{\qfrak})}^{-s}{C(\nu^{-1}\chi_{\pfrak})}^{-s} } \times \\ 
& \times \frac{{G(\chi_{\pfrak}\chi'_{\pfrak})}{G(\nu\chi'_{\pfrak}\chi_{\qfrak})}{G(\nu^{-1}\chi'_{\pfrak}\chi_{\pfrak})}}{{G(\chi_{\pfrak})}{G(\nu\chi_{\qfrak})}{G(\nu^{-1}\chi_{\pfrak})}}.
\end{align*}
\end{lemma}

%%%%%%%%%%%%%%%%%%%%%%%%%%%%%%%%%%%%%%%%%%%%%%%%%%%%%%%%%%%%%%%%%%%%%%%%%%%%%%%%%%%%%%%%%%%%%%%%%%%%%%%%%%%%%%%%%%%%%%%%%%%%%%%%%%%%%%%%%%%%%%%%%%%%%%%%%%%%%%%%%%%%%% 
 
\section{Some useful operators}\label{operators}
We define in this section certain operators which will be useful for the construction of $p$-adic $L$-functions in Section \ref{padicL}.

\subsection{The Atkin-Lehner involution on $\gl$}\label{Atkin}
Let $\Nfrak$ be an integral ideal and $\nfrak$ a finite id\`ele which represents $\Nfrak$: $\nfrak \rif = \Nfrak$. \\
Let  $\boldf$ be an element of $\boldM_{k,w}(U_0(\Nfrak),\psi,\psi')$, where $k$ and $w$ are two integral weights. We define 
\begin{align*}
 \boldf| \tau'(\nfrak)  =  & \psi^{-1}(det(x))\boldf\left(x\left( 
\begin{array}{cc}
 0 & -1 \\
 \dfrak^{2}\nfrak & 0 
\end{array}
\right)\right).
\end{align*}
This operator does not change the level but the Nebentypus of $\boldf| \tau'(\nfrak)$ is $(\psi^{-1},{\psi'}^{-1})$. \\

Take an in integral ideal $\mathfrak{L}$ prime to $p$, and a finite id\`ele $\mathfrak{l}$ such that $\mathfrak{L}=\mathfrak{l}\rif$ and $\mathfrak{l}_p =1$.
We define a {\it level raising} operator
\begin{align*}
  [\mathfrak{l}] : \boldf \mapsto \N(\mathfrak{L})^{-1} \boldf|
\left(
\begin{array}{cc}
 \mathfrak{l^{-1}} & 0\\
 0 & 1
\end{array}
\right).
\end{align*}
If the level of $\boldf$ is big enough, at least $V_1(\Nfrak)$, we have independence of the choice of $\mathfrak{l}$.

Which is the relation between $\tau'(\nfrak)$, $[\mathfrak{l}]$ and $\tau'(\nfrak\mathfrak{l})$?
We have 
\begin{align*}
\boldf|\tau'(\nfrak)|[\mathfrak{l}] (x) = & \N(\mathfrak{l})^{-1}\chi^{-1}(det (x) {\mathfrak{l}}^{-1}) \boldf\left( x \left(
\begin{array}{cc}
 \mathfrak{l^{-1}} & 0\\
 0 & 1
\end{array}
\right) 
\left(
\begin{array}{cc}
 0 & -1\\
 \dfrak^2\nfrak & 0
 \end{array}
\right)  \right) \\
 =  &  {\N(\mathfrak{l})}^{-m-1}\boldf|\tau'(\nfrak\mathfrak{l}).
 \end{align*}
 
If $\mathfrak{m}\rif | \mathfrak{n}\rif$ we have also 
\begin{align*}
\boldf|[\mathfrak{m}]|\tau'(\nfrak) (x)=  {\N(\mathfrak{m})}^{-1}\boldf |\tau'(\mathfrak{n}\mathfrak{m}^{-1}). 
\end{align*}

The operator $\tau'(\nfrak^2)$ differs from the operator of half-integral weight $\tau(\nfrak^2)$ defined at the end of Section \ref{HIntweight} by a constant which corresponds to the central character of $\boldf$: $\boldf|\tau'(\nfrak^2)= \psi(\nfrak\dfrak)\N(\nfrak\dfrak)^{m}\boldf|\tau(\nfrak^2)$. \\

Suppose now that $l'$ is an half-integral weight, $l'=l+ \frac{t}{2}$, and let $\nfrak, \mfrak$ be two id\`eles which represent $\Nfrak$ and $\Mfrak$, with  $\Mfrak | \Nfrak$. For the operators  $\tau(\nfrak^2)$ and $[\mfrak^2]$ defined at the end of Section \ref{HIntweight}, we have 
\begin{align*} 
f|[\mfrak^2]|\tau(\nfrak^2)= & f|\tau(\nfrak^2 \mfrak^{-2}) \N(\Mfrak)^{-1/2}, \\
f|\tau(\nfrak^2)|[\mfrak^2]= & f|\tau(\nfrak^2 \mfrak^2 ) \N(\Mfrak)^{-1/2} .
\end{align*}

\subsection{Some trace operators}
We recall some trace operators defined in \cite[7.D,E]{Hr3}. Let $\Nfrak$ and $\mathfrak{L}$ two integral ideals of level prime to $p$ such that $ \mathfrak{L}\mathfrak{N} = \Mfrak$, we define a trace operator
$$ 
\begin{array}{cccc}
Tr_{\Mfrak/\mathfrak{N}} : & \overline{\boldM}_{k,w}(U(\mathfrak{N},\mathfrak{L})) & \rightarrow & \overline{\boldM}_{k,w} (U(\mathfrak{N}))\\
& \boldf & \mapsto & \sum_{x \in U(\mathfrak{N},\mathfrak{L}) / U(\mathfrak{N})} \boldf|x
\end{array}.$$
It naturally extends to $p$-adic modular form if $p$ is coprime with $\Mfrak$.\\
We define then a {\it twisted} trace operator
$$ 
\begin{array}{cccc}
T_{\Mfrak/\Nfrak} = 
 Tr_{\Mfrak/\mathfrak{N}}\circ \left(
\begin{array}{cc}
 \mathfrak{l} & 0\\
 0 & 1
\end{array}
\right)  : & \overline{\boldM}_{k,w} (U(\Mfrak)) & \rightarrow &  \overline{\boldM}_{k,w} (U(\Nfrak)),
\end{array}
$$
where $\mathfrak{l}$ and $\nfrak$ are two id\`eles representing the ideals $\mathfrak{L}$ and $\Nfrak$.
%%%%%%%%%%%%%%%%%%%%%%%%%%%%%%%%%%%%%%%%%%%%%%%%%%%%%%%%%%%55555
\subsection{The Petersson product}
We recall briefly the definition of the Petersson inner product given in \cite[\S 4]{Hr3}. For $\boldf$ and $\boldg$ in $\boldM_{k,v}(U_0(\Nfrak),\psi,\psi')$, we define
\begin{align*}
\lla \boldf,\boldg \rra_{\Nfrak} = \int_{X_0(\Nfrak)} \overline{\boldf(x)} \boldg(x) |\mathrm{det}(x)|_{\Aa}^m \textup{d} \mu_{\Nfrak}(x),
\end{align*}
where $X_0(\Nfrak)$ is the Shimura variety associated with $U_0(\Nfrak)C_{\infty +}$ and $\mu_{\Nfrak}$ is a measure on $X_0(\Nfrak)$ which is induced from the standard measure on the Borel of $\gl(\fa)$. Let us point out that we do not divide by the volume of the corresponding Shimura variety.\\
Let $h$ be the strict class number of $F$ and $\afrak_i$ a set of representatives, using the decomposition 
\begin{align*}
X_0(\Nfrak) = &  \bigcup_{i=1}^h \h^{I}/\G[\Nfrak\afrak_i,\afrak_i^{-1}],   
\end{align*}
we have 
\begin{align*}
\lla \boldf,\boldg\rra_{\Nfrak} = \sum_{i=1}^h  \lla \boldf_i,\boldg_i \rra_{\Nfrak\afrak_i}
\end{align*}
where 
\begin{align*}
\lla \boldf_i,\boldg_i\rra_{\Nfrak\afrak_i} = \N(\afrak_i)^{m} \int_{\h^{I}/\G[\Nfrak\afrak_i,\afrak_i^{-1}]} \overline{\boldf_i(z)}\boldg_i(z) y^{k}\textup{d} \mu(z)
\end{align*}
 $y=\mathrm{Im}(z)$ and $\mu(z)$ is the standard measure on $\h^I$ invariant under linear fractional transformations.\\
Denote by $\boldf^c$ the Hilbert modular form whose Fourier coefficients are the complex conjugate of $\boldf$. If we define $(\boldf,\boldg)=\lla \boldf^c|\tau'(\nfrak),\boldg\rra$ (we dropped from the notation the dependence on the level), we have then that the Hecke algebra is self-adjoint this Petersson product $(-,-)$. We have, \cite[7.2]{Hr3},  the following adjunction formula  
\begin{align*}
  \lla \boldf, \boldg| T_{\mathfrak{L}/\Nfrak} \rra_{\mathfrak{N}} = & {\{\mathfrak{L}/\Nfrak \}}^{-v} {\N(\mathfrak{L}/\Nfrak)}^{1-m}\lla \boldf|[\mathfrak{m}],\boldg \rra_{\mathfrak{L}} 
\end{align*}
and consequently, if $\mathfrak{L}$ and $\Nfrak$ are prime to $p$,
\begin{align*} 
 \lla \boldf |\tau'(\mathfrak{l}), \boldg| T_{\Nfrak/\mathfrak{L}} \rra_{\mathfrak{N}} = & \lla \boldf|\tau'(\nfrak), \boldg \rra_{\mathfrak{L}} .
\end{align*}
Let $\boldf$ be a Hilbert modular form in $\boldM_{k,v}(U_0(2^{-1}\Nfrak p^r,2),\psi,\psi')$ and $g$ be a Hilbert modular form  of half integral weight $l'$ in $\boldM_{l'}(\Nfrak p^r,\psi_2,\C)$, we can now restate Proposition \ref{IntExpre} in the following way
\begin{align*} 
\lla \boldf^c, \mathrm{Tr}_{\Sl}^{\gl}(v,\psi') \left({g}\mathcal{E}\left(z,s; k-l'-\frac{1}{2}t,\chi_{-1}\psi_2\psi{\psi'}^{-2}, \Nfrak\right)\right) \rra_{\Nfrak} = & 
\end{align*}
\begin{align*}
  =& 2^{d}  D_F^{-3/2} {(2\pi)}^{ - d(s-3/4) - \frac{k+l'}{2}} \times  \\
  &\times \G_{\infty}\left(\left(s-\frac{3}{4}\right)t + \frac{k+l'}{2}\right) \mathcal{D}\left(s-\frac{1}{2},\boldf_1,g\right). 
\end{align*}

%%%%%a%%%%%%%%%%%%%%%%%%%%%%%%%%%%%%%%%%%%%%%%%%%%%%%%%%%%%%%%%%%%%%%%%%%%%%%%%%%%%%
%%%%%%%%%%%%%%%%%%%%%%%%%%%%%%%%%%%%%%%%%%%%%%%%%%%%%%%%%%%%%

\section{Arithmetic measures}\label{p-measure}

In this section we recall the notion of an arithmetic measure and construct some of them, in the spirit of \cite{H6,Wu}. We will first construct a many variable measure $\calE^{\chi\chi_{-1}}_c\ast \Theta_{\chi} |[\mathfrak{l^2}] $ which will be used for the construction of the $p$-adic $L$-function in $[F:\Q]+2+\delta$ variable; then, we will construct a one variable measure $E^{\chi,+}_c$ will be used for the construction of the ``improved'' $p$-adic $L$-function (see Section \ref{padicL}). \\
An {\itshape arithmetic measure of half integral weight}  is a measure from a $p$-adic space $V$ on which $\rif_p^{\times}$ acts to the space of $p$-adic modular forms of half-integral weight which satisfies certain conditions (cfr. \cite[\S 4 ]{H6}).\\
More precisely, an arithmetic measure of half integral weight is a $\oo$-linear map $\mu: \Cc(V,\oo) \rightarrow \overline{\boldM}_{\mathrm{half}}(\Nfrak,\oo)$ such that
\begin{itemize}
	\item[A.1] There exist a non-negative integer $\kappa$ such that for all $\phi \in \Ll\Cc(V,\oo)$, there exists an integer $r$ such that 
	\begin{align*}
	 \mu(\phi) \in &  \boldM_{\left(\kappa + \frac{1}{2}\right)t}(\Nfrak p^r, \overline{\oo})
	\end{align*}
	\item[A.2] There is a finite order character $\psi : \rif_p^{\times} \rightarrow \oo^{\times}$ such that for the action $|$ of $\rif_p^{\times}$ defined in Section \ref{halfp}, 
	\begin{align*}
	 \mu(\phi)| b = {b}^{\kappa t + \frac{1}{2}t} \psi(b) \mu(\phi|b),
	\end{align*}
	where $ \phi|b(v) =  \phi(b^{-1}v)$.
	\item[A.3] There is a continuous function $\nu:V \rightarrow \oo$ such that
	\begin{align*}
            (\nu | b) (v) = b^{2t} \nu(v) \mbox{ and } & \textup{d}^t (\mu(\phi)) = \mu(\nu \phi). 
         \end{align*}	
\end{itemize}
	
We say that such a measure is {\itshape supersingular} if $ \iota_p \mu = \mu$ and {\itshape cuspidal} if $\mu$ has values in $\overline{\boldS}_{\mathrm{half}}(\Nfrak,\oo)$. 
Under some hypotheses, such as Leopodt's conjecture (but even under weaker hypotheses, cfr. \cite[(8.2)]{Hr3}), it is possible to show that supersingular implies cuspidal as in \cite[Lemma 4.1]{H6}.\\
We can define an arithmetic measure of half-integral weight for $\clpinf$ after obvious changes in the properties {\it A.1 - A.3}. Note that the action of $\rif_p^{\times}$ is trivial on the closure of global units. \\
Before giving some examples, we recall that we have the following theorem on the existence of $p$-adic $L$-function for Hecke character of a totally real field \cite{DelRib}
\begin{theo}\label{Ribet}
Let $\chi$ be a primitive character of finite order of conductor $\cfrak$. For all $c \in \clpinf$, we have a measure $\zeta_{\chi,c}$ on $\clpinf$ such that
\begin{align*} 
\int_{\clpinf} \psi(z){\lla \N_p(z)\rra}^n \textup{d}\zeta_{\chi,c}(z) = & (1- \chi\psi(c){\N(c)}^{n+1}) \prod_{\pfrak|p}(1-(\psi\chi)_0(\pfrak){\N(\pfrak)}^{n}) L(-n,(\psi\chi)_0). 
\end{align*}
for all $n \geq 0$ and for all finite order character $\psi$, where $(\chi\psi)_0$ denotes the primitive character associated with $\chi\psi$.
\end{theo}
Here $\N_p$ is the $p$-adic cyclotomic character. As in the case of Kubota-Leopoldt $p$-adic $L$-functions, if $\chi$ is odd, then the above measure is $0$ for all $n$. In the sequel, fix $c \in \clpinf$ such that $\lla \N_p(c) \rra$ generates the free part of the $p^{\infty}$-cyclotomic extension of $F$.\\
It is possible to interpolate the values of the imprimitive $L$-function. What we have to do is to remove the factor at prime ideals $\qfrak$, for $\qfrak$ prime to $p$ and ranging in a fixed finite set of prime ideals. We multiply the measure $\zeta_{\chi,c}$ by the factor 
\begin{align*}
(1-(\psi\chi)_0(\qfrak){A_{ \N_p(\qfrak) }(X)})
\end{align*}
for $\qfrak$ in this fixed set. Here $A_z(X)$ denotes the formal power series ${(1+X)}^{\log_p(z)/\log_p(u)}$.\\
Let us explain how this theorem implies that the Eisenstein series of Theorem \ref{EisFou} can be $p$-adically interpolated, by interpolating their Fourier coefficients as $p$-adic analytic functions.\\
Fix an integral ideal $\Mfrak$ of $F$ prime to $p$ and divisible by $4$. Up to enlarging $\Mfrak$, we can suppose that it is the square of a principal ideal. Call $\afrak_p$ and $\mathfrak{b}_p$ the image of $\afrak$ and $\mathfrak{b}$ in $\N_p(\clpinfm)$.\\ 
The first {\it Eisenstein measure} (of level $\Mfrak$) which we define is

\begin{align*} 
\int_{\clpinfm} \psi(z) \textup{d}E^{\chi}_{ss,c}(z) = & \sum_{\tiny{\begin{array}{c}
                                            0 \ll \xi \in \rif , \\
                                              (\xi,p)=\oo
                                             \end{array}}}  q^{\xi/2}  \sum_{\tiny{\begin{array}{c}
                                             \afrak^2\mathfrak{b}^2 | \xi\\
                                             \afrak, \mathfrak{b} \mbox{ prime to } p\Mfrak
                                             \end{array}}}
\mu(\afrak) \omega_{\xi}(\afrak)\N(\mathfrak{b}) \int_{\clpinfm} \psi|(\afrak_p\mathfrak{b}^2_p) \textup{d}\zeta_{\chi\omega_{\xi},c}
\end{align*}
Here the action of $z$ in $\clpinfm$ on $\Cc(\clpinfm,\oo)$ is given by $\phi|z(v)=\phi(z^{-1}v)$.\\
We see that for a character $\psi$ of $\clpinfm$ of finite order of $p$-conductor $p^{\alpha}$, $\alpha={(\alpha_{\pfrak})}_{\pfrak \mid p}$, with $\alpha_{\pfrak} \geq 0 $ for all prime ideal $\pfrak$,   we have 

\begin{align*}
 E^{\chi}_{ss,c} ({\psi(z)\N(z)}^{m-1}) = &(1- \omega^{-n}\chi\psi(c){\N(c)}^{m}) A_0^{-1} \times \\
  & \times {\calE '}\left(z,\frac{1-m}{2};mt,\psi\chi,\Mfrak^2 p^{2s}\right)  |[\mfrak^2\varpi^{2\alpha}4^{-1}]|\iota_p \\
 & \in \boldM_{mt+\frac{1}{2}t}(\Mfrak^2 p^{2\alpha},\chi^{-1}\psi^{-1}),
\end{align*} 
for $A_0=  i^{md}{\pi}^{d}  {2}^{(m-\frac{1}{2})d } \psi(\dfrak  \nfrak 2^{-1})D_F{\N(\dfrak\nfrak 2^{-1})}^{m-2}$ of Proposition \ref{Eisen_other2}.\\
Here $\varpi$ is a product of fixed uniformizers at $\pfrak_i$ for $\pfrak_i |p$. In general, we shall use the notation $p^{\alpha}= \prod_{i=1}^e \pfrak_i^{\alpha_i}$, with $\alpha = (\alpha_i)$ and $\pfrak_1, \ldots, \pfrak_e$ the divisors of $p$ in $\rif$.\\
%We remark that if Leopoldt's conjecture for $F$ and $p$ is false, then the measure that we have just defined would be well defined only on the part of $\clpinfm$ corresponding to the cyclotomic extension.\\ 
We define then the measure $\mathcal{E}_c^{\chi}$ on $\boldG=\clpinfm \times \rif_p^{\times}$; let $\psi$ be a finite order character of $\clpinfm$ and $\psi'$ a finite order character of $\rif_p^{\times}$ which we supposed induced by a finite order Hecke character of $F$ which we denote by the same symbol. We define for the function $a^{v}\psi'(a)\psi(z){\N_p(z)}^{m}$
%We see $\psi$, resp. ${\psi'}^{-1}$,  as character of the torus $\mathbb{T}_{Z}$, resp. $\mathbb{T}_{\alpha}$, defined in Section \ref{halpf} via the map (\ref{GinGL2}).  

\begin{align*}
\int_{\boldG} \psi(z){\N_p(z)}^{m}a^{v}\psi'(a) \textup{d}\mathcal{E}_c^{\chi}(z,a) = & \textup{d}^{-v} \left(\int_{\clpinf} \psi(z){\psi'}^{-2}(z){\N_p(z)}^m \textup{d}E^{\chi}_{ss,c} \right).
\end{align*}
Note that such a functions are dense between the continuous function on $\boldG$. We have that $\mathcal{E}^{\chi}_c$  is an arithmetic measure of half-integral weight. In fact, it verifies {\it A.1} with $\kappa = 1$.\\
For {\it A.2}, we define an action of $b \in \rif_p^{\times}$  on $\Cc(\boldG,\oo)$ as $\phi_1(a)\phi_2(z)| b = \phi_1(a{b}^{-2})\phi_2(zb)$. To show that 
\begin{align*}
 \mathcal{E}_c^{\chi}(\phi_1,\phi_2) |b = \mathcal{E}_c^{\chi}((\phi_1,\phi_2)|b)
\end{align*}
it is enough to check the formula on functions of type $a^{v}\psi'(a)\psi(z)$ with $\psi$ and $\psi'$ characters of finite order. Notice that $(\textup{d}^{\sigma}f)| b ={b}^{2\sigma}\textup{d}^{\sigma}(f|b)$; we have then
\begin{align*}
 \mathcal{E}_c^{\chi}(a^{v}\psi'(a)\psi(z)) |b = & (\textup{d}^{-v}E_{ss,c}^{\chi}({\psi'}^{-2}\psi(z)))|b \\
  = & A' {b}^{-2v} \textup{d}^{-v}({\calE}'\left(z,0; t,\psi{\psi'}^{-2},\Mfrak^2 p^{2\alpha}\right)  |[\varpi^{2\alpha}\mfrak^2 4^{-1}]|\iota_p | b) \\
  = & A' {b}^{-2v} \psi^{-1}{\psi'}^2(b)\textup{d}^{-v}{\calE}'\left(z,0; t,\psi{\psi'}^{-2},\Mfrak^2 p^{2 \alpha}\right)|[\varpi^{2 \alpha}\mfrak^2 4^{-1}]|\iota_p \\
  = & \mathcal{E}_c^{\chi}((a^{v}\psi'(a)\psi(z))|b),
\end{align*}
for $A'=(1- \omega^{-n}\chi\psi(c){\N(c)}^{m}) A_0^{-1}$.\\
Then {\it A.3} is verified by $(a,z)\mapsto a^{-t}$.\\

Furthermore, we define a third Eisenstein measure $E^{\chi,+}_c$ on $\clpinf$. Let $\chi$ be a Hecke character of finite order, we define $E^{\chi,+}_c$ as the restriction on the divisor $D := ( Y +1 - \lla \N_p(c) \rra {(X+1)}^{2})$  of the following measure $\mu$

\begin{align*} 
\int_{\clpinf^2} \phi_1(z_1)\phi_2(z_2) \textup{d}\mu = &(1-\chi(c)\N_p(c)(1+X)) \int_{\clpinf}\phi_2  \textup{d}\zeta_{\chi^2,c} + \\
                                       & + (1-\chi^2(c)\N_p^2(c)(1+Y))\sum_{\tiny{\begin{array}{c}
                                             \xi \gg 0, \\
                                             \xi \in \rif
                                             \end{array}}} q^{\xi/2}  \times \\
                                           & \left(  
                                         \sum_{\everymath{\scriptstyle}
\tiny{\begin{array}{c}
                                             \afrak^2\mathfrak{b}^2 | \sigma\\
                                             (\afrak, \mathfrak{b},  p\cfrak)=1
                                             \end{array}}}
\mu(\afrak) \omega_{\xi}(\afrak)\N(\mathfrak{b}) \int_{\clpinf} \phi_1|\afrak_p\mathfrak{b}_p^2 \textup{d}\zeta_{\chi\omega_{\xi},c} \right).
\end{align*}

We define now a { \it theta measure}; fix a Hecke character $\chi$ of level $\cfrak$, we pose
$$
\begin{array}{cccc}
\Theta_{\chi}: & \Cc(\clpinfc,\oo) & \rightarrow & \overline{\boldM}_{\mathrm{half}}(4\cfrak^2,\oo) \\
&  \eps & \mapsto &\sum_{\xi \gg 0 }\chi\eps(\xi){\N(\xi)}^{\alpha} q^{\frac{\xi^2}{2}} ,
\end{array}
$$
where  $\alpha \in \lgr 0,1 \rgr$ and  $\eps$ a character of finite order of conductor $p^s$, such that $\chi\eps(-1)={(-1)}^{\alpha}$. We have seen that $\Theta_{\chi}(\eps) \in \boldS_{\frac{t}{2}+ \alpha t }(4 \cfrak^2 p^{2s} , \chi\eps)$. \\
We define now the convolution of two measures of half-integral weight; let $\mu_1 $ and $\mu_2$ be two measures, defined respectively on $\boldG$ and $\clpinf$, the {\it convolution} $\mu_1 \ast \mu_2$ is defined by
\begin{align*}
 \int_{\clpinf}\left(\int_{\boldG} \Phi(z^{-1} z_1)(z,a)\textup{d}\mu_1\right)\textup{d}\mu_2, 
\end{align*}
for $z_1 \in \clpinf$, $(z,a) \in \boldG = \clpinf \times \rif_p^{\times}$ and $\Phi: \clpinf^* \rightarrow \Cc(\boldG,\oo) $ a continuous  morphism. Here $\phantom{e}^*$ denotes the $\oo$-dual. If we let $\boldG$ acts on $\clpinf \times \boldG$ via 
\begin{align*} 
 (z_0,a_0)(z_1,z,a)= (z_1 z_0^{-1},zz_0, a a_0^{-1} ),
\end{align*}
  we see that this action is compatible with the action of $\rif_p^{\times}$ on $\Cc(\boldG,\oo)$  defined above if we send 
\begin{align*}
 \rif_p^{\times} \ni b \mapsto (b^{-1},b^2) \in \clpinf \times  \rif_p^{\times}.
\end{align*}
 
Let $\chi$ be a character of conductor $\cfrak$ and $\Nfrak$ a fixed ideal (which in what follows will be the level of our Hilbert modular form), and pose $\Mfrak=lcm(4,\cfrak^2,\Nfrak^2)$ Let us write $\Lfrak^2=\frac{\Mfrak}{4 \cfrak^2}$ and let $\lfrak$ be an id\`ele representing $\Lfrak$. We take $\mu_1=\calE^{\chi\chi_{-1}}_c$ (of level $\Mfrak$) and $\mu_2=\Theta_{\chi}|[\mathfrak{l^2}]$. Recall that $\chi_{-1}$ is the character corresponding to the extension $F(i)$ defined in Section \ref{HIntweight}.  We have then that $\calE^{\chi\chi_{-1}}_c\ast \Theta_{\chi} |[\mathfrak{l^2}] $ is a measure which takes values in $\overline{\boldM}_{k} (\G^1[\Mfrak],\oo)$.

We compose now with $ \mathrm{Tr}_{\Sl}^{\gl}$ and we see that such an action is compatible with the action of $\boldG$ on $\overline{\boldm}(\Mfrak(p^{\infty}),\oo)$, so what we have just constructed is a morphism of $\oo[[\boldG]]$-modules.

%To explain the choice of the action: a acts via T(a^{-1},1) which correspond to the action $(a^{ -1},0,0,1)_{\infty}$
%%%%%%%%%%%%%%%%%%%%%%%%%%%%%%%%%%%%%%%%%%%%%%%%%%%%%%%%%%%%%%%%%%%%%%%%%%%%%%%%%%%%%%%%%%%%%%%%%%%%%%%%%%%%%%%%%%%%%%%%%%%%%%%%%%%%%%%%%%%%%%%%%%%%%%%%%%%%%%%%%%%

\section{Some $p$-adic $L$-functions}\label{padicL}
In this section we construct two $p$-adic $L$-functions for the symmetric square (Theorem \ref{T1}); to do this, we first recall the definition of the {\it $p$-adic Petersson product} $l_{\lambda}$  which is a key tool in the construction of $p$-adic $L$-function {\it \`a la Hida}.\\
 Let us denote by $\delta$ Leopoldt's defect for $F$ and $p$; the first $p$-adic $L$-function $\Ll_p(Q,P)$ has $[F:\Q]+2+2\delta$ and  to construct it we use the method of Hida \cite{H6} as generalized by Wu \cite{Wu}. We improve the result of Wu (which requires the strict class number of $F$ to be equal to $1$) and correct some minor errors.\\ 
We construct also a one variable $p$-adic $L$-function $\Ll^+_p(Q,P)$ which we call the ``improved'' $p$-adic $L$-function. It is constructed similarly to \cite{HT}; instead of considering the convolution of measures, we multiply the measure $E^{\chi,+}_c$ of the previous section by a fixed theta series $\theta(\chi)$. By doing so, we lose the cyclotomic variable but in return we have that when $\chi$ is not of conductor divisible by $p$ we do not have to consider theta series of level divisble by $p$. As a consequence, in the notation of Theorem \ref{T1}, the factor $E_1(Q,P)$ (whose origin has been explained in \ref{thetashift}) does not appear. These two $p$-adic $L$-functions are related by Corollary \ref{Coroll} which is the key for the proof of Theorem \ref{MainTh}; indeed, $E_1(Q,P)$ is exactly the Euler factor which brings the trivial zero for $\boldf$ as in Theorem \ref{MainTh}.\\
We recall that we defined $\LL=\oo[[\boldW]]$, for $\boldW$ the free part of $\boldG$. Let $\boldI$ be a finite, integrally closed extension of $\LL$ and let $\boldF$ be a family of nearly-ordinary forms which corresponds by duality to a morphism $\lambda:\boldh^{\mathrm{n.ord}}(\Nfrak,\oo) \rightarrow \boldI$  as in Section \ref{ptheory}. We suppose that $\lambda$ is associated with a family of $\Nfrak$-new forms. In the follow, we will denote by $\boldh^{\mathrm{ord}}$, resp. $\boldI^{\mathrm{ord}}$, $\boldF^{\mathrm{ord}}$ the ordinary part of $\boldh^{\mathrm{n.ord}}$, resp. $\boldI$, $\boldF$, i.e. the specialization at $v=0$.\\
Such a morphism $\lambda$ induces two finite order characters $\psi$, resp. $\psi'$ of the torsion part of $\clpinf$, resp. $\rif_p^{\times}$. We define, following \cite[\S 9]{Hr3}, the congruence ideal of $\lambda$.  By abuse of notation, we denote again by $\lambda$ the following morphism
\begin{align*}
\lambda : \boldh^{\mathrm{n.ord}}(\Nfrak,\oo)  \otimes_{\LL} \boldI \rightarrow \boldI
\end{align*}
which is the composition of multiplication $\boldI \otimes \boldI \rightarrow \boldI$ and  $\lambda \otimes \mathrm{id}_{\boldI}$.\\
Let $\boldK$ be the field of fraction of $\boldI$; such a morphism $\lambda$ induce a decomposition of $\LL$-algebra
\begin{align*}
 \boldh^{\mathrm{n.ord}}(\Nfrak,\oo)  \otimes_{\LL} \boldK \cong \boldK \oplus \boldB,
\end{align*}
where the projection on $\boldK$ is induced by $\lambda$ and $\boldB$ is a complement. Let us call $1_{\lambda}$ the idempotent corresponding to the projection onto $\boldK$ and $\mathrm{proj}$ the projection on $\boldB$. We define the congruence ideal  
\begin{align*}
C(\lambda) := & (\boldI \oplus \mathrm{proj}(\boldh^{\mathrm{n.ord}}(\Nfrak,\oo)  \otimes_{\LL} \boldI) ) / (\boldh^{\mathrm{n.ord}}(\Nfrak,\oo)  \otimes_{\LL} \boldI). 
\end{align*} 
We know that $C(\lambda)$ is a $\boldI$-module of torsion and we fix an element $H$ of $\boldI$ such that $H C(\lambda)=0$.\\
Let us define $\overline{\boldS}(\Nfrak p^{\infty},\psi,\psi',\oo)$ as the part of $\overline{\boldS}(\Nfrak p^{\infty},\oo)$ on which the torsion of $\boldG$ acts via the finite order characters $\psi$ and $\psi'$ defined as above. Write $\hat{\boldI}= \mathrm{Hom}_{\LL}(\boldI,\LL)$; we have a $\oo$-linear form 
\begin{align*}
l_{\lambda}: \overline{\boldS}(\Nfrak p^{\infty},\psi,\psi',\oo)\otimes_{\LL} \hat{\boldI} \rightarrow \oo .
\end{align*}
Let us denote by $X(\boldI)$ the subset of arithmetic points of $\mathrm{Spec}(\boldI)$; let $P$ be in $X(\boldI)$ such that $P$ resticted to $\LL$ is of type $(m_P,v_P,\eps_P,\eps'_P)$. Let $p^{\alpha}$ be the smallest ideal divisible by the conductors of $\eps_P$ and $\eps'_P$. The point $P$ corresponds by duality  to a form $\boldf_P$ of weight $(k_P=(m_P+2)t -2v_P=\sum_{\sigma \in I}k_{P,\sigma}\sigma, w_P=t-v_P)$, level $\Nfrak p^{\alpha}$ and character $\psi_P=\psi \omega^{-m}\eps_P$, $\psi'_P(\zeta,w)=\eps'_P(w)\psi'(\zeta)\zeta^{-v}$, where we have decomposed $\rif_p^{\times}=\mu \times \boldW'$ as in the end of section \ref{halfp}, and $a \in \rif_p^{\times}$ corresponds to $(\zeta,w)$. When there is no possibility of confusion, in the follow we will drop the subscript $\phantom{e}_P$.\\
We have an explicit formula for $ l_{\lambda_P}:= P \circ l_{\lambda} $ given by  \cite[Lemma 9.3]{Hr3}
\begin{align*}
 l_{\lambda_P}(\boldg) = & H(P)\frac{\lla \boldf_P^{c}|\tau'(\nfrak\varpi^{\alpha}), \boldg \rra_{\Nfrak p^{\alpha}}}{\lla \boldf_P^{c}|\tau'(\nfrak\varpi^{\alpha}), \boldf_P \rra_{\Nfrak p^{\alpha}}} 
 \end{align*}
for all $\boldg \in \boldS_{k_P,w_P}(\Nfrak p^{\alpha},\psi_P,\psi'_P,\Qb )$. Here $\varpi$ denotes the product over $\pfrak \mid p$ of some fixed uniformizers $\varpi_{\pfrak}$.\\

It is clear that we can extend this morphism in a unique way to $\overline{\boldm}_{k,w}(\Nfrak p^{\infty},\oo)$, using the duality with $\boldH_{k,w}(\Nfrak p^{\infty},\oo)$ \cite[Theorem 3.1]{Hr3} and then projecting onto $\boldh_{k,w}(\Nfrak p^{\infty},\oo)$.
We have 
\begin{lemma}
Let $\boldf_P$ and $\lambda$ as above, let $\boldg \in \boldM_{k_P,w_P}(\Nfrak p^{\beta},\psi_P,\psi_P',\overline{\Q})$ with $\beta \geq \alpha$. Then 
\begin{align*}
l_{\lambda_P}(\boldg) & ={H(P)}{\lambda(T(\varpi^{\beta - \alpha}))}^{-1}  \frac{\lla \boldf_P^{c}|\tau(\nfrak\varpi^{\beta}), \boldg \rra_{\beta}}{\lla \boldf_P^{c}|\tau(\nfrak\varpi^{\alpha}), \boldf_P \rra_{\alpha}}
 \end{align*}
\end{lemma}
\begin{proof} 
This is proven in \cite{Mok} on pages 29-30 in the case $v=0$. If $\alpha = \beta $ it is clear. Otherwise, we write $\boldg ' = \boldg | T_0(p^{\beta - \alpha})$ and proceed as in \cite[Proposition 4.5]{H1} 
\begin{align*}
{H(P)}^{-1}{\lambda(T_0(\varpi^{\beta - \alpha}))}l_{\lambda_P}(\boldg) & = l_{\lambda_P}(\boldg ') \\
               & =  \frac{\lla \boldf_P^{c}|\tau'(\nfrak\varpi^{\alpha}), \boldg' \rra_{{\alpha}}}{\lla \boldf_P^{c}|\tau'(\nfrak\varpi^{\alpha}), \boldf_P \rra_{{\alpha}}}\\
               & =  \frac{\lla \boldf_P^{c}|\tau'(\nfrak\varpi^{\alpha})| T^*_0(p^{\beta - \alpha}), \boldg \rra_{{\beta}}}{\lla \boldf_P^{c}|\tau'(\nfrak\varpi^{\alpha}), \boldf_P \rra_{{\alpha}}}\\
 & =  \lgr \varpi^{-v(\beta - \alpha)} \rgr \frac{\lla \boldf_P^{c}|\tau'(\nfrak\varpi^{\beta}), \boldg \rra_{{\beta}}}{\lla \boldf_P^{c}|\tau'(\nfrak\varpi^{\alpha}), \boldf_P \rra_{{\alpha}}}.
 \end{align*}
\end{proof}
We can now state the main theorem of the section
\begin{theo}\label{T1}
Fix an adelic character $\chi$ of level $\mathfrak{c}$, such that $\chi_{\sigma}(-1)=1$ for all $\sigma|\infty$. We have two $p$-adic $L$-functions $\Ll_p(Q,P)$ in the total ring of fractions of  $\oo[[X ]] \hat{\otimes}  \boldI$ and $\Ll_p^+(P)$ in the fraction field of $\boldI^{ord}$ such that the following interpolation properties hold 
\begin{itemize} 
	\item [i)] for (almost) all arithmetic points $(Q,P)$ of type $(s_Q,\eps_Q; m_P,\eps_P, v_P, \eps_P')$, with $ m_P - k_{P,0}  + 2  \leq s_Q \leq m_P $ (for $k_{P,0}$ equal to the minimum of $k_{P,\sigma}$'s) and such that the $p$-part of the conductor of $\omega^s\eps_Q^{-1}\chi^{-1}\psi_P{\psi_P'}^{-2} $ is $p^{\alpha}$, $\alpha$ positive integer,  the following interpolation formula holds 
\begin{align*} 
\Ll_p(Q,P) = & C_1 E_{1}(Q,P)E_{2}(Q,P)\frac{2^d \Ll(s_Q+1, \boldf_P,\eps_Q^{-1}\omega^{s_Q}\chi^{-1})}{(2\pi)^{ds}\Omega(\boldf_P)},
\end{align*}
\end{itemize}
where the Euler factor $E_1(Q,P)$ and $E_2(Q,P)$ are defined below.\\  
Suppose now that $\chi$ is of conductor not divisible by all $\pfrak | p$. Suppose moreover that $\lambda|_{\clpinf} = \psi^{(p)}\omega^{m_0}$, with $\psi^{(p)}$ of conductor coprime to $p$, then we have a generalized $p$-adic $L$-function $\Ll_p(P)$ in the fraction field of $\boldI^{ord}$ satisfying the following interpolation property
\begin{itemize}
\item [ii)] for (almost) all arithmetic points $P$ of type $(m_P,\eps_P)$, with $ m_P \geq 0$ and $m_P\equiv m_0 \bmod p-1$, the following  formula holds 
\begin{align*} 
\Ll_p(P) = C_1 E'_1(P)E_{2}(P)\frac{\Ll(1, \boldf_P,\chi_0)}{\Omega(\boldf_P)},
\end{align*}
where the Euler factor $E_1'(P)$ and $E_2(P)$ are defined below.
\end{itemize}
\end{theo}
The term $C_1=C_1(Q,P)$ can be found in the proof of the theorem and it is a non-zero  algebraic number;  $\Omega(\boldf_P)$ is the Petersson norm of $\boldf^{\circ}_P$, the primitive form associated with $\boldf_p$, times ${(2 \pi) }^{d-2v}$.\\ 
We want to point out that number of independent variables of $\Ll_p(Q,P)$ is $[F:\Q]+1+\delta$, as  along the lines $s_Q -m_P = c$, for $c$ a fixed integer, this function is constant.\\
In the statement of the theorem we have to exclude a finite number of points which are the zeros of an Iwasawa function interpolating an Euler factor at $2$ for which the interpolation formula, a priori, does not hold. If $2$ divides the conductor of $\boldf_P$ or $\chi$, then this function is identically $1$. If $2$ does not divide neither the conductor of $\boldf_P$ or of $\chi$, then a zero of this Iwasawa function appears for $P$ such that $\boldf_P$ is not primitive at  $p$ and $Q=(m_P,\mathbf{1})$. In the Appendix we will weaken this assumption.\\
The power $ {(2\pi)}^{ds}$ corresponds essentially to the periods of the Tate's motive; this corresponds, from the point of view of $L$-functions, to consider the values at $s+1$ instead that in $1$.\\  
According to Deligne's conjecture \cite[\S 7]{Del}, we would expect as a period for the symmetric square, instead of $\Omega(\boldf_P)$, the product of the plus and minus period associated with $\boldf$ via the Eichler-Harder-Shimura isomorphism (see for example \cite[Definition 2.9]{Dim}); but it is well-known that the Petersson norm of $\boldf_P$ differs from  Deligne's period by $\pi^{d-2v}$ and a non-zero algebraic number (see \cite[Main Theorem]{Yo}). The power $\pi^{2v}$ gives (part of) the factor at infinity of the automorphic $L$-function of $\mathrm{Sym}^2(\boldf_P)$ (see Section \ref{Extrafactors}).   \\
The term $E_1(Q,P)$ and $E_2(Q,P)$ are the Euler factor at $p$ which has to be removed to allow the $p$-adic interpolation of the special values as predicted by Coates and Perrin-Riou \cite{CPR}.\\
The factor $E_2(Q,P)$ contains the Euler factor of the primitive $L$-function which are missing when $\boldf_P$ is not primitive; we define, if $\boldf_P$ is not primitive at $\pfrak_i$ 
\begin{align*}
E_{\pfrak_i}(Q,P) = &(1 - (\chi^{-1}\eps_Q^{-1}\omega^{s}\psi_P)_0(\pfrak_i){\N(\pfrak_i)}^{m-s})\times \\
& \times (1 - (\chi^{-1}\eps_Q^{-1}\omega^{s}\psi_P^2)_0(\pfrak_i) {\lambda(T(\varpi_{\pfrak_i}))}^{-2}{\N(\varpi_i)}^{2m +1 -s}) 
\end{align*}
and if $\boldf_P$ is  primitive at $\pfrak_i$ then $E_{\pfrak_i}(Q,P) =1$.\\
The factor $E_2(Q,P)$ is consequently
\begin{align*}
E_2(Q,P)= & \prod_{\pfrak_i | p} E_{\pfrak_i}(Q,P).
\end{align*}
 The factor $E_1(Q,P)$ comes from the fact that we are using theta series of level divisible by $p$ but whose conductor is not necessarily divisible by $p$. For two points  $Q$ and $P$ as above, let us denote by $p^\beta$ the minimum power of $p$ such that the conductors of all the characters appearing in $Q$ and $P$ and the ideal $\varpi \rif$ divide $p^{\beta}$, and let us denote by $p^{\alpha_0}$ the $p$-part of the conductor of $\chi\eps_Q\omega^{-s}$. We have 
\begin{align*}
E_1(Q,P) = & \lambda(T(\varpi^{2\beta - 2\alpha_0})){\N(\varpi^{\beta - \alpha_0})}^{-(s+1)}\prod_{\pfrak | p}(1-{(\chi\eps_Q\omega^{-s})}_0(\pfrak)\N(\pfrak)^{s}{\lambda(T(\varpi_{\pfrak}))}^{-2}).
\end{align*}
The factor $E_1$ is the Euler factor which gives the trivial zero in the case of Theorem \ref{MainTh} and that we remove in the second part of the above theorem. The factor $E'_1(P)$ which appears in that second part is $\lambda(T(\varpi^{2\alpha - 2\alpha_0})){\N(\varpi^{\alpha - \alpha_0})}^{-2}$.\\
We point out that the factor $E_1$ is in reality the second term of the factor $E_2$ when evaluated at $2m + 3-s$ (so it an Euler factor of the dual motive of $\mathrm{Sym}^2(\boldf)$).\\

This theorem deals with the imprimitive $L$-functions. We will show the analog of  Theorem \ref{T1}, i) in the Appendix \ref{AppB} for the $p$-adic $L$-function interpolating the primitive one. \\%when the family is minimal or when we specialize to a particular modular form $\boldf$.

As an immediate consequences we have the following 
\begin{coro}[Important Corollary]\label{Coroll}
Let $Q_0$ be the point $(0,\mathbf{1})$ of $\mathrm{Spec}(\oo\left[\left[  \clpinf \right]\right])$, then we have the following factorization in the fraction field of $\boldI^{\mathrm{ord}}$: 
\begin{align}\label{fact}
\Ll_p(Q_0,P)=& \prod_{\pfrak_i | p}(1 - {\lambda(T(\varpi_i))}^{-2}) \Ll_p^{+}(P).
\end{align}
\end{coro}
Note that a similar formula could be proved for all fixed $v$. \\
\begin{theo}\label{1var}
Let $\boldf=\boldf_P$ be a nearly ordinary form of Nebentypus $(\psi,\psi')$ and weight $k\geq 2t$ which we decompose as $k = m +2t -2v$. Let $\chi$ be a Hecke character of level $\mathfrak{c}$, such that $\chi_{\sigma}(-1)=1$ for all $\sigma|\infty$. Then we have a a formal series $G(X,\boldf,\chi)$ in $\oo((X))\left[ \frac{1}{p}\right]$  such that for all finite order character $\eps$ of $1+p\Z_p \cong u^{\Z_p}$, of conductor $p^{\alpha_0}$, and  $s \in [m-k_0+ 2, m] $ with $n\equiv s$ which are not a pole of $G(X)$ we have 
\begin{align*}
 G(\eps(u)u^s -1,\boldf,\chi) = & i^{(s+1)d-k} 2^{-d\frac{s+n}{2}}  s! D_F^{s-\frac{3}{2}}\times \\
            & G(\chi\eps\omega^{-s}_0)  \eta^{-1}(\mfrak \varpi^{\alpha} \dfrak 2^{-1})  \chi\eps\omega^{-s}_0(\dfrak\cfrak\varpi^{\alpha_0})  {\N(\mfrak )}^{s} \times\\
            & \frac{ \psi(\dfrak\mfrak\varpi^{\beta}) {\N(\mathfrak{l}) }^{-2[s/2]}    {\N(\varpi)}^{(s+1) \beta -\alpha_0 +\frac{(\alpha -\alpha') m}{2} } }{\lambda(T(\varpi^{2\beta-\alpha'}))  W'(\boldf) S(P) \prod_{\pfrak} \frac{\eta\nu(\dfrak_{\pfrak})}{|\eta\nu(\dfrak_{\pfrak})|} \prod_{J}G(\nu\psi')} \times \\
            & E_1(s+1)E_2(s+1) \frac{ 2^d \Ll(s+1,\boldf,\chi^{-1}\eps^{-1}\omega^{s})}{{(2 \pi )^{ds}}\Omega(\boldf)}.
\end{align*}
for $\eta=\omega^{s}\eps^{-1}\chi_{-1}^{-1}\chi^{-1}\psi{\psi'}^{-2}$.
\end{theo}
Here $\beta=\alpha_0$ if $\eps$ is not trivial, and $\beta =1$ otherwise. The factors $E_1(s+1)$ and  $E_2(s+1)$ are the factors $E_1(Q,P)$ and $E_2(Q,P)$ above, for $P$ such that $\boldf_P = \boldf$ and $Q$ of type $(s,\eps)$. The factor $S(P)$ is defined below.\\
We explain the $p$-part of the fudge factor.
Fix a prime $\pfrak$ dividing $p$ and let $\pi_{\pfrak}$ be the local component of $\pi(\boldf)$, as in Section \ref{Extrafactors}.  Then $\pi_{\pfrak}=\pi(\eta,\nu)$ or $\pi_{\pfrak}=\sigma(\eta,\nu)$, according to the fact that $\pi_{\pfrak}$ is a principal series or a special representation. In particular we have $\boldf | T(\varpi_{\pfrak})= \eta(\varpi_{\pfrak}) \boldf $.\\ 
Let $\psi'_{\pfrak}$ be the local component at $\pfrak$ of $\psi'$. We know that $\eta\psi'_{\pfrak}$ is unramified. Let us denote by $\varpi_{\pfrak}^{\alpha'_{\pfrak}}$ the conductor of $\nu\psi'_{\pfrak}$. Then $p^{\alpha'}$ is the $p$-part of the conductor of the representation $\pi(\boldf) \otimes \psi'$.

The factor $S(P)$ is defined in \cite[page 355]{Hr3} as a product over $\pfrak | p$ and each factor depends only on the local representation at $\pfrak$; if $\pi_{\pfrak}$ is special the factor  is $-1$, if $\pi_{\pfrak}$ is a ramified principal series, the factor  is $\eta^{-1}\nu(\varpi_{\pfrak}^{\alpha'_{\pfrak}}) {|\varpi_{\pfrak}^{\alpha'_{\pfrak}}|}_{\pfrak}^{-1}$ and finally  if $\pi_{\pfrak}$ is unramified the factor is $$(1 - \eta^{-1}\nu(\varpi_{\pfrak}) {|\varpi_{\pfrak}|}_{\pfrak}^{-1})(1 - \eta^{-1}\nu(\varpi_{\pfrak}) ).$$ 
Moreover, let us denote by $\boldf^u$ the unitarization of $\boldf$,  as defined in Section \ref{ComplexHilbert}, and if we suppose $\boldf$ primitive of conductor $\Mfrak p^{\alpha}$ then we know that $\boldf^u | \tau'(\mfrak \varpi^{\alpha})= W(\boldf)\boldf^{u,c}$,  for  an algebraic number $W(\boldf)$ of complex absolute value equal to $1$ (here $c$ stands for complex conjugation). Moreover $W(\boldf)$ can be written as a product of local $W_{\qfrak}(\boldf)$, and we write it as $W(\boldf)=W'(\boldf)\prod_{\pfrak \mid p} W_{\pfrak}(\boldf)$. For the explicit expression of  $W_{\pfrak}(\boldf)$, we refer to \cite[\S 4]{Hr3}. From the formulae in {\it loc. cit.} we see that the factors $\frac{\eta\nu(\dfrak_{\pfrak})}{|\eta\nu(\dfrak_{\pfrak})|}$ and $G(\nu\psi')$ come from $W_{\pfrak}$. The remaining part of $W_{\pfrak}$ is incorporated in the Hecke eigenvalue $\lambda(T(\varpi^{\alpha'}))$. \\
We warn the reader that the proof of Theorem \ref{T1} is  computational, and the reader is advised to skip it on a first reading. We recall in the following lemma some well-known results  which we will use when evaluating the $p$-adic $L$-function; recall from Section \ref{EisSer} that we have defined $H$, the {\it holomorphic projector}, $\partial_p^q$, the Mass-Shimura differential operator, for $\sigma \in I$, $\textup{d}^{\sigma}$, the holomorphic differential operator on $\h^{I}$, and $c$, the {\it constant term projection} of a nearly-holomorphic modular form.
\begin{lemma}\label{lemma}
Let $f \in \boldS_{k}(\G^1[\Nfrak])$, $g \in \boldN^{s_1}_{l + \frac{t}{2}}(\G^1[\Nfrak]) $ and $h \in \boldN^{s_2}_{m+ \frac{t}{2}}(\G^1[\Nfrak,\afrak])$ with $k,l,m \in \Z[I]$ and  $l > 2s_1$, $m > 2s_2$ (recall that these two condition are automatic if $F\neq \Q$). Let $r \in \Z[I]$.\\ 
Then we have
\begin{align*}
\lla f, H(g) \rra = & \lla f, g \rra,\\
H(g\partial_{m + \frac{t}{2}}^{r}h)= & {(-1)}^r H(\partial_{l+ \frac{t}{2}}^r g h), \\
eH(g\partial_{m+ \frac{t}{2}}^{r} h) = & e (g \textup{d}^r h).
\end{align*} 
\end{lemma}
\begin{proof}% No pub
The first formula is \cite[Lemma 4.11]{ShH6}.\\
The second formula can be proven exactly as \cite[Proposition 7.2]{Hr3}.\\
The last one can be shown as in \cite[Proposition 7.3]{Hr3}, writing $\partial_{m+ \frac{t}{2}}$ in term of $\textup{d}^{j}$ for $0 \leq j \leq r$ (in $\Z[I]$) and noticing that $e(\textup{d}^{\sigma} f)=0$ for all $\sigma$ in $I$.
\end{proof}
%%%%%%
\begin{proof}[Proof of Theorem \ref{T1} i)]
We pose $\Mfrak=lcm(4,\cfrak^2,\Nfrak^2)=4 \cfrak^2 \Lfrak^2$ as in the end of Section \ref{p-measure}. Fix two id\`eles $\mfrak$ and $\mathfrak{l}$  such that $\mfrak^2$ represents $\Mfrak$ and $\mathfrak{l}$ represents $\Lfrak$. We suppose moreover that $\mfrak\lfrak^{-1}$ represents $\cfrak$. To lighten the following formulae, let us pose 
\begin{align*}
 \mathrm{Tr}_1 = T_{\Mfrak/\Nfrak} \,\, ,  \mathrm{Tr}_2= Tr_{\Sl}^{\gl}. 
\end{align*}
For each $\qfrak$ dividing $2$ in $\rif$, we pose 
\begin{align*}
\calE_{\qfrak}(X,Y)=& (1-\psi^2\chi^{-2}(\qfrak)A_{\N(\qfrak)}^{-2}(X) A_{\N(\qfrak)}^2(Y)), \\
\calE_{2}(X,Y)=& \prod_{\qfrak |2} \calE_{\qfrak}(X,Y),
\end{align*}
where $X$ is a variable on the free part of $\Z_p^{\times}$ in $\clpinf$ (it corresponds to the variable $Q$) and $Y$ on the free part of $\Z_p^{\times}$ embedded in the fist component of $\boldG=\clpinfm \times \rif_p^{\times}$ (it corresponds to the variable $m$ of $P$).\\
As in the end of Section \ref{halfp}, we have $A_{z}(X)={(1+X)}^{\log_p(z)/\log_p(u)}$, for $u$ a fixed topological generator of $1+p\Z_p$.\\
We pose moreover $\Delta(X,Y)=\left(1-{\psi'}^{-2}\chi_{-1}\chi^{-1}\psi(c)\N(c)\frac{1+Y}{1+X}\right)$ where, as in Section \ref{p-measure}, $c$ is chosen such that $\lla \N(c) \rra$ correspond to the generator $u$ fixed above.  \\ 
We define a first $p$-adic $L$-function 
$$ \Ll_p = {(\Delta(X,Y) \calE_2(X,Y)H)}^{-1}l_{\lambda} e  \mathrm{Tr}_1 ((( \mathrm{Tr}_{2} (\calE^{\chi\chi_{-1}}_c \ast \Theta_{\chi} |[\mathfrak{l^2}]) |) )\Xi_2) $$
where $\Xi_2 = 
{\left( \begin{array}{cc} 
2 & 0\\
0 & 1 \end{array} \right)}_f $. \\
The  operator $\Xi_2$ is necessary as the family $\boldF$ is of level $U(\Nfrak)$ while our convolution of measures  is of level $U(2^{-1} \Mfrak ,2)$. %No art

This $p$-adic $L$-function is an element of the total fraction field of $\oo[[\clpinf ]] \hat{\otimes}  \boldI$. In fact, $\boldI$ is a reflexive $\LL$-module as it is integrally closed and finite over $\LL$.
We want to evaluate $\Ll_p$ on an arithmetic point $(Q,P) \in X(\clpinf) \times X (\boldI)$ of type $( s,\eps_Q; m,\eps_P, v, \eps_P')$ with $s \leq m$. We let $\psi_P$ and $\psi'_P$ as in the beginning of this section. Let $n \in \lgr 0,1 \rgr $, $n \equiv s \bmod 2$, we evaluate
\begin{align*}
 g = & ({\Delta(X,Y)}^{-1}\calE^{\chi\chi_{-1}}\ast \Theta_{\chi} |[\mathfrak{l^2}])(Q,P)\\
   = & {\Delta(Q,P)}^{-1}\left(\left(\int_{\clpinf } \eps_Q\omega^{-s}(z_1){\N_p(z_1) }^{s} \textup{d}\Theta_{\chi}(z_1)\right)|\left[\mathfrak{l^2} \right] \times \right.\\ 
 & \times \left. \int_{\boldG} \eps_Q^{-1} \omega^{s}\psi_P(z){\lla \N_p(z) \rra}^{m-s} \psi_P'(a) a^{v} \textup{d}\calE^{\chi\chi_{-1}}_c(z,a)\right) \\
   = & A_0^{-1}{\N(\mathfrak{l}) }^{-2[s/2]} \left( \textup{d}^{[s/2]t}(\theta_n(\chi\eps_Q\omega^{-s})|\left[\mathfrak{l^2} \right])\right. \times \\ 
     & \left. \times \textup{d}^{-v} \left( {\calE}'\left(z,\frac{s-m}{2};s_0 t,\eta\right)|[\mfrak^2 \varpi^{2 \beta} 4^{-1}] \right) \right)  
\end{align*}

for $s_0 = {m-s+1}$, $\eta = \omega^s\eps_Q^{-1}\chi_{-1}^{-1}\chi^{-1}\psi_P{\psi_P'}^{-2} $, $\beta$ such that the conductors of $\eps_Q$, $\eps_P$ and $\eps'_P$ and the level of $\boldf_P$ divide $p^{\beta}$  and, we recall,
\begin{align*}
 A_0=\eta(\mfrak \varpi^{\alpha} \dfrak 2^{-1})D_F{\N({\mfrak} \varpi^{ \beta}\dfrak 2^{-1})}^{m-s-1}i^{s_0 d}\pi^d 2^{d\left(s_0 - \frac{1}{2}\right)} .
\end{align*}
For an integral id\`ele $y$ we use, by abuse of notation, the expression $\N(y)$ to denote the norm of the corresponding ideal $y\rif$.\\
Here we have used the relation $\textup{d}^{t}[\lfrak^2]={\N(\mathfrak{l})}^2[\lfrak^2]\textup{d}^{t}$. We use Lemma \ref{lemma}, formula (\ref{Massgamma}), the property that $e$ commutes with the operator $\I$ and $\mathrm{Tr}_2$, as seen at the end of Section \ref{halfp}, and the formula  (\ref{deltaEis}) applied two times  to obtain 
%\begin{align*}
%{\calE }(z,\frac{s-m}{2};(m -n +1) t,\chi^{-1}\psi_P, {N_0}^2p^{2r})  = & {(4\pi)}^{[s/2]t}{(-1)^{d[s/2]}([s/2])!)}^{-d} \partial_{(s_0 + 1/2)t}^{[s/2]t} {\calE}(z,\frac{s-m}{2};s_0 t,\chi^{-1}\psi_P, N_0^2 p^{2r})  \\ 
% {\calE}(z,\frac{s-m}{2};(m -n +1) t,\chi^{-1}\psi_P, {N_0}^2 p^{2r})  = &  (4\pi)^v (-1)^v { \prod \frac{([s/2] -v_{\sigma})!}{([s/2])!} } \\
% & \times   \partial_{(m-n+3/2)t}^v {\calE}(z,\frac{s-m}{2};k-n-t,\chi^{-1}\psi_P, N_0^2p^{2r})
%t
%\end{align*} 

\begin{align*}
 \Ll_p(Q,P) = & \frac{\lla \boldf_P^{c}|\tau'(\nfrak \varpi^{2 \alpha}),  e  \mathrm{Tr}_1 ( \mathrm{Tr}_{2} (g) |\Xi_2) \rra_{\Nfrak p^{2\alpha}}} {\lla \boldf_P^{c}|\tau'(\nfrak\varpi^{\alpha}), \boldf_P \rra_{\Nfrak p^{\alpha}}}, \\
 =  & C \frac{\lla \boldf_{P}^{c}|\tau(\mfrak^2 \varpi^{2 \beta }),  \mathrm{Tr}_{2} (g) |\Xi_2\rra_{\Mfrak p^{2\beta}}}{\lla \boldf_P^{c}|\tau'(\nfrak\varpi^{\alpha}), \boldf_P \rra_{\Nfrak p^{\alpha}}}, \\ % No pub 
 = & C  \frac{\lla \boldf_{P,1}^{c}|\Xi_2^{-1}, h|\tau(\mfrak^2 \varpi^{2 \beta } ) \rra_{\Mfrak p^{2\beta}}}{\lla \boldf_P^{c}|\tau'(\nfrak\varpi^{\alpha}), \boldf_P \rra_{\Nfrak p^{\alpha}}},
\end{align*}
for % the -1^{[s/2]} of Maas-Shimura e di quando si scambia \textup{d} grazie a $e$ si cancellano
\begin{align*}
h = & \theta_{n}(\chi\eps\omega^{-s})|\left[\mathfrak{l}^2\right] \times \calE\left(z,\frac{s-m}{2};k-(n+1)t,\eta\right)|\tau(\mfrak \varpi^{2\beta}),\\
C = & \psi_P(\mfrak \varpi^{\beta}\dfrak) \G_{\infty}(([s/2]+1)t-v)2^{ 2v - sd+nd -ds_0 +\frac{1}{2}d}  {\pi}^{ v- [s/2]d -d } i^{(s+1)d-k}\times \\
 &  \eta^{-1}( \mfrak \varpi^{\alpha} \dfrak 2^{-1}){\N(\mathfrak{l}) }^{-2[s/2]}{\N({\mfrak} \varpi^{ \beta} 2^{-1})}^{1/2}{\N(\mfrak \varpi^{ \beta} \dfrak 2^{-1})}^{s}{\lambda(T(\varpi^{\alpha -2\beta}))}
\end{align*}
 where we used the formula for the change of variable for the Petersson product and the relation
\begin{align*}
 \Xi_2 \tau'(\mfrak^2)\Xi_2^{-1} =  \left(
\begin{array}{cc}
 0 & -2\\
 2^{-1} \dfrak^2\mfrak^2 & 0
 \end{array}
\right)_f.
\end{align*}

We apply the duplication formula of the $\G$ function to the factors $\G_{\infty}\left( \left(\left[ \frac{s}{2}\right] + n + \frac{1}{2}\right)t - v \right)$ coming from Propostion \ref{IntExpre}  and to the $\G$ factor appearing in the constant $C$ above, to obtain, for $z=\left[\frac{s}{2}\right] - v_{\sigma} +\frac{n+1}{2}$, that
\begin{align*}
 \G(z)\G(z+1/2)= 2^{2v_{\sigma} - s}(s - 2v_{\sigma})! \sqrt{\pi}.
\end{align*}

We are left to evaluate $\theta_{n}(\chi\eps\omega^{-s})|\left[\mathfrak{l}^2\right]|\tau(\mfrak^2 \varpi^{2 \beta} ) $. Let us denote by $\chi'$ the primitive character associated with $\chi\eps\omega^{-s}$ and let us denote by $p^{\alpha_0}$ the $p$-part of the conductor of $\chi'$. Recall the relations given at the end on Section \ref{Atkin}, Proposition \ref{Theta-Atkin} and the formula 
\begin{align*}
\theta_{n}(\chi\eps\omega^{-s}) = \sum_{\efrak \rif \mid p} \mu(\efrak)\chi'(\efrak) \theta_{n}(\chi')|[\efrak^2].
\end{align*}
We have then 

\begin{align*}
 \theta_{n}(\chi')|\left[\mathfrak{l}^2\right]|\tau(\mfrak^2 \varpi^{2 \beta } ) = &  G(\chi'){\N(2^{-1}\mfrak \varpi^{\alpha_0})}^{-1/2}\chi'(\dfrak\cfrak\varpi^{\alpha_0}) {\N(\varpi^{\beta-\alpha_0})}^{1/2}   \times \\
  & \times \sum \mu(\efrak)\chi'(\efrak){\N(\efrak)}^{-1} \theta_{n}({\chi'}^{-1})| \left[\frac{\varpi^{2\beta}}{\varpi^{2\alpha_0}\efrak^2} \right].
\end{align*}
We know from \cite[Lemma 5.3 (vi)]{Hr3} an explicit expression for  $\lla \boldf_P^{c}|\tau'(\nfrak\varpi^{\alpha}), \boldf_P \rra_{\Nfrak\varpi^{\alpha}}$ in terms of the Petersson norm of the primitive form $\boldf^{\circ}_P$ associated with $\boldf_P$ (and $\psi'_P$).\\ 
We refer to the discussion after Theorem \ref{1var} for the notation, we have then

\begin{align*}
  \frac{\lla \boldf_P^{c}|\tau'(\nfrak\varpi^{\alpha}), \boldf_P \rra_{\Nfrak\varpi^{\alpha}}}{\lla \boldf_P^{\circ}, \boldf_P^{\circ} \rra_{\Nfrak\varpi^{\alpha'}}} = & \N(\varpi)^{-(\alpha -\alpha') m /2} {\lambda(T(\varpi^{\alpha -\alpha'}))}\psi_{\infty}(-1)  \times \\ 
& \times W'(\boldf_P) S(P)  \prod_{\pfrak} \frac{\eta\nu(\dfrak_{\pfrak})}{|\eta\nu(\dfrak_{\pfrak})|} \prod_{J}G(\nu\psi'),
\end{align*}
where $J$ is the set of $\pfrak | p$ such that $\pi_{\pfrak}$ is a ramified principal series.\\
% Let us denote the right hand side by $C_p$.\\
We can conclude that
\begin{align*}
\Ll_p (Q, P) = & i^{(s+1)d-k} 2^{1+ 2v -(s+m)d +nd  -d\frac{s+n}{2}} \pi^{2v -d(s+1)} (st-2v)! D_F^{s-3/2}   \times \\
            & G(\chi\eps\omega^{-s}_0)  \eta^{-1}(\mfrak \varpi^{\alpha} \dfrak 2^{-1})  \chi\eps\omega^{-s}_0(\dfrak\cfrak\varpi^{\alpha_0}) \psi_P(\dfrak \mfrak \varpi^{\beta})\times\\
            & \frac{  {\N(\mathfrak{l}) }^{-2[s/2]} {\N(\mfrak)}^{s}  {\N(\varpi)}^{(s+1)\beta -\alpha_0 +\frac{(\alpha -\alpha') m}{2} } }{\lambda(T(\varpi^{2\beta - \alpha'}))  W'(\boldf_P) S(P) \prod_{\pfrak} \frac{\eta\nu(\dfrak_{\pfrak})}{|\eta\nu(\dfrak_{\pfrak})|} \prod_{J}G(\nu\psi')\psi_{\infty}(-1)} \times \\
            & E_1(Q,P)E_2(Q,P) \frac{ \Ll(s+1,\boldf,\chi^{-1}\eps_Q^{-1}\omega^{s})}{2^{d-2v}\lla \boldf_P^{\circ}, \boldf_P^{\circ} \rra_{\Nfrak\varpi^{\alpha'}}}.
\end{align*}

The factors $E_1(Q,P)$ and $E_2(Q,P)$ are the ones defined after \ref{T1}. To obtain the explicit expression of $E_1(Q,P)$ we have used the formula (\ref{thetashift}).
\end{proof}
%%remarco sulle potenze di p e confronto con Hida; se \beta=\alpha_0, il +b si cancella con il - \alpha_0, se \beta=1 \alpha_0 =0, si cancella con E_2
\begin{proof}[Proof of ii)]
We define the improved $p$-adic $L$-function on the ordinary Hecke algebra; let $\boldI^{\mathrm{ord}}$ be the ordinary part of $\boldI$, i.e. the fiber above $v=0$. It is a finite flat extension of $\oo[[\boldW_0]]$, where $\boldW_0$ is the free part of $\clpinf$ corresponding, via class field theory, to the cyclotomic extension.\\ 
We pose $\Mfrak=lcm(4,\cfrak^2,\Nfrak^2)=4 \cfrak^2 \Lfrak^2$ as above. We fix the same two id\`eles $\mfrak$ such that $\mfrak^2$ represents $\Mfrak$ and $\mathfrak{l}$ which represents $\Lfrak$ as before. We pose again
\begin{align*}
 \mathrm{Tr}_1 = T_{\Mfrak/\Nfrak} \,\, ,  \mathrm{Tr}_2= Tr_{\Sl}^{\gl} 
\end{align*}
For each $\qfrak$ dividing $2$ in $\rif$, we pose 
\begin{align*}
\calE_{\qfrak}(Y)=& (1-\psi^2\chi_0^{-2}(\qfrak)A_{\N(\qfrak)}^2(Y))\\
\calE_{2}(Y)=& \prod_{\qfrak |2} E_{\qfrak}(Y),
\end{align*}
where $Y$ is a coordinate  of $\oo[[\boldW_0]]$ (it corresponds to the variable $m$ of the weight).\\
As before we set 
\begin{align*}
 \Delta(Y)= & 1-\chi_{-1}\chi^{-1}\psi(c)\N(c)(1+Y), \\ 
\Delta'(Y)= & 1-\chi(c)^{-2}\psi^2{\N(c)}^{2} \lla \N(c) \rra(1+Y)^{2}.
\end{align*}

We define the improved $p$-adic $L$-function  
\begin{align*}
 \Ll^{+}_p =  {(\Delta(Y) \Delta'(Y) \calE_2(Y)H)}^{-1} l_{\lambda} e \mathrm{Tr}_1 \left(\mathrm{Tr}_{2}  \left(\theta(\chi_0)|\left[ \lfrak^{2}  \right]E^+\right) \right)| \Xi_2     & \in \boldI^{\mathrm{ord}}. 
 \end{align*}
It can be seen that it is an element of $\boldI^{\mathrm{ord}}$ as in \cite[\S 6.2]{Mok} from the duality of \cite[Theorem 3.1]{Hr3}.\\
Let $P$ be an arithmetic point of type $(\eps_P,m)$, such that $\eps_P$ factors through the cyclotomic character. Let $k=m+2$, we have
\begin{align*}
l_{\lambda_P} e \mathrm{Tr}_1 ( \I \circ \mathrm{Tr}_{2}  \theta(\chi_0)|\left[ \lfrak^{2}  \right]E^+ )| \Xi_2   =  & \frac{ \lla \boldf_P^{c}|\tau'(\nfrak \varpi^{ \alpha}),  e  \mathrm{Tr}_1 ( \mathrm{Tr}_{2} g |\Xi_2)   \rra }{\lla \boldf_P^{c}|\tau(\Nfrak\omega^{\alpha}), \boldf_P \rra},
\end{align*}
for
\begin{align*}
g = & \theta(\chi_0)|\left[ \lfrak^{2}  \right]{\N(\mfrak \varpi^{ \beta} 2^{-1})}^{-1/2}A_0^{-1} \mathcal{E}\left(z,\frac{2-k}{2}, (k-1)t,\psi_P \chi_1^{-1}\chi_0^{-1}\right)|\tau(\mfrak\varpi^{2\alpha}). 
\end{align*}
With  calculations analogous to the previous one we obtain 
\begin{align*}
 \Ll^{+}_p = &   
             \frac{\chi_{-1}\chi_0^{-1}(\dfrak\varpi^{\alpha} \mfrak 2^{-1})\chi_0(\dfrak\cfrak\varpi^{\alpha_0}) D_F^{-\frac{3}{2}} 2^{1-k+2d}   {\N(\varpi)}^{\frac{\alpha- \alpha'}{2}m} }{\lambda(T(\varpi^{2\alpha_0 - \alpha'}))  W'(\boldf_P) S(P) \prod_{\pfrak} \frac{\eta\nu(\dfrak_{\pfrak})}{|\eta\nu(\dfrak_{\pfrak})|} \prod_{J}G(\nu\psi')} \times \\
&  \times i^{d-k}G(\chi_0)\psi_P^{-1}(2)\frac{ \Ll(1,\boldf,\chi_0^{-1})}{ {(2\pi)}^{d}  \lla \boldf_P^{\circ}, \boldf_P^{\circ} \rra_{\Nfrak\varpi^{\alpha'}}}.
 \end{align*}

\end{proof}

Before the end of the section, we give a proposition about the behavior of $\Ll_p(Q,P)$ along the element 
\begin{align*}
\Delta(X,Y)=\left(1-\chi_{-1}^{-1}\chi^{-1}{\psi'}^{-2}\psi(c)\N(c)\frac{1+Y}{1+X}\right).
\end{align*}

\begin{prop}\label{reszeta}
Suppose that $\omega\chi^{-1}\psi{\psi'}^{-2}$ is quadratic imaginary, that its conductor is prime to $p$ and that the family associated with $\lambda$ has CM by this character. Then $\Ll_p(Q,P)$ has a pole along $\Delta(X,Y)=0$ of the same order as the $p$-adic zeta function. Otherwise $\Ll_p(Q,P)$ is holomorphic along $\Delta(X,Y)=0$.
\end{prop}
\begin{proof}
We proceed as in \cite[Proposition 5.2]{H6}.  First of all, notice that $\Delta(X,Y)$ has a simple zero along $s=m+1$, $\eps_Q=\eps_P=\eps'_P=\mathbf{1}$ if and only if the $p$-part of $\omega\chi^{-1}\psi {\psi'}^{-2}$ is  trivial. \\ 
The $\xi$-th Fourier coefficient of the Eisenstein series is a multiple of $\int_{\clpinf}{\lla \N_p(z) \rra}^{-1}\textup{d}\zeta_{\eta} $, where $\eta=\omega\chi^{-1}\psi {\psi'}^{-2}\omega_{-\xi}$ and $\omega_{-\xi}$ is the primitive quadratic character associated with $F(\sqrt{-\xi})$. So if $\psi{\psi'}^{-2}\chi^{-1}$ is not imaginary quadratic, then  \begin{align*}
\int_{\clpinf} {\lla \N_p(z) \rra}^{-1}\textup{d}\zeta_{\eta} =0 
\end{align*} 
for all $\xi$. In fact, we can take $s_0\equiv 0 \bmod p^n$, $s_0$ odd, and we have 
then  
\begin{align*}
\int_{\clpinf} {\lla \N_p(z) \rra}^{s_0-1}\textup{d}\zeta_{\eta} = \ast L(1-s_0,\eta). 
\end{align*} 
which is $0$ as $\eta$ is even. We know then that the Eisenstein series we obtain is $0$.\\
If  $\psi{\psi'}^{-2}\chi^{-1}=\omega_{-\xi_0}$, then the value of this integral is a non zero multiple (because the missing Euler factors at primes dividing $\Nfrak$ but not the conductor of $\omega_{-\xi_0}$ do not vanish) of the $p$-adic regulator of the global units of $F$ \cite{Colmez}.\\
Let 
\begin{align*}
 g =  &  (\textup{d}^{[s/2]t}\theta(\chi\omega^{-s}))|\left[\mathfrak{l^2} \right] \times \textup{d}^{-v} \left( {\calE}'(z,\frac{1}{2}; s_0 t,\eta)|[\mfrak^2 \varpi^{2 \beta} 4^{-1}] \right); 
 \end{align*}
we have to show that if $\boldf_P$ has not CM by a $\omega_{-\xi_0}$ then $g$  is killed by the Petersson product with $\boldf_P$.\\
 A necessary condition for $\lambda(\xi,\rif; g, \psi_P\psi_P^{-2})$ to be different from $0$ is $\xi= \xi_2^2 + \xi_1^2 \xi_0$, with $\xi_1$ and $\xi_2$ in $\rif$. Then $\xi$ is a norm from $F(\sqrt{-\xi_0})$ to $F$. Take $\xi$ such that $\xi \rif$ is a prime ideal of $F$ which remains prime in $F(\sqrt{-\xi_0})$ (i.e. it is not a norm) and such that $ \lambda_P(T(\xi\rif))\neq 0$. So $T(\xi\rif)$ acts as a non-zero scalar on $\boldf_P$ and as $0$ on $g$, and $g$ must be orthogonal to $\boldf_P$.
\end{proof}
If $\lambda$ has CM by $\omega_{-\xi_0}$, then Leopoldt's conjecture for $p$ and $F$ is equivalent to the fact that $\Ll_p(Q,P)$ has a simple pole along $\Delta(X,Y)=0$.
%%% pi OK!  2 is OK!!!! D_F is OK!! \N(lfrak) e \N(mfrak) OK!!

%%%%%%%%%%%%%%%%%%%%%%%%%%%%%%%%%%%%%%%%%%%%%%%%%%%%%%%%%%%%%%%%%%%%%%%%%%%%%%%%%%%%%%%%%%%%%%%%
%%%%%%%%%%%%%%%%%%%%%%%%%%%%%%%%%%%%%%%%%%%%%%%%
\section{A formula for the derivative}\label{proofMT}
We can apply now the classic method of Greenberg and Stevens \cite{SSS} to the formula (\ref{fact}) in order to prove Theorem \ref{MainTh} which has been stated in the introduction and which we recall now.
\begin{theo}
Let $p\geq 3$ be a prime such that there is only one prime ideal of $F$ above $p$ and let $\boldf$ be a Hilbert  cuspidal eigenform  of parallel weight $2$, trivial Nebentypus and conductor $\Nfrak p$. Suppose that $\Nfrak$ is squarefree and divisible by all the primes of $F$ above $2$; suppose moreover that $\pi(\boldf)_p$ is the Steinberg representation. Then the formula for the derivative in Conjecture \ref{MainCo} is true (i.e. when $g=1$). \end{theo} 
We fix a Hilbert modular form $\boldf$ of parallel weight $2$ and of Nebentypus $\psi$ trivial at $p$.\\
Let $\boldI^{\mathrm{ord}}$ be the integral closure of an irreducible component of $\boldh^{\mathrm{ord}}$ and let $P_{\boldf}$ be an arithmetic point of $X(\boldI^{\mathrm{ord}})$ corresponding to $\boldf$. Let us denote by $\lambda$ the corresponding structural morphism from $\boldh^{\mathrm{ord}}$ to $\boldI^{\mathrm{ord}}$. Let $L^{\mathrm{ord}}_p(Q,P)$ be the $p$-adic $L$ functions obtained by specializing the $p$-adic $L$-function of Theorem \ref{T1} to $\oo[[\clpinf]] \hat{\otimes} \boldI^{\mathrm{ord}}$. In particular, we have that, up to some non-zero factor,  $L^{\mathrm{ord}}_p(s,P_{\boldf})$ coincides with $L_p(s,\mathrm{Sym}^2(\boldf))$, the $p$-adic $L$ function associated with $\boldf$. \\
We remark that if $\boldf$ is primitive, then the first statement of Theorem \ref{controlThm} tell us that $\boldI$ is \'etale at $P_{\boldf}$ over $\LL$.% In fact, even if the form $\boldf$ is not of exact level $p$ (for example, when $p$ is ramified), there are not other nearly ordinary forms of level $p$ which share with $\boldf$ almost all  the same Hecke eigenvalues.  
This allows  us to define a $\boldI^{\mathrm{ord}}$-algebra structure $\varphi$ (which depends on the point $P_{\boldf}$) on the field of meromorphic functions around $2$ of fixed, positive radius of convergence $\mathrm{Mer}(D_2)$, which is a subfield of $\Qb_p((k-2))$.\\
Namely we consider the continuous group homomorphism $$
\begin{array}{cccc}
 \varphi_0 : & \clpinf & \rightarrow  & \mathrm{Mer}(D(2))^{\times} \\
 & \afrak & \mapsto &  ( k \mapsto \psi(\N(\afrak)){\lla \N(\afrak) \rra}^{k-2}).
\end{array}$$
This morphism extends to a continuous algebra homomorphism 
 $$
\begin{array}{cccc}
 \varphi_0 : & \oo[[\clpinf]] & \rightarrow  & \mathrm{Mer}(D(2)).
\end{array}$$
Let us denote by $ \boldI^{\mathrm{ord}}_{P_{\boldf}}$ the localization-completion of $ \boldI^{\mathrm{ord}}$ at ${P_{\boldf}}$. \\
As $\boldI^{\mathrm{ord}}$ is \'etale around $P_{\boldf}$ over $\oo[[X]]$ and because the ring of meromorphic functions of fixed, positive radius of convergence is henselian \cite[Theorem 45.5]{Nag} we have that it exist a morphism $\varphi$ such that 
$$ 
\xymatrix{ 
 {\boldI^{\mathrm{ord}}_{P_{\boldf}}} \ar[r]^-{\varphi}  & \mathrm{Mer}(D(2)) \\
{\oo[[\clpinf]].} \ar[u] \ar[ur]^-{ \varphi_0  }}
$$

As $\boldI^{\mathrm{ord}}$ is generated by a finite number of elements $\lambda(T(\qfrak))$,  we deduce that there exists a disc of positive radius around $2$ such that $\varphi(i)$ is convergent for all $i$ in $\boldI^{\mathrm{ord}}$.\\ 
For each $k$ in this disc, we can define a new point $P_k$ of $\boldI^{\mathrm{ord}}$;   $P_k(i)=\varphi(i)(k)$. It is obvious that $P_2= P_{\boldf}$. So, it makes sense now to derive elements of $\boldI^{\mathrm{ord}}$ with respect to $k$.\\
For all arithmetic points $P$ of $\boldI^{\mathrm{ord}}$ and $i$ in $\boldI^{\mathrm{ord}}$, we have that $P(i)=\varphi(i)(2)$.\\
We can do the same also for the cyclotomic variable, that is we define 
$$
\begin{array}{cccc}
 \varphi' : & \clpinf & \rightarrow  & \mathrm{Mer}(D(0))^{\times} \\
 & \afrak & \mapsto &  ( s \mapsto {\lla \N(\afrak) \rra}^{s}),
\end{array}$$ 
where $\mathrm{Mer}(D(0))$ denotes the  ring of meromorphic functions of fixed, positive radius of convergence around $0$, and then we extend $\varphi' $ to a continuous algebra homomorphism 
 $$
\begin{array}{cccc}
 \varphi' : & \oo[[\clpinf]] & \rightarrow  & \mathrm{Mer}(D(0)).
\end{array}$$
We define then  
 $$
\begin{array}{cccc}
 \varphi \otimes \varphi': & \oo[[\clpinf]] \hat{\otimes} \boldI^{\mathrm{ord}} & \rightarrow  & \mathrm{Mer}(D(0,2)).
\end{array}$$ where $\mathrm{Mer}(D(0,2))$ denotes the ring of meromorphic functions of fixed, positive radius of convergence around $(0,2)$. We consider $\mathrm{Mer}(D(0,2))$ as a subring of  $\Qb_p((s,k-2))$. Let us denote by $L_p(s,k)$ the image of $L^{\mathrm{ord}}_p(Q,P)$ through $\varphi \otimes \varphi'$ (notice that we have a change of variable $s \mapsto s -1$ from the $p$-adic $L$ function of the introduction).\\
Before proving the theorem, we recall some facts about the arithmetic $\Ll$-invariant of the Galois representation of  $\mathrm{Ind}_F^{\Q}(\mathrm{Ad}(\rho_{\boldf}))$; it has been calculated, under certain hypotheses, in \cite[Theorem 3.73]{HIwa} following the definition given in \cite{TTT}.  \\
Untill further notice, suppose $p$ unramified. The main hypothesis required in the proof of \cite[Theorem 3.73]{HIwa} is the so-called  $R=T$ theorem (\cite[(\mbox{vsl})]{HIwa}). Namely, let   $\boldh^{\mathrm{cycl}}(\Nfrak,\oo)$ be the cyclotomic Hecke algebra (defined in Section \ref{ptheory} of this paper). Let  $$
\begin{array}{cccc}
 Def_{\overline{\rho}_{\boldf}} : & CLN(\oo) & \rightarrow  & Set
\end{array}$$ be the deformation functor from the category of local, profinite $\oo$-algebra with same residue fields as $\oo$ into sets which associate to $A$ in $CNL(\oo)$ the set of $A$-deformations which satisfy certain ramification conditions outside $p$ and  an ordinarity condition at $\pfrak$ dividing $p$.  Let $R^{\mathrm{ver}}$ be the minimal versal hull for $Def_{\overline{\rho}_{\boldf}}$ and let $\mfrak$ be the ideal of  $R^{\mathrm{ver}}$  corresponding to $\rho_{\boldf}$. Let $T_{\boldf}$ be the localization-completion of  $\boldh^{\mathrm{cycl}}(\Nfrak,\oo)$ at ${P_{\boldf}}$. We have then a morphism $u: R^{\mathrm{ver}} \rightarrow T_{\boldf}$. Let $R$ be the  localization-completion of $R^{\mathrm{ver}}$ at $\mfrak$ and let $T$ be the localization-completion of $T_{\boldf}$ at $\mfrak$. Then the above mentioned condition (\mbox{vsl}) states  that $u: R \rightarrow T$
 is an isomorphism. \\
Such a condition holds in many cases thanks to the work of Fujiwara \cite{Fu}. For the Hilbert modular form of Theorem \ref{MainTh}, a proof can be found in \cite[Theorem 3.50]{HIwa}. \\ 
Let  $\boldI^{\mathrm{cycl}}$ be the irreducible component of $\boldh^{\mathrm{cycl}}(\Nfrak,\oo)$ to which $\boldf$ belong and let $\lambda^{\mathrm{cycl}}$ be the corresponding structural morphism; $\boldI^{\mathrm{cycl}}$ is finite flat over $\oo[[x_{\pfrak}]]_{\pfrak |p}$. The $\Ll$-invariant is expressed in term of partial derivative of $\lambda^{\mathrm{cycl}}(T_0(\pfrak))$ with respect to $x_{\pfrak}$. This is not surprising,  as from its own definition, the $\Ll$-invariant is defined in term of the Galois cohomology  of the restriction of the representation to $\mathrm{Gal}(\Qb/ \Q(\mu_{p^{\infty}}))$ and, as we said at the end of Section \ref{ptheory}, the cyclotomic Hecke-algebra takes into account only modular forms whose local Galois representation at $\pfrak$ is of fixed type outside  $\mathrm{Gal}(F_{\pfrak}(\mu_{\p^{\infty}})/F_{\pfrak})$. \\
 In the particular case of an elliptic curve $E$ with multiplicative reduction at all $\pfrak|p$, the calculation of the $\Ll$-invariant does not require to consider  all the  all the cyclotomic deformations, but only the ordinary, as explained below. \\
%We have moreover that this $\Ll$-invariant coincides with the $\Ll$-invariant of $\rho_{\boldf}$; this is not a coincidence, but follows %from an argument of \cite[Section 3]{TTT}. For an ordinary representation $V$ we define $W_p$ as the subspace where the eigenvalues of %$Frob_p \chi^{-i}$ for $\chi$ the cyclotomic character and $i=0, -1$ is exactly $1$ or $p$. In {\it cit. loc.} it is shown that the 
%$\Ll$-invariant depends only on the restriction of $V$ to $\mathrm{Gal}(\overline{\Q}_p/\Q)$ exactly when $W_p=W_p(V)$ does not %contain neither $\Q_p$ nor $\Q_p(1)$. This is exactly the case of $\rho_{\boldf}(=W_p(\rho_{\boldf}))$ and  
%$\mathrm{Ad}(\rho_{\boldf})$, for which $W_p=\rho_{\boldf}$. The same is true also when taking induction. \\
If $E$ has split multiplicative reduction at $\pfrak$, let us denote by $Q_{\pfrak}$ the Tate period of $E$ over $F_{\pfrak}$ and pose $q_{\pfrak}=\N_{F_{\pfrak}/\Q_p}(Q_{\pfrak})$. If the reduction at $\pfrak$ is not split, then ${\lambda^{\mathrm{cycl}}(T(\varpi_{\pfrak}))(2)}=-1$; we consider a quadratic twist $E_d$ such that $E_d$ has split multiplicative reduction and we define  $Q_{\pfrak}$ and $q_{\pfrak}$ as above. We have  then a Galois theoretic formula for the $\Ll$-invariant \cite[Corollary 3.74]{HIwa} 
\begin{align*}
 \Ll(\mathrm{Ind}_F^{\Q}(\mathrm{Ad}(\rho_{\boldf}))) &=  \prod_{\pfrak | p} \frac{\log_p(q_{\pfrak})}{ord_p(q_{\pfrak})}.
\end{align*}
The right hand side as been calculated by  Mok in term of logarithmic derivative of the character 
\begin{align*}
\lambda : \boldh^{\mathrm{ord}} \rightarrow \boldI^{\mathrm{ord}}  .                                                                                           \end{align*}
 We obtain from \cite[Proposition 8.7]{Mok} 
\begin{align*}
\frac{\partial \lambda(T(\varpi_{\pfrak}))}{\partial k} (2){\lambda(T(\varpi_{\pfrak}))(2)}^{-1} =-\frac{1}{2} f_{\pfrak} \frac{\log_p(q_{\pfrak})}{ord_p(q_{\pfrak})}
\end{align*}
where $f_{\pfrak}=[\F_{\pfrak}:\F_p]$. \\
We remark that the proof of \cite[Proposition 8.7]{Mok} works without any change when $p$ is $2$, $3$ or ramified in $F$.\\% if one keeps Mok's convention for the local Artin map (uniformizer to geometric Frobenius) and normalise the valuation $ord_{F_{\pfrak}}$ of $F_{\pfrak}$ such that $ord_{F_{\pfrak}}(p)=e_{\pfrak}$. \\
The presence of $f_{\pfrak}$ has been explained in the introduction of \cite{HTate1} and is related to the Euler factors which we removed at $p$ in the formula presented in Conjecture \ref{MainCo}. More precisely, in this formula appears the product $E^*(m+1,\boldf)$ of all the non-zero Euler factors which have to be removed to obtain $p$-adic interpolation.  Among the factor which we removed, there is $(1-p^{(1-s) f_{\pfrak}})$ for the primes $\pfrak$ such that ${\lambda(T(\varpi_{\pfrak}))}^{-2}=1$. We can rewrite this as $ \prod_{\zeta \in \mu_{f_{\pfrak}}}  (1-\zeta p^{1-s})$.  In the case of the elliptic curve $E$, we have the trivial zeros at $s=1$; hence the factor $E^*(1,\boldf)$ should then contain $ \prod_{1 \neq \zeta \in \mu_{f_{\pfrak}}}  (1-\zeta )=f_{\pfrak}$. \\
We remark that the fact that the $\Ll$-invariants of $\rho_{\boldf}$ and $\mathrm{Ad}(\rho_{\boldf})$ are the same is not a coincidence. Let $W$ be the part of Hodge-Tate weights $1$ and $0$ of this two representations restricted to $\mathrm{Gal}(\Qb_p/\Q_p)$ (which are isomorphic); we have that $W$ is a non-split extension of $\Q_p$ by $\Q_p(1)$  and this ensures that the $\Ll$-invariant is determined only by the restriction to $W$ (as said at the end of page 169 of \cite{TTT}).

\begin{proof}[Proof of Theorem \ref{MainTh}]
 Using the interpolation formula of \ref{T1} we see that $L_p(k-2,k)=0$. In particular, we apply the operator $\frac{\textup{d}}{\textup{d} k}$ and we get
\begin{align*}
{\frac{\textup{d}L_p(s,k)}{\textup{d}s}}|_{s=k-2} & = - {\frac{\textup{d}L_p(s,k)}{\textup{d}k}}|_{s=k-2}.
\end{align*}
If we suppose that $2$ divides the level of $\boldf$, than we have that $E_2(0,k)$, defined in the proof of Theorem \ref{T1}{\it i)}, is identically $1$. Let $\Sigma_p$ be set of $\pfrak \mid p$ such that $\boldf$ is a Steinberg representation at $\pfrak$, and $\Pi_p$ the set  $\pfrak \mid p$ such that $\boldf$ is a principal series  representation at $\pfrak$. Let us denote by $g_1$ the cardinality of $\Sigma_p$, from the factorization in \ref{fact}, we also have
\begin{align*}
L_p(0,k) = & {2}^{g_1}\prod_{\pfrak \in \Sigma_p}\frac{\lambda'(T(\varpi_{\pfrak}))(2)}{{\lambda(T(\varpi_{\pfrak}))(2)}^3}\prod_{\pfrak \in \Pi_p}(1 -{\lambda(T(\varpi_{\pfrak}))(2)}^{-2} ) E_{\pfrak}(Q_0,P_{\boldf}) \frac{\Ll(1,\boldf)}{\Omega(\boldf)} (k-2)^{g_1} + \\
&  + ({(k-2)}^{g_1+1}).
\end{align*}
If we suppose that the Nebentypus of $\boldf$ is trivial and that the conductor of $\boldf$ is squarefree, then we have that $\boldf$ is Steinberg at all primes dividing its conductor. Using the explicit description of the Euler factors at the ramification primes for $\boldf$ given in Section \ref{Extrafactors}, case {\it (iii)}, we see that there aren't any missing  Euler factors and consequently  the primitive $L$-function and the imprimitive one coincide.\\
Using the explicit expression of the $\Ll$-invariant in term derivative of Hecke eigenvalue given above, we obtain for $\boldf$ as in the theorem
\begin{align*}
{\frac{\textup{d}L_p(s,2)}{\textup{d}s}}|_{s=0} = \Ll(\mathrm{Ind}_F^{\Q}(\mathrm{Ad}(\rho_{\boldf}))) f_p\frac{L(1,\mathrm{Sym}^2(\boldf))}{\Omega(\boldf)}.
\end{align*}
\end{proof}
We remark that we have something more; if there is at least a prime $\pfrak$ at which $\boldf$ is Steinberg and  $g\geq 2$, then the order at $s=0$ of $L_p(s,\mathrm{Sym}^{2}(\boldf))$ is at least $2$ (keeping the assumptions on the conductor outside $p$).\\
Let us point out that the hypothesis that $2$ divides the conductor of $\boldf$ is necessary, as the missing Euler factor at $2$ which we remove when interpolating the imprimitive $L$-function is zero. But a more carefull study of the missing Euler factor at $\Nfrak$ could allow us to weaken the hypothesis on the conductor of Theorem \ref{MainTh}, as not all the missing Euler factors have to vanish.\\ 
For example, we could prove Conjecture \ref{MainCo} for $L(s,\mathrm{Sym}^2(\boldf),\chi)$ when the character $\chi$ is very ramified modulo $2\Nfrak$.\\
Let us give now the example of an elliptic curve $E$ over $\Q$. Let $q$ be a prime of potentially good reduction of $E$, we have to exclude the cases when one of the missing Euler factor at $q$ is $(1-q^{1-s})$. We have that in this case the missing Euler factors have been explicitly calculated in \cite{CS}; let $l$ be an auxiliary prime different from $2$ and $q$, $\Q_q(E[l])$ the extension of $\Q_q$ generated by the coordinates of the torsion points of order $l$ of $E$ and $I_q$ the inertia of $\mathrm{Gal}(\Q_q(E[l])/\Q_q)$. By \cite[Lemma 1.3, 1.4]{CS} we have that there is no Euler factor of type $(1-q^{1-s})$ at $q$ if and only if one of the following is satisfied:
\begin{itemize}
\item[a)] $I_q$ is not cyclic, 
\item[b)] $I_q$ is cyclic and  $\Q_q(E[l])/\Q_q$ is not abelian. 
\end{itemize}
We have the following corollary:
\begin{coro}
Let $E$ be an elliptic curve over $\Q$ and let $\boldf$ be the associated modular form. Suppose that $E$ has multiplicative reduction at $p$, even conductor and that for all prime of additive reduction one between {\it a)} or {\it b)} is satisfied, then \ref{MainCo} for $\boldf$ is true.
\end{coro}

%%%%%%%%%%%%%%%%%%%%%%%%%%%%%%%%%%%%%%%%%%%%%%%%%%%%%%%%%%%%%%%%%%%%%%%%%%%%%%%%%%%%%%%%%%%%%%%%%%%%%%%%%%%%%%%%%%%%%%%%%%%%%%%%%%%%%%%%%%%%%%%%%%%%%%%%%
\section{A formula for higher order derivative}\label{proofBC}
In this section we want to provide an example of \ref{MainCo} in which $g \geq 2$. We begin considering a form $\boldf$ has in \ref{MainTh} and we show that \ref{MainCo} is true (under mild hypothesis) for $\mathrm{Sym}^2(\boldg)$, where $\boldg$ denotes an abelian base change (is a sense we specify below) of $\boldf$. We follow closely the strategy of \cite[\S 9]{Mok}.
Let $F$ be a totally real number field as before, and $E$ a totally real abelian extension of $F$ of degree  $n$. For this whole section, we will suppose for simplicity that $p$ is unramified in $E$, even though we belive this is not necessary, to simplify the calculations.
We fix, for this section, $H=\mathrm{Gal(E/F)}$. Let $\boldf$ be a Hilbert modular form for $F$, following  \cite[Chapter 3, \S 6]{AC} we can define its base-change $\boldg$ to $E$. Let $\qfrak'$ be a prime of $E$ above a prime $\qfrak$ of $F$, suppose that $\boldf$ at $\qfrak$ is minimal but not supercuspidal, then $\bolda(\qfrak',\boldg)=\bolda(\N_{E/F}(\qfrak'),\boldf)$. Let us denote by $\hat{H}$ the group of characters of $H$; via class field field theory, we can identify each element of $\hat{H}$ with a Hecke character of $F$ factoring through $F^{\times}\N_{E/F}(\mathbb{A}^{\times}_E)$.  We have then
\begin{align*}
L(s,\mathrm{Sym}^2 (\boldg))=\prod_{\phi \in \hat{H}}L(s,\mathrm{Sym}^2 (\boldf),\phi).
\end{align*}
Note that if $\boldf$ is nearly-ordinary (resp. ordinary), then $\boldg$ is nearly-ordinary (resp. ordinary) too.\\
Assuming that Conjecture \ref{MainCo} holds for $\boldf$ we shall show that the same is true for $\boldg$ by factoring the $p$-adic $L$-function $L_p(s,\mathrm{Sym}^2(\boldg))$  in terms of $L_p(s,\mathrm{Sym}^2(\boldf),\phi)$.
\begin{lemma}\label{Gausssum}
Let $\chi$ a gr\"ossencharacter of $F$ of conductor $\cfrak$ and $\chi_E=\chi \circ \N_{E/F}$. Assuming $D_{E/F}$ and $\cfrak$ coprime, for the Gau\ss{} sum defined before Proposition \ref{Theta-Atkin}, we have
$$G(\chi_E) = \chi(D_{E/F}) {G(\chi)}^n.$$
\end{lemma}
\begin{proof} % No pub
As in \cite[Lemma 9.4]{MokT}, we use the functional equation for an Hecke character which we recall in Appendix \ref{App} and the factorization $L(\chi_E,s)$ in term of $L(\chi\phi,s)$, for $\phi$ in $\hat{H}$. \\
We let $\cfrak_{\phi}$ be the conductor of $\chi\phi$ and $\ffrak_{\phi}$ the conductor of $\phi$. Using the formula for the discriminant in tower of extensions 
\begin{align*}
 D_{E} = \N_{F/\Q}(D_{E/F}) D_{F}^n,
\end{align*}
 we obtain the equality
\begin{align*}
\prod_{\phi \in \hat{H}} G(\chi\phi) {\N_{F/\Q}(\cfrak_{\phi})}^{-s} & = G(\chi_E){\N_{E/\Q}(\cfrak)}^{-s}\N_{F/\Q}(\dfrak_{E/F})^{-s +1/2}.
\end{align*}
As $\cfrak$ in unramified in $E$, we can split the Gau\ss{} sums and the conductor of $\chi\phi$ for all $\phi$ into a $\cfrak$-part and a prime-to-$\cfrak$-part. We conclude with the following formulae from class field theory \cite[\S7 (11.9),(6.4)]{Neu}
\begin{align*}
\prod_{\phi \in \hat{H}} \ffrak_{\phi} =   D_{E/F} , \: \: \: \:  \prod_{\phi \in \hat{H}}G(\phi) =   \N_{F/\Q}(D_{E/F}^{1/2}).
\end{align*}
\end{proof}

%Let $E$, $F$, $\boldg$ and $\boldf$ as at the beginning of the section and $f$ the residual degree of any $\qfrak|\pfrak|p$, we have
%\begin{align*}
% \prod_{\qfrak|\pfrak} (1 - \eps(\pfrak)^{f_{\qfrak}} \N(\pfrak)^{-f(1+s+m)} ) = & \prod_{\phi \in \hat{H}} (1 - \eps\phi(\pfrak)\N(\pfrak)^{-(1+s+m)}  ) \\
%\prod_{\qfrak|\pfrak} (1 - \eps(\pfrak)^{f_{\qfrak}} \N(\pfrak)^{-fs}a(\qfrak,\boldf)^{-2} ) =  & \prod_{\phi \in \hat{H}} (1 - \eps\phi(\pfrak) %\N(\pfrak)^{-s}a(\pfrak,\boldf)^{-2}),
%\end{align*}
%which is proven noticing that $H/\lgr \phi \in \hat{H} \mbox{ s.t. } \phi(\pfrak)=1 \rgr  \cong \mu_f$, the groups of $f$-roots of unity, via $\phi \mapsto \phi(\pfrak)$.\\
Let us denote by $S(\boldf)$ (resp. $S(\boldg)$) the factor $S(P) $ defined after Theorem \ref{1var} for $P$ the point corresponding to $\boldf$ (resp. $\boldg$). 
We show now the following
\begin{prop}\label{Factpfun}
 For $E$, $F$, $\boldg$,  $\boldf$ as above. Let $\chi$ be a finite order Hecke character of $F$, and suppose that the conductor of $\boldf$ and $\chi$ are both unramified in the extension $E/ F$.    We have then, for all point  $u^s\eps(u)-1$ in the interpolation range of Theorem \ref{1var},
\begin{align*}
 G(u^s\eps(u)-1,\boldg,\chi_{E}) = & C \prod_{\phi \in \hat{H}} G(u^s\eps(u) -1 ,\boldf,\chi\phi),
\end{align*}
where $G(X,\boldf,\chi\phi)$ (resp. $G(X,\boldg,\chi_{E})$) is the Iwasawa function of Theorem \ref{1var} giving the $p$-adic $L$-function $L_p(s,\mathrm{Sym}^2(\boldf),\chi\phi)$ (resp. $L_p(s,\mathrm{Sym}^2(\boldg),\chi_{E})$) and 
\begin{align*}
 C={\N_{F/\Q}(D_{E/F})}^{-2}\frac{{(S(\boldf)W(\boldf)'\Omega(\boldf))}^{n}}{S(\boldg)W(\boldg)'\Omega(\boldg)}.
\end{align*}
\end{prop}

\begin{proof}
 We use Theorem \ref{1var} to evaluate both sides and show that they coincide.\\ 
We note that we are implicitly identifying the $\Z_p$-cyclotomic extension of $E$ with the one of $F$ via $\N_{E/F}$; in particular if $p^{\alpha_0}$ is the conductor of the character $\eps$, seen as a character of $\clpinf$ of $F$, then the conductor of $\eps_{E}$ is $p^{n \alpha_0}$. This is important to write correctly the evaluation formula of \ref{1var}.\\ 
 We see that the contributions of the factors with $2$ and $i$ match as $d_E= n d_F$. The complex formula for the $L$-function of $\mathrm{Sym}^2(\boldg)$ gives the equality of the $L$-values.  We use Lemma \ref{Gausssum} and the relation between $D_F$ and $D_E$ to control the contribution of the Gau\ss{} sums and discriminants.
\end{proof}

In particular this tells us that ${\Omega(\boldf)}^{n}$ and ${\Omega(\boldg)}$ differ by a non-zero algebraic number. We suppose that the conductor of $\boldf$ and $\chi$ are both unramified in the extension $E/ F$ to simplify the calculation, although it may not be necessary.\\
It is well known that for  $E/F$ an abelian extension of Galois group $H$ and for $\pfrak$ a prime ideal of $F$ which is  unramified in $E$, we have  
\begin{align*}
 \left|\lgr \phi \in \hat{H} \mbox{ s.t. } \phi(\pfrak)=1 \rgr \right| & = g,
\end{align*}
where  $g$ is the number of primes of $E$ above $\pfrak$.\\
From now on, suppose that $p$ is inert in $F$. Let $\pfrak$ be a prime ideal of $E$ above $p$ and let us denote by $f_{\pfrak}$ the residual degree of $E_{\pfrak}$, we have  the following 
\begin{theo}\label{MainThBC}
 Let $E$, $F$, $\boldg$,  $\boldf$ be as in Proposition \ref{Factpfun}. Suppose that $\boldf$ satisfies the hypothesis of Theorem \ref{MainTh}. Then we have 
\begin{align*}
 {\frac{\textup{d}^g L_p(s,\mathrm{Sym}^2(\boldg))}{\textup{d}s^g}}|_{s=0} = & C \Ll(\mathrm{Ind}_F^{\Q}(\mathrm{Ad}(\rho_{\boldg }))) {f_\pfrak^g} \frac{L_p(1,\mathrm{Sym}^2(\boldg))}{\Omega(\boldg)},
\end{align*}
where $g$ denotes the number of primes above $p$ in $E$ and $C$ is a non-zero algebraic number. 
\end{theo}
 \begin{proof}
  First of all from \cite[Proposition 8.7]{Mok} we obtain
\begin{align*}
 \Ll(\mathrm{Ind}_F^{\Q}(\mathrm{Ad}(\rho_{\boldg }))) = {\Ll(\mathrm{Ind}_F^{\Q}(\mathrm{Ad}(\rho_{\boldf })))}^g.
\end{align*}
Let us denote by $f_p$ the residual degree of $F_p$; whenever $\phi(p)=1$, we can show, in the same way as we proved Theorem \ref{MainTh}, that 
\begin{align*}
 {\frac{\textup{d} L_p(s,\mathrm{Sym}^2(\boldf),\phi)}{\textup{d}s}}|_{s=0} = & C_{\phi}  \Ll(\mathrm{Ind}_F^{\Q}(\mathrm{Ad}(\rho_{\boldf }))) f_p\frac{L(1,\mathrm{Sym}^2(\boldf),\phi)}{\Omega(\boldf)}.
\end{align*}
If $\phi(p)\neq 1$, we have instead 
\begin{align*}
 L_p(0,\mathrm{Sym}^2(\boldf),\phi)=C_{\phi} (1 - \phi(p))  \frac{L(1,\mathrm{Sym}^2(\boldf),\phi)}{\Omega(\boldf)}
\end{align*}
 If we write $f=n/g$, we have that 
\begin{align*}
\prod_{\phi \in \hat{H}, \phi(p)\neq 1} (1 - \phi(p)) = f^g 
 \end{align*}
then we use Proposition \ref{Factpfun} and the fact that $f_{\pfrak}=f_p f$  to conclude.
 \end{proof}
We remark that this method does not work for an abelian base change $\boldh$ of $\boldg$. In fact, if we want to apply the same strategy as in the proof of Theorem \ref{MainThBC}, we would need a formula for the derivative of  $L_p(s,\mathrm{Sym}^2(\boldg),\psi)$. But for the complex $L$-function $L(s,\mathrm{Sym}^2(\boldg),\psi)$ there is no factorization  in term of $L$-functions of $\mathrm{Sym}^2(\boldf)$ (unless $\psi$ is a base change from $F$ too).

%%%%%%%%%%%%%%%%%%%%%%%%%%%%%%%%%%%%%%%%%%%%%%%%%%%%%%%%%%%%%%%%%%%%%%%%%%%%%%%%%%%%%%%%%%%%%%%%%%%%%%%%%%%%%%%%%%%%%%%%%%%%%%%%%5
\section{Application to the Main Conjecture}\label{AppMC}
In the introduction we said that Theorem \ref{MainTh} has application to the main conjecture for the symmetric square and the aim of this section is to explain how.\\ 
We suppose now $F=\Q$; let $G(T,\mathrm{Sym}^2(\boldf))$ be the Iwasawa function for $\boldf$ of Theorem \ref{1var}. We shall write 
\begin{align*}
L_p^{\mathrm{an}}(T,\mathrm{Sym}^2(\boldf))=\frac{\lla \boldf, \boldf \rra \pi^2}{i \Omega^+ \Omega^-}G(T,\mathrm{Sym}^2(\boldf)),
\end{align*}
for $\Omega^{\pm}$ the $\pm$ period associated with $\boldf$ via the Eichler-Shimura isomorphism. This change of periods is necessary for what follows and is due to the fact that our $p$-adic $L$-function has a denominator $H_{\boldf}$. Moreover, the evaluations of $L_p^{\mathrm{an}}(T,\mathrm{Sym}^2(\boldf))$ are compatible with Deligne's conjecture on special values and periods.\\
Let $L_p^{\mathrm{al}}(T,\mathrm{Sym}^2(\boldf))$ be the algebraic $p$-adic $L$-function defined as $\Psi(s)$ in \cite[Theorem 6.3]{ASG}; it is the characteristic ideal of a certain Selmer group. The Greenberg-Iwasawa main conjecture for $\mathrm{Ind}_F^{\Q}(\mathrm{Sym}^2(\boldf))$ says that the ideal generated by $L_p^{\mathrm{al}}(T,\mathrm{Sym}^2(\boldf))$ in $\oo[[T]]$ is the same as the one generated by $L_p^{\mathrm{an}}(T,\mathrm{Sym}^2(\boldf))$.\\
In what follows we will denote by $L_p^{\mathrm{al}}(0,\mathrm{Sym}^2(\boldf))^*$ resp. $L_p^{\mathrm{an}}(T,\mathrm{Sym}^2(\boldf))^*$ the first non-zero coefficient of $L_p^{\mathrm{al}}(T,\mathrm{Sym}^2(\boldf))$ resp. $L_p^{\mathrm{an}}(T,\mathrm{Sym}^2(\boldf))$.\\
Let $\eta$ be the characteristic ideal of the Pontryagin dual of $\mathrm{Sel}(\mathrm{Ad}(\rho_{\boldf}))$ (defined as in \cite{ASG}); in \cite[Theorem 6.3 (3)]{ASG} it is shown that $\mathrm{ord}_{T=0} L_p^{\mathrm{al}}(T,\mathrm{Sym}^2(\boldf))=1$ and
\begin{align*}
 L_p^{\mathrm{al}}(T,\mathrm{Sym}^2(\boldf))^* \mid \Ll(\mathrm{Ad}(\rho_{\boldf})) \eta.
\end{align*}
From the work of Urban \cite{Urb} we know that, in most of the cases for $F=\Q$ (for example when $\boldf$ is as in Theorem \ref{MainTh}), $L_p^{\mathrm{an}}(T,\mathrm{Sym}^2(\boldf))$ divides $L_p^{\mathrm{al}}(T,\mathrm{Sym}^2(\boldf))$ in $\oo[[T]][T^{-1}]$.\\
We have, from works of Hida and Wiles, a formula which equals the $p$-adic valuation of the special value $L_p^{\mathrm{an}}(0,\mathrm{Sym}^2(\boldf))$ and of $\eta$ \cite[1.2.3]{Urb}.  For totally real fields, we mention that in \cite{DimIha} the special value is related to the cardinality of the Fitting ideal. \\
If in addition we also know Conjecture \ref{MainCo}, then we deduce that the two series have the same order at $T$ and
\begin{align*}
L_p^{\mathrm{al}}(0,\mathrm{Sym}^2(\boldf))^* | L_p^{\mathrm{an}}(0,\mathrm{Sym}^2(\boldf))^* \;\;\:\: \mathrm{in} \;\; \oo.
\end{align*}
We can write $L_p^{\mathrm{al}}(T,\mathrm{Sym}^2(\boldf)) = A(T) L_p^{\mathrm{an}}(0,\mathrm{Sym}^2(\boldf))$, with $A(T)$ in $\oo[[T]]$; the above divisibility tells us that ${A(0)}^{-1}$ belongs to $\oo$, i.e. $A(T)$ is a unit (because its constant term is a unit) and the main conjecture is proven.
%%%%%%%%%%%%%%%%%%%%%%%%%%%%%%%%%%%%%%%%%%%%%%%%%%%%%%%%%%%%%%%%%%%%%%%%%%%%%%%%%%%%%%%%%%%%%%%%%%%%%%%%%%%%%%%%%%%%%%%%%

%%%%%%%%%%%%%%%%%%%%%%%%%%%%%%%%%%%%%%%%%%%%%%%%%%%%%%%%%%%%%%%%%%%%%%%%%%%%%%%%%%%%%%%%%%%%%%%%%%%%%%%%%%%%%%%%%%%%%%%%%%%%%%%%%%
\appendix
\section{Holomorphicity in one variable}\label{App}
In this appendix, we want to show that the $p$-adic $L$-function of Theorem \ref{T1} {\it  i)} can be modified (dividing by suitable Iwasawa functions interpolating the missing Euler factors) into an Iwasawa function interpolating special values of the primitive complex $L$-function.\\ 
To do this, we follow an idea of Schmidt \cite{Sc} and Dabrowski and Delbourgo \cite{DD}; first we construct another $p$-adic $L$-function which interpolates the other half of critical values and then we relate the two $p$-adic $L$-functions via the functional equation of Section \ref{Extrafactors}. We shal use this results in the second part of the appendix to show that the many variable $p$-adic $L$-function is holomorphic too, following \cite{H6}.\\
Before doing this, we recall the functional equation satisfied by $L(s,\chi)$, for a primitive unitary gr\"ossencaracter $\chi$.  Let $\ffrak$ be its conductor, we define the complete $L$-function 
\begin{align*}\LL(s,\chi) = & \prod_{v|\infty}{\left(\pi^{-(s+p_v)/2}\G\left(\frac{s+p_v}{2}\right)\right)} L(s,\chi), \end{align*}
for $p_v=0$ or $1$ such that $\chi_v(-1)={(-1)}^{p_v}$.
It is well known since Hecke (see also Tate's thesis \cite{TateCF}) the existence of a functional equation 
\begin{align*}
\LL(s,\chi) = & \beps(s,\chi)\LL(1-s,\chi^{-1})
\end{align*}
for $\beps(s,\chi)=i^{\sum_v p_v} \chi_{\infty}(-1) G(\chi)\N(\ffrak)^{-s}{|\N(\dfrak)|}^{-s +1/2}$ 
where as before
\begin{align*} 
G(\chi)= & \sum_{x \in \ffrak^{-1}\dfrak^{-1}/\dfrak^{-1}} \chi_{\infty}(x)\chi_{f}(x\dfrak\ffrak) \bexp_F(x)\\
 =& \prod G(\chi_{\qfrak}),\\
 G(\chi_{\qfrak}) = &  \chi^{-1}_{\qfrak}(\varpi_{\qfrak}^{\alpha_{\qfrak}+e_{\qfrak}})\sum_{x \bmod \varpi_{\qfrak}^{\alpha_{\qfrak}}} \chi_{\qfrak}(x)  \bexp_{F_{\qfrak}}\left(\frac{x}{\varpi^{e_{\qfrak}+\alpha_{\qfrak}}}\right),
\end{align*}
%coincide con Mok poiché \chi^* = \chi_f poiche all'infinito c'e giusto il segno
for $e_{\qfrak}$ the ramification index of $F_{\qfrak}$ over $\Q_{q}$ and $\alpha_{\qfrak}$ such that $\qfrak^{\alpha_{\qfrak}}$ divides $\ffrak$ exactly. 
%see \cite[(4.8)]{ShH2} or \cite[Definition 3.5]{Mok}.\\
We can now modify Theorem \ref{Ribet} to see
\begin{prop}\label{RibetFE}
Let $\chi$ be a character of finite order of conductor $\ffrak$. Suppose $\chi_{\sigma}(-1)=1$ for all $\sigma \in I$. For all $c \in \clpinf$, we have a measure $\zeta'_{\chi,c}$ on $\clpinf$ such that, for all finite order characters $\eps$ and for all $s \geq 0$, we have 
\begin{align*}
 \int_{\clpinf} \eps(z){ \N_p(z) }^s \textup{d}\zeta'_{\chi,c}(z) = & (1- {(\chi\eps)}_0(c){\N_p(c)}^{s+1}) \prod_{\pfrak|p}(1-{(\chi\eps)}_0(\pfrak){\N(\pfrak)}^{-s}) \times \\
 & \times \Omega(\eps \chi,s) \frac{\prod_{\pfrak|p}G({(\chi\eps_0)}_{\pfrak})}{{\N(\varpi^{\alpha_0})}^{-s}} L(1+s,(\chi^{-1}\eps^{-1})_0).
 \end{align*}
where ${(\chi\eps)}_0$ is the primitive character associated with the character  ${\chi\eps}$ of $\rif/\cfrak p^{\alpha}$, $p^{\alpha_0}$ is the $p$-part of its conductor, and 
\begin{align*}
\Omega(\eps\chi,s)=& i^{\kappa d}  {D_{F,p}}^{s+\frac{1}{2}} \eps\chi_{\infty}(-1)\frac{{(\pi^{-(s+1+\kappa)/2}\G((s+1+\kappa)/2))}^d}{{(\pi^{(s-\kappa)/2}\G(-(s-\kappa)/2))}^d}.
\end{align*}
for  $\kappa=0,1$, congruent to $s$ modulo $2$ 
\end{prop}
% It coincides with DD as $G(\chi)G(\overline{\chi})=p^{\alpha_0}$ and the exponential of c^{-s} can be exchanged 
\begin{proof}
Let $\ffrak^{(p)}$ be the conductor of $\chi$ outside $p$ and let $\dfrak^{(p)}$ (resp. $\dfrak_p$) be the prime-to-$p$ part (the $p$-part) of $\dfrak$. Let us denote $\N(\dfrak_{p})$ by $D_{F,p}$ and $\N(\dfrak^{(p)})$ by $D_F'$. From Theorem \ref{Ribet} and the functional equation above, we simply divide by the functions $A_{\N(\dfrak^{(p)}\ffrak^{(p)})}(X)$,  $N(\dfrak)^{-1/2}$ and the Gau\ss{} sums for $\qfrak|\ffrak'$.
\end{proof}

We recall that for an element $\cfrak$ of $\clpinf$, we have defined $\lla \N_p(\cfrak) \rra = \omega^{-1}(\N_p(\cfrak))\N_p(\cfrak)$ and that $\lla  \N_p(\cfrak)  \rra $ is an element of $1+p\Z_p$, so it makes sense to define 
\begin{align*}
A_{ \N_p(\cfrak) }(X)= & {(1+X)}^{\log_p(\lla  \N_p(\cfrak)  \rra)/\log_p(u)},
\end{align*}
where $u$ is a generator of $1+p\Z_p$. In case we have to consider the $L$-values of imprimitive characters, we can obtain the $p$-adic interpolation of the values $L_{\Nfrak}(1+n,(\psi\eps)^{-1})$ simply multiplying by $(1-(\psi\eps)_0(\qfrak)^{-1}{\lla  \N_p(\qfrak)  \rra}^{-1} {A_{ \N_p(\qfrak) }(X)}^{-1})$.\\

Fix now $\boldf$, a nearly ordinary form of Nebentypus $(\psi,\psi')$ and weight $k\geq 2t$ which we decompose as $k = (m +2)t -2v$ as in Theorem \ref{1var}. Suppose that $\psi'$ comes from a character of $\clpinf$. We give now the analogue of Theorem \ref{1var} for the other critical values, namely the one in the strip $\left[ m+2 m+k_0 \right] $.

\begin{theo}\label{1var2}
Fix a primitive adelic character $\chi$ of level $\mathfrak{c}$, such that $\chi_{\sigma}(-1)=1$ for all $\sigma|\infty$. We have a $p$-adic $L$-functions $\Ll^-_p(Q)$ in the fraction field of  $\oo[[\clpinf ]]$ such that the following interpolation property hold; for all arithmetic points $Q$ of type $(s,\eps_Q)$, with $ m - k_0 + 2  \leq s \leq m $   the following interpolation formula holds 
\begin{align*} 
\Ll_p^{-}(Q) = & C_2 E^{-}_{1}(Q)E^{-}_{2}(Q)E_3(Q)\frac{\Ll(2m +2 -s, \boldf,\eps_Q \omega^{-s}\psi^{-2} \chi)}{(2\pi)^{d s}\Omega(\boldf)}.
\end{align*}
\end{theo}

The Euler factors $E^{-}_{1}(Q)$ and $E^{-}_{2}(Q)$ at $p$  could be thought as the factors $E_1(\tilde{Q})$ and $E_2(\tilde{Q})$, for $\tilde{Q}= \eps_Q^{-1}(u)u^{2m-2-s} -1$, the {\it symmetric } of $Q$; more precisely, they are
\begin{align*}
E_1(Q) =& \lambda(T(\varpi^{2\beta -2\alpha_0})){\N(\varpi^{\beta -\alpha_0})}^{s-2m-2}\times \\
 & \times \prod_{\pfrak | p}( 1 - (\chi^{-1}\psi^{2}\eps_Q^{-1}\omega^{s})_0(\pfrak) \N(\pfrak)^{2m+1-s}{\lambda(T(\pfrak))}^{-2}),\\
E_2(Q) =& \prod_{\pfrak \mid p} (1 - (\chi\eps_Q\omega^{-s}\psi^{-1})_0(\pfrak){\N(\pfrak)}^{s-m-1})\times \\
& \times (1 - (\chi\eps_Q\omega^{-s})_0(\pfrak) {\lambda(T(\varpi_{\pfrak}))}^{-2}{\N(\varpi)}^{s}).
\end{align*}
 The Euler factor $E_3(Q)$ is defined as
\begin{align*}
\prod_{\pfrak|p}\frac{G({\eta^{-1}}_{\pfrak})(1-{(\eta^{-1})}_0(\pfrak){\N(\pfrak)}^{m-s})} {(1-{(\eta)}_0(\pfrak){\N(\pfrak)}^{s-m-1})},
\end{align*}
for $\eta=\eps_Q\omega^{-s} \chi_{-1}\chi\psi^{-1}{\psi'}^{-2}$. The constant $C_2$ can be found the following proof of  Theorem \ref{1var2}.\\
Let us point out that this symmetric square $p$-adic $L$-function  interpolates the imprimitive special values,  but that none of the Euler factors which we have removed outside $p$ is zero. This means that if we could find a formula for the first derivative of it as we did in Theorem \ref{MainTh}, we could then prove Conjecture \ref{MainCo} in the case in which $p$ is inert {\it without any assumption on the conductor of $\boldf$}.

 \begin{proof}
We use the same notation of Sections \ref{operators} and \ref{padicL}. In particular we have $\Mfrak= 4 \cfrak^2 \Lfrak^2$.\\  We let $\alpha$ such that $\boldf$ is of level $p^{\alpha}$.
 Let $Q$ be a point  of type $(s,\eps_Q)$, with $ m - k_0 + 2  \leq s \leq m$ and $\eps_Q$ a finite order character as in the statement of the theorem; in what follows, we shall write $\eps$ for $\eps_Q$. We can define an Eisenstein measure ${E}_c^{\chi,-}$ on $\clpinf$ similarly to Section \ref{p-measure}; let us define its value on $Q$ as 
\begin{align*}
\int_{\clpinf} \eps(z){\lla \N_p(z) \rra}^{s} \textup{d}{E}_c^{\chi,-}(z) = & (1- \eta^{-1}(c){\N_p(c)}^{s_0-m})\Omega(\eta^{-1}, s_0-m-1) \times \\ 
 & \prod_{\pfrak \mid p} G(\eta_{\pfrak}^{-1})\times {\N(\varpi)}^{(s_0-m-1)\beta_0}  e'_0(k,s_0,\eta)|\nu_Q,
\end{align*}
where $\eta=\eps\omega^{-s} \chi_{-1}\chi\psi^{-1}{\psi'}^{-2}$, $\beta$ is such that $\eta$ is a character modulo $\rif/\Mfrak p^{\beta}$ (we do not assume $\eta$ primitive), $\beta_0$ is the exact $p$-power of the conductor of $\eta_0$ and $s_0= 2m+1-s$ and $\nu_Q$ is the continuous function on $\rif_p^{\times}$ define on $\xi \in F$, $\xi \gg 0$ by
\begin{align*}
\nu_Q( \xi) = \frac{ \prod_{\pfrak|p}(1-{(\eta^{-1}\omega_{\xi})}_0(\pfrak){\N(\pfrak)}^{s_0 -m-1})}{\prod_{\pfrak|p}(1-{(\omega_{\xi}\eta)}_0(\pfrak){\N(\pfrak)}^{m-s_0})}.
\end{align*}
 
We recall from Proposition \ref{Eisen_other2} 
\begin{align*}
{e}'_0(k,s_0,\eta) = & \sum_{\xi \in 2^{-1}, \xi \gg 0, (p,\xi)=\rif} \xi^{ -2v } g_{f}\left(\xi,s_0-m,\eta\right)L_{p\Mfrak}(s_0-m,\eta\omega_{2\xi})q^{\xi}.
\end{align*} Let us point out that it is a $p$-adic measure as its Fourier coefficients are the $p$-adic $L$-functions of Proposition \ref{RibetFE} (twisted by $\eps(z) \mapsto \eps^{-1}(z) \lla \N_p(z)^{m+2} \rra$) and some exponential functions as in Section \ref{p-measure}. The presence of the twist by the function $\nu_Q$ is due to the fact that the factor appearing in the interpolation formula of Proposition \ref{RibetFE} for $\chi=\eta\omega_{2\xi}$ is not independent of $\xi$. \\
The value of this measure corresponds, up to some normalization factors, with the {\it constant term projection} (cf. Section \ref{EisSer}) of  $\calE' \left(\frac{s_0-m}{2}, k -(n+1)t, \eta \right) |[\mfrak^2 \varpi^{2\beta}/4]$, with $n=0,1$ and $n \equiv s \bmod 2$.\\
 
Let $\Cfrak_0$ be the conductor (outside $p$) of $\chi^{-1}\psi^{2}$ and let $\cfrak_0$ be an id\`ele representing it. Let $\Theta_{\chi\psi^{2}}$ be the theta measure of Section \ref{p-measure} {\it which we consider now of level $\Cfrak_0$}. We define $\Theta_{\chi^{-1}\psi^2}' (\eps) = \Theta_{\chi^{-1}\psi^2} (\eps^{-1})$ and then we define the product measure $\mu:= \Theta'_{\chi^{-1}\psi^2}|\left[\frac{\mfrak^2}{\cfrak_0^2}\right] \times {E}_c^{\chi,-}$  as $p-1$ measures $\mu_i$ on $1+ p \Z_p$; for $s$ in $\Z_{\geq 0}$, $s\equiv i \bmod p-1$, we pose
\begin{align*}
\int_{1 + p\Z_p }\eps_Q(z){\lla \N_p(z) \rra}^{s}  \textup{d}\mu_i(x)= 
\end{align*}
\begin{align*}= \left(\int_{\clpinf}\omega^{-s}\eps_Q(z) \textup{d}\Theta'_{\chi^{-1}\psi^2}(z)\right)|\left[\frac{\mfrak^2}{4\cfrak_0^2}\right] \times \int_{\clpinf} \eps_Q {\N_p(z)}^{s} \textup{d}{E}_c^{\chi,-}. &
\end{align*}
%weight (s+2m +1/2)t -2v +(-s-m +1/2)t = m+1 - 2v Dunque -s-m >0 e 2m > -s, ovvero 2m>-s>m, che e 
% test 1, with m - k_0 < s <m, weight s + m - s -2v, E((m-s)/2,\psi_P^{-1} )
Evaluating at $Q$ as above we obtain 
\begin{align*}
 &  {A_0}^{-1}(1- \eta^{-1}(c){\N_p(c)}^{s_0 -m}) \frac{\Omega(\eta^{-1}, s_0-m-1)}{{\N(\varpi^)}^{-(s_0-m-1)\beta_0}} \times \\
 & \times\theta(\chi^{-1}\omega^{s} \eps_Q^{-1}\psi^{2})|\left[\frac{\mfrak^2}{4\cfrak_0^2}\right] c\left(\calE' \left(\frac{s_0-m}{2}, k-(n+1)t, \eta  |[\mfrak^2 \varpi^{2\beta}2^{-1}] \right)\right)|\nu_Q.
\end{align*}
where  $A_0$ is given in Proposition \ref{Eisen_other2} and is equal to 
% m-s +1 -n+k
\begin{align*}
A_0 = &  i^{k -(n+1)d}{\pi}^{\alpha'} {\G_{\infty}\left(\alpha' \right)}^{-1} {2}^{k - \left(n+\frac{3}{2}\right)d } \times \\ 
      & \times  \eta(\mfrak \varpi^{\beta}  \dfrak 2^{-1}) {\N(\dfrak)}^{m-s_0}{\N(\mfrak \varpi^{\beta} 2^{-1})}^{m-s_0-1},\\
\alpha' = & \frac{(m-s +1 -n)t+k}{2}. 
\end{align*}
Let us  denote by  $\chi'$ the primitive character associated with $\omega^{s} \eps^{-1}\chi^{-1}\psi^{2}$. 
 We use the relations given in Section \ref{ComLFun} to obtain
\begin{align*}
\theta(\chi^{-1}\omega^{s} \eps^{-1}\psi^{2})\left[\frac{\mfrak^2}{4\cfrak_0^2}\right]|\tau(\mfrak^2 \varpi^{2\beta}) & =C(\chi'){\N(2^{-1}\mfrak\cfrak_0^{-1})}^{-1/2} {\N( \varpi^{\beta-\alpha_0})}^{1/2}\times \\
& \times \sum_{\efrak |p}{\N(\efrak)}\mu(\efrak)\chi'(\efrak)\theta({\chi'}^{-1})|\left[\frac{\varpi^{2\beta}}{\varpi^{2\alpha_0}\efrak^2} \right].
\end{align*}
For each $\qfrak$ dividing $2$ in $\rif$, we pose 
\begin{align*}
\calE^-_{\qfrak}(X)=& (1-\psi^{-2}\chi^{2}(\qfrak)A_{\N(\qfrak)}^{2}(X)2^{-2m-2} )\\
\calE^-_{2}(X)=& \prod_{\qfrak |2} \calE^-_{\qfrak}(X),
\end{align*}
where $X$ is a variable on the free part of the cyclotomic extension in $\clpinf$ and corresponds to the variable $Q$. We put moreover $\Delta^-(Y)=(1-\chi_{-1}\chi^{-1}{\psi}^{-1}{\psi'}^{-2}(c){\N_p(c)}^{m+2}{(1+Y)}^{-1})$ where $c$ is chosen such that $\lla \N(c) \rra$ correspond to the   fixed generator $u$.  \\
We define a $p$-adic $L$-function 
\begin{align*}
 \Ll^-_p (X) = & {( \Delta^-(X) \calE^-_2(X)H_{\boldf})}^{-1}l_{\lambda_{\boldf}} e  \mathrm{Tr}_1 ( \mathrm{Tr}_{2} \Theta'_{\chi} \times {E}_c^{\chi^{-1},-}| \Xi_2).
\end{align*}
where 
\begin{align*}
 \mathrm{Tr}_1 = T_{\Mfrak/\Nfrak} \,\, ,  \mathrm{Tr}_2= Tr_{\Sl}^{\gl}(v,\psi'), 
\end{align*} 
and $H_{\boldf}$ is the congruence number of $\boldf$ (the specialization of $H$, defined at the beginning of Section \ref{padicL}, at the point corresponding to $\boldf$). \\
Let us point out that we can move the twist by $\nu_Q$ from the Eisenstein series to the theta series as we did to prove the formulas at the end of Section \ref{halfp}, and that on $\theta(\chi^{-1}\omega^{s} \eps_Q^{-1}\psi^{2})|\left[\frac{\mfrak^2}{4\cfrak_0^2}\right]$ this twist is simply multiplication by $\psi^{(p)}(\chi^{(p)})^{-1}\chi_{-1}(p^{\alpha_0})E_3(Q)$.\\
We have then for $Q=X+1 -\eps(u)u^s$
\begin{align*}
 \Ll^-_p (\eps(u)u^s - 1)= &{\calE_2^-(\eps(u)u^s - 1)}^{-1}  \frac{{\N(\mfrak \varpi^{\beta} 2^{-1})}^{-\frac{1}{2}} E_3(Q) \Omega(\eta^{-1}, s_0-m-1)} {A_0{\lambda(T(\varpi^{2\beta-\alpha}))}{\N(\varpi)}^{(m-s_0+1)\beta_0}} \times \\
 &\times \frac{ \lla \boldf^{c}, \mathrm{Tr}_1 \circ  \mathrm{Tr}_{2} \left(\theta_n(\chi')|\left[\frac{\mfrak^2}{4\cfrak_0^2}\right] \calE \left(\frac{s_0-m}{2}, k-(n+1)t, \eta \right)\tau(\mfrak^2 \varpi^{2\beta}\right)  \rra_{\Mfrak p^{2\beta}} } {\lla \boldf^{c}|\tau'(\nfrak\varpi^{\alpha}), \boldf \rra_{\Nfrak p^{\alpha}}}, \\
 =  &\frac{C^{(p)}}{C_p} E^{-}_{1}(Q)E^{-}_{2}(Q)E_3(Q)\frac{\Ll(s_0+1, \boldf,\eps_Q \omega^{-s}\psi^{-2} \chi)}{\lla \boldf^{\circ}, \boldf^{\circ} \rra_{\Nfrak p^{\alpha}}}. 
\end{align*}
Here we have
%1-v +(s-m-k-n)/2 -1/2d  = 1-v + d(s-n-1)/2 + d -k-v = 1+d(s-n-1)/2 -m +3d
\begin{align*}
C^{(p)}= &  2^{1-v} i^{(n+1)d-k} 2^{\left( n -\frac{3}{2} \right)d -k} {\pi}^{\frac{(s+n -m-1)t-k}{2}} {\G_{\infty}\left( \frac{(m-s-n+1)t+k}{2}\right)} \times \\
        & \times {(2\pi )}^{\frac{(s-m-n)d-k}{2}}\G_{\infty}\left( \frac{(m-s+n)t+k}{2}\right)  \times \\ 
      & \times  \eta^{-1}(\mfrak \varpi^{\beta}\dfrak 2^{-1}){\N(\dfrak)}^{m-s -\frac{1}{2}}{\N(\mfrak \varpi^{\beta} 2^{-1})}^{ m+ \frac{3}{2} -s}\times \\
      & \times \frac{\Omega(\eta^{-1}, m-s)}{{\N(\varpi)}^{(s-m)\beta_0}}G(\chi'){\N(\mfrak 2^{-1})}^{-1/2} {\N(\varpi^{\frac{1}{2}\beta-\alpha_0})}  \\
     & \times  {\N(2^{-1}\dfrak\mfrak\varpi^{\beta})}^m \psi(\mfrak\varpi^{\beta}\dfrak)  \\
    = & i^{(n+1)d-k} {2}^{ 1 + d(s+n -4m + \frac{s-n-11}{2} ) -2v}{\pi}^{(s-m)d-k} 2^{(s-m+1)d-k} ((m-s-1)t+k)! \times \\ 
      &  \times \eta^{-1}(\mfrak \varpi^{\beta}\dfrak 2^{-1})\psi(\mfrak\varpi^{\alpha}\dfrak) {\N(\dfrak)}^{2m -s -\frac{1}{2}}{\N(\varpi)}^{(m-s)\beta_0 +(2m-s + 2)\beta - \alpha_0 } \\
     &\times  G(\omega^s\eps^{-1}\chi^{-1}\psi^2)(\omega^s\eps^{-1}\chi^{-1}\psi^2)(\dfrak\cfrak_0\varpi^{\alpha_0})\\
  & \times \Omega(\eta^{-1}, m-s)  {\N(\mfrak)}^{2 m +1-s},\\
C_p = & \N(\varpi)^{-(\alpha -\alpha') m /2} {\lambda(T(\varpi^{2\beta -\alpha'}))}\psi_{\infty}(-1)  \times \\ 
      & \times W'(\boldf_P) S(P) \times \prod_{\pfrak} \frac{\eta\nu(\dfrak_{\pfrak})}{|\eta\nu(\dfrak_{\pfrak})|} \prod_{J}G(\nu\psi'),
\end{align*}
where we applied the duplication formula  
\begin{align*}
 & {\G_{\infty}\left( \frac{(m-s-n+1)t+k}{2}\right)}\G_{\infty}\left( \frac{(m-s+n)t+k}{2}\right) =   \\
  = & \G_{\infty}((m-s)t+k) 2^{d(s-m+1)-k} \pi^{1/2}. &
 \end{align*}
\end{proof}

% Theta' series of neben psi^2 eps^{-1}\chi^{-1}, f^c |\tau of neben \psi_P^{-1}\psi'^{2} \calE' must be of neben \psi_P \psi'^{-2} \psi_P^{-2} eps \chi = \psi'^{-2}\psi_P^{-1}\eps\chi
%weights are (2m -s +1/2)t + ( 1+s-m)-2v = m -2v
%Forse ci va theta  psi^{-2} eps^{-1}\chi^{-1}
We can now construct the primitive symmetric square $p$-adic $L$-function as a one variable Iwasawa function.\\ 
Recall the factor $\calE_{\Nfrak}(s,\boldf,\chi)$ defined in Section  \ref{Extrafactors} and the partition of the set of primes dividing $\Nfrak$ in the four subsets {\it i)}, {\it ii)}, {\it iii)}, {\it iv)}. \\
For each prime $\qfrak$ dividing $\Nfrak$, let $\boldf_{\qfrak}$ be a twist of $\boldf$ which is minimal among all the twists at $\qfrak$ (in the sense of Section \ref{Extrafactors}), and let $\lambda_{\qfrak}(T(\qfrak))$ be the Hecke eigenvalues of $\boldf_{\qfrak}$ at $T(\qfrak)$. In this case, let us denote by $\alpha_{\qfrak}$ and $\beta_{\qfrak}$ the two roots of the Hecke polynomial at $\qfrak$. The existence of a global character $\eta$ such that $\boldf \otimes \eta$ is of the desired type at $\qfrak$  is guaranteed by \cite{Chev}. \\
We define the following factors; if $\qfrak$ belongs to cases {\it i)} or {\it ii)} we have two possibilities, according to ${\lambda(T(\qfrak))}=0$ or not. Only for these two cases, we mean $\chi(\qfrak)=0$ if $\chi$ is ramified. In the first case, we put
\begin{align*}
\calE^+_{\qfrak}(X) = & \left(1-\chi^{-1}(\qfrak){\alpha^2_{\qfrak}} {\N(\qfrak)}^{-1}A_{\N(\qfrak)}^{-1}(X)\right) \left(1-\psi\chi^{-1}(\qfrak){\N(\qfrak)}^{m}A_{\N(\qfrak)}^{-1}(X)\right)\times \\
 &\left(1-\psi^{2}\chi^{-1}(\qfrak){\alpha_{\qfrak}(T(\qfrak))}^{-2}{\N(\qfrak)}^{2m+1} A_{\N(\qfrak)}^{-1}(X)\right),\\
\calE^-_{\qfrak}(X) = & \left(1-\psi^{-2}\chi(\qfrak){\alpha^2_{\qfrak}}{\N(\qfrak)}^{-2m-2}A_{\N(\qfrak)}(X)\right) \left(1-\psi^{-1}\chi(\qfrak) {\N(\qfrak)}^{-1-m}A_{\N(\qfrak)}(X)\right) \times \\
 &\left(1-\chi(\qfrak){\alpha^{-2}_{\qfrak}}A_{\N(\qfrak)}(X)\right),
  \end{align*}
  while if ${\lambda(T(\qfrak))} \neq 0$
  \begin{align*}
\calE^+_{\qfrak}(X) = & \left(1-\psi\chi^{-1}(\qfrak){\N(\qfrak)}^{m}A_{\N(\qfrak)}^{-1}(X)\right) 
 \left(1-\psi^{2}\chi^{-1}(\qfrak){\lambda_{\qfrak}(T(\qfrak))}^{-2}{\N(\qfrak)}^{2m+1} A_{\N(\qfrak)}^{-1}(X)\right),\\
\calE^-_{\qfrak}(X) = & \left(1-\psi^{-1}\chi(\qfrak) {\N(\qfrak)}^{-1-m}A_{\N(\qfrak)}(X)\right) \left(1-\chi(\qfrak){\lambda_{\qfrak}(T(\qfrak))}^{-2}A_{\N(\qfrak)}(X)\right).
  \end{align*}
  If $\qfrak$ is in case {\it iv)}, recall that there are several subcases; if ${(\psi\chi)}^2$ is ramified, we do not need to define any extra factors. \\
  If ${\psi\chi}$ is unramified, we define
  \begin{align*}
  \calE_{\qfrak}^+(X) & =1+\psi\chi^{-1}(\qfrak) {\N(\qfrak)}^{m}A_{\N(\qfrak)}^{-1}(X),\\
   \calE_{\qfrak}^-(X) & =1+\chi\psi^{-2}(\qfrak) {\N(\qfrak)}^{-m-1}A_{\N(\qfrak)}(X).
  \end{align*}
  For the remaining case, ${(\psi\chi)}^2$  unramified and ${\psi\chi}$ ramified, recall that in Section \ref{Extrafactors} we have defined two characters $\lambda_i$, $i=1,2$, such that ${\psi\chi\lambda_i}$ is unramified. The extra factor we need to define are
  \begin{align*}
  \calE_{\qfrak,i}^+(X) & = 1-\psi\chi^{-1}\lambda_i(\qfrak) {\N(\qfrak)}^{m}A_{\N(\qfrak)}(X)^{-1},\\
   \calE_{\qfrak,i}^-(X) & = 1-\chi\psi^{-1}\lambda_i(\qfrak) {\N(\qfrak)}^{-m-1}A_{\N(\qfrak)}(X) ,
  \end{align*}
  and then 
  \begin{align*}
  \calE_{\qfrak}^+(X) & = \prod_{i \mbox{ s.t.} \pi_{\qfrak}\cong\pi_{\qfrak}\otimes\lambda_i}\calE_{\qfrak,i}^+(X),\\
   \calE_{\qfrak}^-(X) & = \prod_{i \mbox{ s.t.} \pi_{\qfrak}\cong\pi_{\qfrak}\otimes\lambda_i}\calE_{\qfrak,i}^-(X).
  \end{align*}
  Let $G(X)$ the formal Laurent series of Theorem \ref{1var}. Then the $p$-adic $L$-function interpolating the primitive values is
  \begin{align*}
  F(X):= G(X)\prod_{\qfrak | \Nfrak} {\calE_{\qfrak}^+(X)}^{-1}. 
\end{align*}
%Let us recall now some properties on Gauss and Jacobi sums for character of arbitrary local field. Let for this lines $\eta_1$ and $\eta_2$ two continuos character of the local field $F_{\pfrak}$, $\pfrak \mid p$;  we shall denote by $\pfrak^{\alpha_1}$ (resp. $\pfrak^{\alpha_2}$) the conductor of $\eta_1$ (resp. $\eta_2$) and we will suppose that $\alpha_1 < \alpha_2$. We define, following \cite{Kub} the Jacobi sums associated to $\eta_1$ and $\eta_2$ as
%\begin{align*}
% J(\eta_1,\eta_2)= \sum_{s \bmod \pfrak^{\alpha_1}, s \not\equiv 0} \eta_1(\varpi^{\alpha_2 - \alpha_1}s) \eta_2(1- \varpi^{\alpha_2 - \alpha_1}s).
%\end{align*}
%For the next few lines, let us denote by $\varpi$ a fixed uniformizer of $F_{\pfrak}$. We have then the relation
%\begin{align*} 
%\eta_1^{-1}(\varpi^{\alpha_{1}+e_{\qfrak}})G(\eta_1)\eta_2^{-1}(\varpi^{\alpha_{2}+e_{\qfrak}})G(\eta_2)= \eta_1^{-1}\eta_2^{-1}(\varpi^{\alpha_{2}+e_{\qfrak}})G(\eta_1\eta_2)J(\eta_1,\eta_2).
%\end{align*}
We give the following proposition; 
\begin{prop}\label{Functpadic}
Let $\boldf$ be a neearly-ordinary Hilbert eigenform. For all fixed $s$ in the interpolation range and $\eps$ a finite order character of $1 +p\Z_p$, we have the following identity 
\begin{align*}
F(\eps(u)u^s-1) & = C_s \Ll_p^-(\eps(u)u^s-1) \prod_{\qfrak | \Nfrak} {\calE_{\qfrak}^-(\eps(u)u^s-1)}^{-1},
\end{align*}
with $\Ll_p^-(X)$ of Theorem \ref{1var2} and $C_s$ in $K$.
\end{prop}
\begin{proof}
We fix $s$ in the interpolation range and we compare the two evaluations formula given in Theorem \ref{1var} and \ref{1var2} at points of type $(\eps,s)$ such that $\eps$ is not trivial and its conductor is bigger than the $p$-part of the conductor of $\boldf$, $\psi$ and $\chi$. \\
We will then use the complex functional equation given in \ref{Extrafactors}.\\
Recall that the factors at infinity that we need are 
\begin{align*}
\G_{\R,\infty}(s+1 -m - \kappa_S) = & \pi^{-\frac{s+1-m + \kappa_{\mathrm{S}}}{2}} \G_{\infty}\left( \frac{s+1-m-\kappa_{\mathrm{S}}}{2}\right),\\
\G_{\C,\infty}(s-m-1 + k) = & 2^{2v-sd}(st-2v)!\pi^{2v-d(s+1)} ,\\
\G_{\R,\infty}(m +3 - \kappa_S) = & \pi^{d\frac{s-m+\kappa_{\mathrm{S}}-2}{2}} \G_{\infty}\left( \frac{m+2-s-\kappa_{\mathrm{S}}}{2}\right), \\
\G_{\C,\infty}(m-s+ k) = & 2^{(s-m+1)d-k} \pi^{(s-m)d-k} (k+(m-s-1)t)!,
\end{align*}
where $\kappa_{\mathrm{S}} =0,1$ is congruent to $s+m$ modulo $2$ (notice that we are evaluating the functional equation at $s+1$).\\ 

In the notation of Theorem \ref{1var} and \ref{1var2} we have then $\beta=\alpha=\beta_0=\alpha_0$ and
\begin{align*}
\frac{F(X)\prod_{\qfrak | \Nfrak} {\calE_{\qfrak}^-(X)}}{\Ll_p^-(X)} =&  i^{(s-n)d} 2^{3md - 2sd + \frac{9}{2}d }D_{F,p}^{3s-3m+\frac{1}{2}} {D'_F}^{2s-2m -1}\times \\
 & \times \frac{{\N(\lfrak)}^{n-s} {\N(\mfrak)}^{2s-2m-1} {\N(\varpi^{\alpha})}^{3s-3m-1}G(\chi\eps\omega^{-s}) }{G(\chi^{-1}\eps^{-1}\omega^{s}\psi^2) \chi\psi^{-1}\omega^{-s}\eps(\varpi^{\alpha +e})\prod_{\pfrak}G(\chi^{-1}_{\pfrak}\psi_{\pfrak}\eps^{-1}\omega^{s})}\\
 & \times \frac{(\chi\eps\omega^{-s})_0(\dfrak\cfrak\varpi^{\alpha})\chi\eps\omega^{-s}\psi^{-1}(\dfrak\mfrak\varpi^{\alpha}2^{-1})}{(\chi^{-1}\eps^{-1}\omega^{s}\psi^2)_0(\dfrak\cfrak_0\varpi^{\alpha})\chi^{-1}\eps^{-1}\omega^{s}\psi(\dfrak\mfrak\varpi^{\alpha}2^{-1})} \\
 & \times \frac{\G_{\R,\infty}(s+1 -m - \kappa_S)\G_{\C,\infty}(s-m-1 + k)}{\G_{\R,\infty}(m +3 - \kappa_S)\G_{\C,\infty}(m-s+ k)} \\
 & \times \frac{L(s+1,\mathrm{Sym}^2(\boldf),\chi^{-1}\eps^{-1}\omega^{s})}{L(2m -s+2,\mathrm{Sym}^2(\boldf^c),\chi\eps\omega^{-s})}.
\end{align*}
The occurence of the quotient of real gamma factors comes from the expression of the period $\Omega(\eta^{-1}, m-s)$ in the evaluation formula of Proposition \ref{RibetFE} (noticing that if $\kappa=0,1$ is such that $\eta(-1)={(-1)}^{\kappa}$, we have $\kappa = 1 -\kappa_{\mathrm{S}}$). As we have
\begin{align*}
  \frac{\LL(s+1,\mathrm{Sym}^2(\boldf),\chi^{-1}\eps^{-1}\omega^{s})}{\LL(2m -s+2,\mathrm{Sym}^2(\boldf^c),\chi\eps\omega^{-s})} = & \beps(s-m,\hat{\pi}(\boldf),\chi^{-1}\eps^{-1}\omega^{s}\psi)
\end{align*}
we can  use Lemma \ref{epsfactor} and after noticing that (in the notation of Lemma \ref{epsfactor}) $\nu=\psi_{\qfrak}$ we obtain that the only part on the right hand side which depends on $\eps$ or $p$ is 
%\begin{align*}
%& G((\chi^{-1}\eps^{-1}\omega^{s})_{\pfrak})(\chi\eps\omega^{-s})_{\pfrak}(\varpi_{\pfrak}^{\alpha+e_{\pfrak}}) G(\psi_{\pfrak})\psi_{\pfrak}^{-1}(\varpi_{\pfrak}^{\gamma+e_{\pfrak}}) = \\
%& = G((\chi^{-1}\eps^{-1}\psi\omega^s)_{\pfrak}) (\chi\eps\psi^{-1}\omega^{-s})_{\pfrak}(\varpi_{\pfrak}^{\alpha+e_{\pfrak}}) J(\psi,\chi^{-1}\eps^{-1}\omega^s),\\
%& G((\chi^{-1}\eps^{-1}\psi\omega^s)_{\pfrak}) (\chi\eps\psi^{-1}\omega^{-s})_{\pfrak}(\varpi^{\alpha+e_{\pfrak}}) %G(\psi_{\pfrak})\psi_{\pfrak}^{-1}(\varpi_{\pfrak}^{\gamma+e_{\pfrak}}) = \\
%& = G((\chi^{-1}\eps^{-1}\psi^2\omega^s)_{\pfrak}) (\chi\eps\psi^{-2}\omega^{-s})_{\pfrak}(\varpi_{\pfrak}^{\alpha+e_{\pfrak}}) %J(\psi,\psi\chi^{-1}\eps^{-1}\omega^s)
%\end{align*}
\begin{align*}
 \frac{(\chi\psi^{-1} \eps\omega^{-s}\psi)_p^4(\varpi^{\alpha+e}) {\N(\varpi^{\alpha_1+\alpha_2 + \alpha_3})}^{s-m}}{\chi_p(\varpi^{\alpha_1 + e}) \psi_p^2\chi_p^{-1}(\varpi^{\alpha_2 + e}) \psi_p\chi_p^{-1}(\varpi^{\alpha_3 + e})G(\chi_p^{-1})G(\psi_p^2\chi_p^{-1})G(\psi_p\chi_p^{-1})} 
\end{align*}
where $\alpha_1$ (resp. $\alpha_2$, $\alpha_3$) is the conductor of $\chi_p$ (resp. $\chi_p \psi_p^{-2}$, $\chi_p \psi_p^{-1}$).
If we choose $\varpi_p$ such that $\chi_{\pfrak}(\varpi_{\pfrak})=\psi_{\pfrak}(\varpi_{\pfrak})=1$ and $\N(\varpi_p^e)=p^d$, we obtain that this quantity is independent of $\eps$ and we are done.
\end{proof}
The main theorem of this section is the following;
\begin{theo}\label{1varholo}
Let $\boldf=\boldf_P$ be a nearly ordinary form of Nebentypus $\psi$ and weight $k\geq 2t$ which we decompose as $k = m +2t -2v$. Let $\chi$ be a Hecke character  such that $\chi_{\sigma}(-1)=1$ for all $\sigma|\infty$. Suppose that $\boldf$  has not CM by $\chi\psi^{-1}$ (resp. has  CM), then we have a a formal series $F(X,\boldf,\chi)$ in $\oo[[X]]\left[ \frac{1}{p}\right]$ (resp. $\oo[[X]]\left[ \frac{1}{p}, \frac{1}{1+X -u^{m+1}}\right]$) such that for all finite order character $\eps$ of $1+p\Z_p \cong u^{\Z_p}$, of conductor $p^{\alpha_0}$, and  $s \in [m-k_0+ 2, m] $ with $n\equiv s$ we have 
\begin{align*}
 F(\eps(u)u^s -1,\boldf,\chi) = & i^{(s+1)d+k} 2^{2v - (s+m)d +nd - \frac{s+n}{2}d } (st-2v)! D_F^{s-\frac{3}{2}}\times \\
            & G(\chi\eps\omega^{-s}_0)  \eta^{-1}(\mfrak \varpi^{\alpha_0} \dfrak 2^{-1})  \chi\eps\omega^{-s}_0(\dfrak\cfrak\varpi^{\alpha_0}) \times\\
            & \frac{ \psi(\dfrak\mfrak\varpi^{\beta}) {\N(\mathfrak{l}) }^{-2[s/2]}   {\N(\varpi)}^{(s+1)\beta -\alpha_0 }} {\lambda(T(\varpi^{2\beta-\alpha}))  W'(\boldf) S(P) \prod_{\pfrak} \frac{\eta\nu(\dfrak_{\pfrak})}{|\eta\nu(\dfrak_{\pfrak})|} \prod_{J}G(\nu)} \times \\
            & E_1(s+1)E_2(s+1) \frac{ L(s+1,\mathrm{Sym}^2(\boldf),\chi^{-1}\eps^{-1}\omega^{s})}{{(2 \pi )^{ds}}\Omega(\boldf)}.
\end{align*}
for $\eta=\omega^{s}\eps^{-1}\chi_1^{-1}\chi^{-1}\psi{\psi}^{-2}$ and, we recall, $\alpha$ such that $\varpi^{\alpha}$ is the conductor of $\boldf$ and $\beta$ such that $\varpi^{\beta}$ is divisible by the {\it l.c.m.} of $\varpi^{\alpha}$  and $\varpi^{\alpha_0}$.
\end{theo}

\begin{proof}
Recall the functional equation of Proposition \ref{Functpadic}
\begin{align*}
F((X+1)u^s -1) & = C_s \Ll_p^-((X+1)u^s -1) \prod_{\qfrak | \Nfrak} {\calE_{\qfrak}^-((X+1)u^s -1)}^{-1}.
\end{align*}
We shall show that the possible poles of the left hand side and right hand side are disjoint (are reduce to $X= u^{m+1} -1$ if $\boldf$ has CM by $\psi^{-1}\chi$).\\
Take for example $u=p+1$. The possible poles of $F(X)$ are the zeros of $\prod_{\qfrak | \Nfrak} {\calE_{\qfrak}^+(X)}$, $\calE_2(X)$ and $\Delta(X)$. We recall that 
\begin{align*}
\calE_{2}(X)=& \prod_{\qfrak |2} (1-\psi^2\chi^{-2}(\qfrak)A_{\N(\qfrak)}^{-2}(X) {\N(\qfrak)}^{2m}),\\
\Delta(X) = & (1-\chi(c)\N(c)^{m+1}{(1+X)}^{-1}).
\end{align*}
Therefore the possible poles must be either
\begin{itemize}
\item $s=m$ (occurring in one of the factors $\calE^+_{\qfrak}$ where neither $\lambda(T(\qfrak))$ nor $\alpha_{\qfrak}$ do not appear),
\item or $s=m+1$ (occurring in the factor $\Delta(X)$),
\item or $s$ such that there exists a $p^{\infty}$-root of unit $\zeta$ such that one the following occurs:	
\begin{align*} 
 \psi^{2}\chi^{-1}(\qfrak){K_{\qfrak}}^{-2}{\N(\qfrak)}^{2m+1}A_{\N(\qfrak)}^{-1}(\zeta u^s-1)= &1, \\
\chi^{-1}(\qfrak){K_{\qfrak}}^2 {\N(\qfrak)}^{-1}A_{\N(\qfrak)}^{-1}(\zeta u^s-1)=& 1,
\end{align*}
for $K_{\qfrak}=\lambda_{\qfrak}(T(\qfrak))$ or $\alpha_{\qfrak}$.
\end{itemize}
On the other hand, the possible poles of the right hand side are instead 
\begin{itemize}
\item $s=m+1$  (occurring in one of the factors $\calE^-_{\qfrak}$ where $\lambda(T(\qfrak))$ does not appear),
\item or when there is a $p^{\infty}$-root of unit such that as 
\begin{align*}
\chi(\qfrak){K_{\qfrak}}^{-2}A_{\N(\qfrak)}(\zeta u^s -1)= & 1, \\
\psi^{-2}\chi(\qfrak){K_{\qfrak}}^2 {\N(\qfrak)}^{-2-2m}A_{\N(\qfrak)}(X)(\zeta u^s -1)= & 1,
\end{align*}
for $K_{\qfrak}=\lambda_{\qfrak}(T(\qfrak))$ or $\alpha_{\qfrak}$.
\end{itemize}
From Weierstrass preparation theorem, we know that there is only a finite number of possibilities of the above cases.\\
We proceed now exactly as in \cite[\S 3.1]{DD}. Let $\qfrak$ be a prime ideal such that 
\begin{align*}
 \chi(\qfrak){\lambda(T(\qfrak))}^{-2}A_{\N(\qfrak)}(\zeta u^m -1)=1
\end{align*}
 and $\qfrak'$ a prime ideal such that $\calE_{\qfrak'}$ has a zero at $s=m$, $\eps(u)=\zeta$. Write $\lla \N(\qfrak) \rra = u^{z_{\qfrak}}$, so 
\begin{align*}
 A_{\N(\qfrak)} (X) = {(X+1)}^{z_{\qfrak}}.
\end{align*}
Recall that ${|\lambda_{\qfrak}(T(\qfrak))|}^2_{\C}=\N(\qfrak)^{m+1}$. In particular, we obtain the following relations: $|u^m|_{\C}=\N(\qfrak')^m$ and  $|u^m|_{\C}=\N(\qfrak)^{m+1}$. But this is a contradiction as $|u|_{\C}\neq 1$. The other cases are similar.\\
The case $s=m+1$ has already been treated in Proposition \ref{reszeta}, and we can then conclude.
\end{proof}

%%%%%%%%%%%%%%%%%%%%%%%%%%%%%%%%%%%%%%%%%%%%%%%%%%%%%%%%%%%%%%%%%%%%%%%%%%%%%%%%%%%%%%%%%%%%%%%%%%%%%%%%%%%%%

\section{Holomorphicity in many variable}\label{AppB}
In this section we will show that the many variable $p$-adic $L$-function constructed in Theorem \ref{T1} is holomorphic (when the family has not complex multiplication). We adapt to the totally real case the strategy of proof of Hida \cite[\S 6]{H6}.\\
Fix a family of $\boldI$-adic eigenforms $\boldF$ and let $\lambda$ be the structural morphism as in Section \ref{padicL}; let us denote by $\psi$ the restriction of $\lambda$ to the torsion of $\clpinf$. Before defining the functions interpolating the missing Euler factors, we have to show that we can interpolate $p$-adically the correct Satake parameter at the primes $\qfrak$ at which $\boldF$ is not minimal.\\
Let $\boldF_{\qfrak}$ be a twist of $\boldF$ by a character $\eta_{\qfrak}$ such that $\boldF_{\qfrak}$ is minimal at $\qfrak$  and primitive outside $p$.  This family corresponds to an $\boldI_{\qfrak}$-family of eigenforms, where $\boldI_{\qfrak}$ is finite flat over the Iwasawa algebra in $d+1+\delta$ variables (here $\delta$ is the default of Leopoldt's conjecture for $p$ and $F$).\\ 
Since we have $\lambda_{\qfrak}(T(\qfrak'))=\eta_{\qfrak}(\qfrak')\lambda(T(\qfrak'))$ for almost all prime ideals $\qfrak'$, we see that, after possibly enlarging the coefficients ring $\oo$, we have $\boldI_{\qfrak}=\boldI$. For all arithmetic points $P$ of $\mathrm{Spec}(\boldI)$, we have then an element $\lambda_{\qfrak}(T(\qfrak)) \in \boldI$  such that the values $\lambda_{\qfrak}(T(\qfrak)) \bmod P $ are the Hecke eigenvalues, still denoted by $\lambda_{\qfrak}(T(\qfrak))$, of a form  $\boldf_P$ which is minimal at $\qfrak$. For the prime ideal $\qfrak$ for which $\boldF$ is not minimal, let us denote by $\alpha_{\qfrak}$ and $\beta_{\qfrak}$ the two roots of the Hecke polynomial at $\qfrak$. After possibly enlarging $\boldI$ by a quadratic extension, we may assume that $\alpha_{\qfrak}$ belongs to $\boldI$. \\
We can now define the functions interpolating the missing Euler factors; as in the proof of Theorem \ref{T1} we use $X$ (resp. $Y$) to denote a variable on the free part of $\Z_p^{\times}$ inside $\clpinf$ (resp. $\clpinf$ embedded in $\boldG=\clpinf \times \rif_p^{\times}$).\\ 
We recall that in Section  \ref{Extrafactors} we partitioned the set of primes dividing $\Nfrak$ in the four subsets {\it i)}, {\it ii)}, {\it iii)}, {\it iv)}.
If $\qfrak$ belongs to cases {\it i)} or {\it ii)} we have two possibilities, according to ${\lambda(T(\qfrak))}=0$ or not. Only for these two cases, we mean $\chi(\qfrak)=0$ if $\chi$ is ramified. In the first case, we put
\begin{align*}
\calE^+_{\qfrak}(X,Y) = & \left(1-\chi^{-1}(\qfrak)\alpha_{\qfrak}^2 {\N(\qfrak)}^{-1}A_{\N(\qfrak)}^{-1}(X)\right) \left(1-\psi\chi^{-1}(\qfrak)A_{\N(\qfrak)}(Y)A_{\N(\qfrak)}^{-1}(X)\right)\times \\
 &\left(1-\psi^{2}\chi^{-1}(\qfrak)\alpha_{\qfrak}^{-2}{\N(\qfrak)}A_{\N(\qfrak)}^{2}(Y) A_{\N(\qfrak)}^{-1}(X)\right),\\
\calE^-_{\qfrak}(X,Y) = & \left(1-\psi^{-2}\chi(\qfrak)\alpha_{\qfrak}^2{\N(\qfrak)}^{-2}A_{\N(\qfrak)}^{-2}(Y)A_{\N(\qfrak)}(X)\right)  \left(1-\chi(\qfrak)\alpha_{\qfrak}^{-2}A_{\N(\qfrak)}(X)\right)\times \\
 &\left(1-\psi^{-1}\chi(\qfrak) {\N(\qfrak)}^{-1}A_{\N(\qfrak)}^{-1}(Y)A_{\N(\qfrak)}(X)\right),
  \end{align*}
  while if ${\lambda(T(\qfrak))} \neq 0$
  \begin{align*}
\calE^+_{\qfrak}(X,Y) = & \left(1-\psi\chi^{-1}(\qfrak)A_{\N(\qfrak)}(Y)A_{\N(\qfrak)}^{-1}(X)\right) \left(1-\psi^{2}\chi^{-1}(\qfrak){\lambda_{\qfrak}(T(\qfrak))}^{-2}{\N(\qfrak)}A_{\N(\qfrak)}^{2}(Y) A_{\N(\qfrak)}^{-1}(X)\right),\\
\calE^-_{\qfrak}(X,Y) = & \left(1-\psi^{-1}\chi(\qfrak) {\N(\qfrak)}^{-1}A_{\N(\qfrak)}^{-1}(Y)A_{\N(\qfrak)}(X)\right)  \left(1-\chi(\qfrak){\lambda_{\qfrak}(T(\qfrak))}^{-2}A_{\N(\qfrak)}(X)\right).
  \end{align*}
  If $\qfrak$ is in case {\it iv)}, recall that there are several subcases; if ${(\psi\chi)}^2$ is ramified, we do not need to define any extra factors. \\
  If ${\psi\chi}$ is unramified, we define
  \begin{align*}
  \calE_{\qfrak}^+(X,Y) & =1+\psi\chi^{-1}(\qfrak) A_{\N(\qfrak)}(Y)A_{\N(\qfrak)}^{-1}(X) ,\\
   \calE_{\qfrak}^-(X,Y) & =1+\chi\psi^{-2}(\qfrak) {\N(\qfrak)}^{-1}A_{\N(\qfrak)}^{-1}(Y)A_{\N(\qfrak)}(X).
  \end{align*}
  For the remaining case, ${(\psi\chi)}^2$  unramified and ${\psi\chi}$ ramified, recall that in Section \ref{Extrafactors} we have defined two characters $\lambda_i$, $i=1,2$, such that ${\psi\chi\lambda_i}$ is unramified. The extra factor we need to define are
    \begin{align*}
  \calE_{\qfrak,i}^+(X,Y) & = 1-\psi\chi^{-1}\lambda_i(\qfrak) A_{\N(\qfrak)}(Y)A_{\N(\qfrak)}(X)^{-1},\\
   \calE_{\qfrak,i}^-(X,Y) & = 1-\chi\psi^{-1}\lambda_i(\qfrak) {\N(\qfrak)}^{-1}A_{\N(\qfrak)}^{-1}(Y)A_{\N(\qfrak)}(X) ,
  \end{align*}
  and then 
  \begin{align*}
  \calE_{\qfrak}^+(X,Y) & = \prod_{i \mbox{ s.t.} \pi_{\qfrak}\cong\pi_{\qfrak}\otimes\lambda_i}\calE_{\qfrak,i}^+(X,Y),\\
   \calE_{\qfrak}^-(X,Y) & = \prod_{i \mbox{ s.t.} \pi_{\qfrak}\cong\pi_{\qfrak}\otimes\lambda_i}\calE_{\qfrak,i}^-(X,Y).
  \end{align*}
We define 
\begin{align*}
 \calE^{\pm}(X,Y) = \prod_{\qfrak \mid \Nfrak} \calE_{\qfrak}^{\pm}(X,Y).
\end{align*}
We refer to Section \ref{padicL} for all the notation used in the following theorem.

\begin{theo}\label{T1holo}
Fix an adelic character $\chi$, such that $\chi_{\sigma}(-1)=1$ for all $\sigma|\infty$ and fix $\boldF$, a $\boldI$-adic family of Hilbert eigenforms. If $\chi\psi^{-1}$ is imaginary quadratic, suppose that $\boldF$ has not (resp. has) complex multiplication by $\chi\psi^{-1}$; then we have a $p$-adic $L$-functions $H(P)L_p(Q,P)$ in $\oo[[X ]] \hat{\otimes}  \boldI$  (resp. $\oo[[X ]] \hat{\otimes}  \boldI [{(1+X)-u(1+Y)}^{-1}]$) such that,  for all arithmetic points $(Q,P)$ of type $(s_Q,\eps_Q; m_P,\eps_P, v_P, \eps_P')$, with $ m_P - k_{P,0}  + 2  \leq s_Q \leq m_P $ (for $k_{P,0}$ equal to the minimum of $k_{P,\sigma}$'s) such that $\eps_P'$ factors through $\clpinf$ the following interpolation formula holds 
\begin{align*} 
L_p(Q,P) = & C_1 E_{1}(Q,P)E_{2}(Q,P)\frac{2^d L(s_Q+1, \mathrm{Sym}^2(\boldf_P),\eps_Q^{-1}\omega^{s_Q}\chi^{-1})}{(2\pi)^{ds}\Omega(\boldf_P)}.
\end{align*}
\end{theo}
It is clear that this is generalization of \cite[Theorem]{H6}.\\ 
Remarks: \begin{itemize} 
          \item[i)] If Leopoldt's conjecture for $F$ and $p$ does not hold, then $L_p(Q,P)$ is holomorphic also in the case where $\boldF$ has CM by $\chi\psi^{-1}$.
\item[ii)] If $ \calE^{+}(X,Y) \calE_2(X,Y) \equiv 0 \bmod (Q,P)$ and $s_Q = m_P$, then we are in the case of a trivial zero  and the formula above for this point is not of great interest.
\item[iii)] If we could show an holomorphic analogue of Theorem \ref{T1} {\it ii)}, we could prove Conjecture \ref{MainCo} for simple zeros without any assumption on  the conductor. 
\end{itemize}
\begin{proof}
We follow Hida's strategy \cite[\S 6]{H6}. We define 
\begin{align*}
 L_p(Q,P) = \frac{\Ll_p(Q,P)}{\calE^+(X,Y)},
\end{align*}
We shall denote ${\calE^{+}(X,Y) }{\calE_2(X,Y)}$ (resp. ${\calE^{+}(X,Y) }$) by $A(Q,P)$ (resp. $B(Q,P)$). 
Let us define 
\begin{align*}
 \beps_0(\hat{\pi},\chi^{-1}\psi)=& \prod_{\qfrak \nmid p} \beps(0,\hat{\pi}_{\qfrak},\psi_{\qfrak}\chi_{\qfrak}^{-1}),\\
\beps(X,Y)= & \frac{A_{\N(C_{\pi,\chi})}(X) A^{3}_{D_F'}(X)}{A_{\N(C_{\pi,\chi})}(Y)A_{D_F'}^{3}(Y)},\\
C(X,Y)=&   \frac{\N(\mfrak)A^2_{\N(\mfrak)}(Y)A_{\N(\lfrak)}(X)D_F' A_{D_F'}^{2}(Y)A^2_{\N(2)}(X)} {A^2_{\N(\mfrak)}(Y) 2^{\frac{9}{2}d}A_{D_F'}^{2}(X)\N(\lfrak)^{n}A^{3}_{\N(2)}(X) } ,
\end{align*}
where   $C_{\pi,\chi}$ is the conductor outside $p$ of $\hat{\pi}\otimes \psi \chi^{-1}$, and, as above, $D_F'$ is the prime-to-$p$ part of the discriminant and $\mfrak=2\lfrak\cfrak$ is an id\`ele representing {\it l.c.m.} $(\Nfrak, 4, \cfrak)$ where $\cfrak$ is the conductor of $\chi$ . \\
It is plain that  $\beps(X,Y)$ and $C(X,Y)$ are units in $\oo[[X]]\hat{\otimes}  \boldI$. We define 
\begin{align*}
 L'_p(Q,P) = i^{nd}C_s L_p(Q,P) \beps_0 \beps(X,Y) C(X,Y).
\end{align*}
Suppose that $P \in X(\boldI)$ is such that $\boldf_P$ is not primitive at $p$, then $m_P$ is in a fixed class modulo $p-1$ and $\eps_P$ is trivial. We see, using the formula in Proposition \ref{Functpadic},  that $ C_s B(Q,P) L'_p(Q,P) \equiv \Ll^{-}_p(Q) \bmod P$, where $\Ll^{-}_p(Q)$ is the $p$-adic $L$-function of Theorem \ref{1var2} for $\boldf=\boldf_P$.\\ 
Let us denote by $R$ the localization of $\oo[[X]] \hat{\otimes}  \boldI$ at $((1+X)-u(1+Y))$. As $B(Q,P) L'_p(Q,P) \bmod P$ belongs to $\oo[[X]][p^{-1}]$ (or to $\oo[[X]][p^{-1}, {(1+X -u^{m+1})}^{-1} ]$ if $\boldF$ has CM by $\chi\psi^{-1}$), arguing as in \cite[pag. 137]{H6}, we see that there exists $H'$ in $\boldI$ such that $H'(P) B(Q,P) L'_p(Q,P)$ belongs to $R$, (i.e. the possible denominators do not involve the variable $X$); from the construction in Theorem \ref{T1} we see that the choice $H'=H$, the congruence ideal of the family $\boldF$, works. \\
We know then that 
\begin{align*}
 R[B^{-1}(Q,P)H^{-1}(P)] \ni  L'_p(Q,P) \stackrel{\times}{=} L_p(Q,P)  \in R[A^{-1}(Q,P)H^{-1}(P)],
\end{align*}
where $\stackrel{\times}{=}$ means {\it equality up to a unit}.\\
It remains to show that $A(Q,P)$ and $B(Q,P)$ are coprime. As in \cite[Lemma 6.2]{H6}, we know that a prime factor of $A(Q,P)$  (resp. $B(Q,P)$)  is of the form $(1 +X) - \zeta (1+Y)$ (resp. $(1 +X) - \zeta u^{-1}(1+Y)$), for $\zeta$ a $p$-power-order root of unit, or it is of the form  $(1 +X) -\alpha$ with $\alpha$ in $\boldI$ such that $\alpha  = \zeta {K_{\qfrak}}^{2} {\N(\qfrak)}^{-1}\bmod P$ or $\alpha  = \zeta {\lambda_{\qfrak}(T(\qfrak))}^{-2} {\N(\qfrak)}^{2m+1} \bmod P $  (resp. $\alpha  = \zeta {\lambda_{\qfrak}(T(\qfrak))}^{-2} \bmod  P$ or $\alpha  = \zeta K_{\qfrak}^2 {\N(\qfrak)}^{-2m-2} \bmod P$), where $\zeta$ is a root of unit and $K_{\qfrak}=\lambda_{\qfrak}(T(\qfrak))$ or $\alpha_{\qfrak}$.\\ 
Let us denote by $\varpi_K$ a uniformizer of $\oo$, then $\calE_{\qfrak}^{\pm}(X,Y) \not\equiv 0 \bmod \varpi_K$. By contradiction, suppose that it is not true, then we would obtain $1+X \equiv \alpha \bmod \varpi_K$. But this is not possible as $X$ does not belong to $\boldI/ \varpi_K$.\\
If $(1 +X) -\alpha$ is a factor of both $A(Q,P)$ and  $B(Q,P)$, we must have $\alpha =  \zeta u^{-1}(1+Y)$  (or $\alpha =  \zeta (1+Y)$). Reasoning with complex absolute values as in the proof of Theorem \ref{1varholo}, we see that this is a contradiction.\\
The behavior at $((1+X)-u(1+Y))$ has been studied in  Proposition \ref{reszeta} and we can conclude the first part of the theorem.\\
The interpolation formula at points $(Q,P)$ for which $ \calE^{+}(X,Y) \calE_2(X,Y) \equiv 0 \bmod (Q,P)$ is a consequence of Theorem  \ref{1varholo}.
\end{proof}

\bibliographystyle{alpha}
\bibliography{Bibliografy}

\end{document}